\documentclass[11pt,twoside]{amsart}
\usepackage{float}
\usepackage[utf8]{inputenc}
\usepackage[english]{babel}
\usepackage[T1]{fontenc}
\usepackage{mathtools}
\usepackage{enumitem}
\usepackage{bbm}
\usepackage{cancel}
\usepackage{tabularx}
\usepackage{graphicx}
\usepackage{tikz}
\usepackage[all]{xy}
\usepackage{tcolorbox}
\usepackage{amsmath,
    amsfonts,
    amssymb,
    mathrsfs,
    stackrel}
\usepackage{comment}
\usepackage{bookmark}
\usepackage{hyperref}
\usepackage{cleveref}
\hypersetup{colorlinks,
    citecolor=blue,
    linktocpage}
\usepackage[normalem]{ulem}

\DeclareFontFamily{U}{wncy}{}
\DeclareFontShape{U}{wncy}{m}{n}{<->wncyr10}{}
\DeclareSymbolFont{mcy}{U}{wncy}{m}{n}
\DeclareMathSymbol{\Sha}{\mathord}{mcy}{"58}

\newcommand\cA{{\mathcal A}}

\newcommand\cC{{\mathcal C}}

\newcommand\cE{{\mathcal E}}
\newcommand\cF{{\mathcal F}}

\newcommand\cH{{\mathcal H}}

\newcommand\cL{{\mathcal L}}
\newcommand\cM{{\mathcal M}}

\newcommand\cO{{\mathcal O}}
\newcommand\cP{{\mathcal P}}

\newcommand\cS{{\mathcal S}}
\newcommand\cT{{\mathcal T}}

\newcommand\cX{{\mathcal X}}

\newcommand\fF{{\mathfrak F}}
\newcommand\fG{{\mathfrak G}}

\newcommand\bA{{\mathbb A}}
\newcommand\bB{{\mathbb B}}
\newcommand\bC{{\mathbb C}}
\newcommand\bD{{\mathbb D}}

\newcommand\bF{{\mathbb F}}

\newcommand\bH{{\mathbb H}}

\newcommand\bN{{\mathbb N}}

\newcommand\bP{{\mathbb P}}
\newcommand\bQ{{\mathbb Q}}
\newcommand\bR{{\mathbb R}}

\newcommand\bT{{\mathbb T}}

\newcommand\bV{{\mathbb V}}

\newcommand\bZ{{\mathbb Z}}

\newcommand\wL{{\widetilde L}}

\newcommand\wX{{\widetilde X}}

\newcommand{\set}[1]{\left\{ #1 \right\}}

\newcommand\oA{{\overline A}}

\newcommand\oD{{\overline D}}

\newcommand\oH{{\overline H}}

\newcommand\oX{{\overline X}}
\newcommand\oY{{\overline Y}}

\newcommand\uR{{\underline R}}

\newcommand\bfM{{\bf M}}
\newcommand\bfN{{\bf N}}

\DeclareMathOperator{\Sym}{Sym}

\DeclareMathOperator{\GL}{GL}

\DeclareMathOperator{\Hilb}{Hilb}

\DeclareMathOperator{\im}{Im}

\DeclareMathOperator{\Proj}{Proj}

\DeclareMathOperator{\rk}{rk}
\DeclareMathOperator{\SL}{SL}
\DeclareMathOperator{\Sp}{Sp}

\DeclareMathOperator{\Spec}{Spec}

\newcommand{\wh}[1]{{\widehat{#1}}}

\newtheorem{theorem}{Theorem}[section]

\newtheorem{corollary}[theorem]{Corollary}
\newtheorem{lemma}[theorem]{Lemma}

\newtheorem{proposition}[theorem]{Proposition}

\theoremstyle{definition}

\newtheorem{construction}[theorem]{Construction}

\newtheorem{definition}[theorem]{Definition}
\newtheorem{example}[theorem]{Example}

\newtheorem{notation}[theorem]{Notation}

\newtheorem{question}[theorem]{Question}
\newtheorem{remark}[theorem]{Remark}

\newcommand{\twopartdef}[4]
{
	\left\{
		\begin{array}{ll}
			#1 & \mbox{if } #2 \\
			#3 & \mbox{if } #4
		\end{array}
	\right.
}

\newcommand{\twobytwo}[4]
{
	\begin{pmatrix}
		#1 & #2 \\
		#3 & #4
	\end{pmatrix}
}

\title[The Mumford construction]{Combinatorics
and Hodge theory of degenerations of abelian varieties:
A survey of the Mumford construction}
\author[Engel]{Philip Engel}
\address{Department of Mathematics, Statistics, 
and Computer Science, University of Illinois in Chicago (UIC),
851 S Morgan St, Chicago, IL 60607, USA}
\email{pengel@uic.edu}
\author[de Gaay Fortman]{Olivier de Gaay Fortman}
\address{Department of Mathematics, Utrecht University,  Budapestlaan 6, 3584 CD Utrecht, The Netherlands.}
\email{a.o.d.degaayfortman@uu.nl}
\author[Schreieder]{Stefan Schreieder}
\address{Leibniz University Hannover, Institute of Algebraic Geometry, Welfengarten 1, 30167 Hannover, Germany.}
\email{schreieder@math.uni-hannover.de}

\date{\today}

\begin{document}

\begin{abstract} 
We survey the Mumford construction of degenerating
abelian varieties, with a focus on the analytic version
of the construction, and its relation to toric geometry. Moreover, we study the geometry and Hodge
theory of multivariable  
degenerations of abelian varieties
associated to regular matroids, and 
extend some fundamental results of Clemens
on $1$-parameter semistable degenerations to the
multivariable setting.
\end{abstract}

\setcounter{tocdepth}{2}
\maketitle 

\tableofcontents

\section{Introduction}

In 1972, Mumford gave an analytic construction
of degenerations of abelian varieties
over complete rings \cite{mumford1972}. 
It played an important role in the development
of the theory of toric varieties \cite{kempf}
and toroidal compactifications of 
locally symmetric varieties \cite{mumford-amrt}.
When working over $\bC$, 
the basic idea is that one may view
a degeneration of abelian varieties analytically,
as the quotient of an appropriate intermediate cover
of the universal cover. 

It is well-known that, over the complex numbers, 
any abelian $g$-fold 
$A\simeq \bC^g/H_1(A,\bZ)$ is the quotient 
of a vector space $\bC^g$ by a lattice 
$H_1(A,\bZ)\subset \bC^g$ 
of rank $2g$. Suppose that 
$A=X_t$ is a general fiber of a degenerating
family $X\to \Delta$ of
abelian varieties over the unit disk 
$\Delta\subset \bC$. Let $V_\bZ\coloneqq H_1(X_t,\bZ)$.
Then, there is a saturated sublattice 
$W_{-1}V_\bZ \subset V_\bZ$ of 
the fundamental group 
$V_{\mathbb Z} = \pi_1(X_t)$, consisting of $1$-cycles
which are invariant under the monodromy
of the punctured family $X^*\to \Delta^*$.
It contains a further sublattice $W_{-2}V_\bZ\subset W_{-1}V_\bZ$ consisting of {\it vanishing cycles}; that is, the 
$1$-cycles on $X_t$ which are null-homologous
in $X$.
The filtration 
$W_{-2}V_\bZ\subset W_{-1}V_\bZ\subset W_0V_\bZ=V_\bZ$ defines the 
{\it weight filtration}.

The subgroup $W_{-1}V_\bZ \subset V_\bZ = \pi_1(X_t)$
gives rise to
a cover 
$Y^*\to X^*$ corresponding fiberwise
to the intermediate cover 
$\bC^g\simeq \wX_t\to Y_t\to X_t$ of the universal cover, 
whose Galois group over $X_t$
is the graded piece ${\rm gr}^W_0V_\bZ$.
In general, $Y_t$ is a {\it semiabelian variety}---an
algebraic torus bundle over an abelian variety
of dimension $\tfrac{1}{2}\textrm{rank}\, {\rm gr}_{-1}^W V_\bZ$.  When $W_{-1}V_\bZ$
has rank $g$, 
then we have 
$Y_t\simeq \bC^g/\bZ^g \simeq (\bC^*)^g$, 
and we call the degeneration {\it maximal}.

In the case of a maximal
degeneration, the intermediate
cover $Y^*$ of the punctured family admits an analytic open embedding $Y^*\hookrightarrow 
(\bC^*)^g\times \bC^*$ 
into an algebraic torus, with the map
to $\Delta^*$ given by the projection
to the second factor. 
Thus, by the theory of toric varieties, it is possible
to extend $Y^*\to \Delta^*$ 
to a family $Y\to \Delta$.
More precisely, we take a toric
extension of $(\bC^*)^g\times \bC^*$
for which the fiber-preserving action of ${\rm gr}_0^WV_\bZ$ extends to an action
on the central fiber $Y_0$. 
The quotient of $Y\to \Delta$
by the extended action of ${\rm gr}_0^WV_\bZ$ 
produces a new model of the degeneration
with particularly nice properties, 
e.g.~toroidal singularities.

The original paper \cite{mumford1972} performs
this construction more generally over any
complete ring, corresponding to a possibly
higher-dimensional base, 
and using formal algebraic geometry. 
Mumford's work was inspired by an influential
1959 manuscript of Tate \cite{tate_curve} 
on degenerating elliptic curves.
In the late 1970's, 
Nakamura and Namikawa
worked out the theory in the complex-analytic
setting \cite{nakamura75, nakamura77,
namikawa76, namikawa79, namikawabook}, 
culminating in a method of patching together
Mumford constructions over the cones of the second
Voronoi fan, 
to produce a relatively proper, analytic
extension of the universal abelian variety
$\cX_g\to \cA_g$ to a toroidal 
compactification of $\cA_g$.

An impressive, and notoriously technical, further
advancement was the work of Faltings--Chai \cite{fc} 
in the early 1990's, who extended the Mumford construction 
and the theory of toroidal compactifications 
to the arithmetic setting. In the later 
1990's, Alexeev--Nakamura \cite{alexeev-nakamura}
and Alexeev \cite{alexeev}
used the Mumford construction to compactify the universal
abelian torsor $(\cX_g^\star,\Theta_g)$ with theta
divisor, cf.~Construction \ref{torsor}, 
and provide a modular interpretation
of this compactification, as the normalization
of the closure of the space of 
KSBA-stable pairs $(X,\epsilon \Theta)$ 
in the proper DM stack of log
general type varieties.

\subsection{Contents} 
The goal of this paper is 
to provide a ``working mathematician's guide'' 
to the Mumford construction.
Thus, we usually work analytically over $\bC$,
though we do also touch on the question of algebraicity
of Mumford constructions (Prop.~\ref{algebraicity}).
Furthermore, we largely restrict our attention
to maximal degenerations, though most of the
results presented here apply 
in the non-maximal case.
Many of the ideas of the paper are to be 
found scattered through the literature; some are
difficult to find, and others appear to be new,
such as Theorem \ref{regular-extension-thm}.

After reviewing in Section \ref{sec:prelim}
preliminary material on principally
polarized abelian varieties, their Hodge theory,
their moduli, their degenerations, 
and toroidal extensions and compactifications of $\cA_g$,
we dive into the main constructions in Section 
\ref{sec:mumford}:

\subsubsection*{Mumford constructions and examples}

Using tools from toric geometry, we construct
maximal degenerations of principally polarized
abelian $g$-folds. 
The constructions are presented in increasing
levels of generality and are broadly divided
into two classes: {\it fan} constructions
and {\it polytope} constructions (see Section
\ref{list} for a list). The fan construction
only produces a degeneration complex-analytically,
but has the advantage of being relatively simpler,
and having more readable geometry.
The advantage of
the polytope construction is that it 
outputs a relatively projective degeneration.

The equality
of the two constructions is examined in Section
\ref{sec:comparison}, while the 
topology is discussed in Section \ref{sec:wf}, 
where we analyze the weight filtration 
on the limiting mixed Hodge structure of a general fiber.  
We also study the effect, on Mumford constructions,
of a toroidal base change and 
of replacing the polarization by a multiple,
see Section \ref{sec:base-change}.

\subsubsection*{Regular matroids} 
We continue in Section \ref{sec:regular-matroids} with
a more detailed analysis of Mumford constructions
associated to regular matroids 
(Constr.~\ref{shifted-matroidal-construction}). A
{\it matroid} $\underline{R}$
is a collection of subsets of a finite set $E$,
encoding the notion
of linear
independence of a set 
of vectors in a vector space
(Def.~\ref{matroid-def}).
An embedding $E\hookrightarrow \bF^g$ 
into a vector space
over a field $\bF$, realizing this 
collection of independent sets,
is a {\it realization} of $\underline{R}$. Matroids
which admit a realization over any field
are {\it regular} (Def.~\ref{def:regular}),
though $\bF_2$ and $\bF_3$ suffice.

Associated to a regular matroid $\underline{R}$ are
the so-called {\it matroidal}, {\it shifted matroidal},
and {\it transversely shifted matroidal}
Mumford constructions, see Sections \ref{sec:reg} and \ref{sec:shift-reg}. 
Related degenerations
were explored for cographic
matroids by Dancso--McBreen--Shende
\cite[Sec.~8.3]{mcbreen}, building on
unpublished work of Hausel--Proudfoot.
Perhaps unsurprisingly, regular matroids
are intimately connected to the 
total space of a Mumford
construction being regular, i.e.~smooth. In fact,
a Mumford construction $X\to \Delta^k$ such that
(i) $X$ is regular, and (ii) over each coordinate
hyperplane $\{u_i=0\}$, $i=1,\dots, k$, of the polydisk,
the vanishing cycles span a $1$-dimensional space, 
is necessarily a transversely shifted
matroidal degeneration, and vice versa
(Props.~\ref{X-smooth} and \ref{mono-1}).

As we explain in Section \ref{sec:monodromies}, for a 
family $f^*\colon X^*\to (\Delta^*)^k$ of 
$g$-dimensional PPAVs, the monodromy about the 
$i$-th coordinate hyperplane defines, 
via the principal polarization, a symmetric matrix 
$B_i \in {\rm{Sym}}_{g \times g}(\bZ)$ 
(Def.~\ref{def:Bi}). The cone in 
${\rm{Sym}}_{g \times g}(\bR)$ generated by 
$\{B_1, \dotsc, B_k\}$ is
the \emph{monodromy cone} 
of the family $f^\ast$ (Def.~\ref{monodromy-cone}). 
Such a cone is \emph{matroidal} if it 
is induced by an integral realization of a
regular matroid (Def.~\ref{matroidal-cone}).
Transversely shifted matroidal degenerations 
are examples; they are of particular importance to our 
companion paper \cite{companion}. 

\begin{theorem}\label{regular-extension-thm}
Let $f^*\colon X^*\to (\Delta^*)^k$ be a smooth family
of PPAVs of dimension $g$, whose monodromy cone
is {\it matroidal} (Defs.~\ref{monodromy-cone}, \ref{matroidal-cone}). 
There is a flat, $K$-trivial extension
$f\colon X\to \Delta^k$ which is a 
{\rm nodal} morphism (Def.~\ref{D-nodal}), and $f$ may be assumed {\rm strictly nodal}
if the monodromy about each coordinate hyperplane
is imprimitive.

Moreover, given $k$ generators
of a matroidal cone, there exists a family
$f^\ast\colon X^\ast \to (\Delta^\ast)^k$ 
of PPAVs whose monodromies are the 
specified generators, and an extension as above, 
which is 
the restriction 
of a projective family over a quasiprojective variety
$Y$ to a polydisk $\Delta^k\subset Y$.
\end{theorem}

See Theorem \ref{theorem:extension} and Corollary 
\ref{corollary:degeneration}, respectively, for 
more algebraic formulations of the first and 
second statements of the above theorem.

As a particular application, the 
relative intermediate Jacobian
fibration $IJY^\circ\to (\Delta^*)^{10}$ 
of the punctured universal deformation $Y^\circ\to (\Delta^*)^{10}$
of the Segre cubic $Y_0$ (Ex.~\ref{gwena-ex}) 
admits such a filling $IJY\to \Delta^{10}\simeq {\rm Def}_{Y_0}$
as does the relative Jacobian fibration
of the universal deformation of a nodal
curve 
$C_0$ (Exs.~\ref{nodal}, \ref{cographic-ex}). 
The respective matroids are
the Seymour--Bixby matroid $\underline{R}_{10}$
and the cographic matroid $M^*(G)$ 
of the dual graph $G=\Gamma(C_0)$ of $C_0$
(Exs.~\ref{seymour-bixby}, \ref{cographic}).

\subsubsection*{Nodal and semistable morphisms}

Our investigations connect in a very natural way
to the notion of a {\it semistable morphism}
$f\colon X\to Y$, introduced by Abramovich--Karu 
\cite{ak} and the more restrictive 
notion of a {\it nodal morphism} 
(Def.~\ref{D-nodal}). The former are morphisms
between smooth spaces $X$ and $Y$, which
are \'etale or analytically-locally a 
product of snc degenerations $x^{(1)}\cdots x^{(m)}=u$,
and the latter (nodal morphisms) additionally 
satisfy $m\leq 2$.
The main result of Adiprasito--Liu--Temkin
\cite{alt} is that, for {\it any} 
dominant morphism $f\colon X\to Y$,
there is 
a regular
alteration $Y'\to Y$ of the base and a birational
modification of the base change $X'\to X\times_{Y}Y'$ so that $f'\colon X'\to Y'$ is semistable.
Furthermore, once semistability is achieved, further
base changes admit a functorial semistable resolution.

Once we have proven that transversely
shifted matroidal
degenerations are nodal morphisms, we continue
with a general analysis of the Hodge theory
of semistable morphisms, in Section \ref{sec:semistable}.
We define a multi-parameter analogue 
(Prop.~\ref{clemens-collapse}) of the 
Clemens retraction for $1$-parameter semistable
degenerations, and investigate the relationship
(Prop.~\ref{specialization}) 
between the dual complex of the central
fiber, and the graded piece
${\rm gr}^W_0V_\bZ$ of the weight filtration
on $V_\bZ = H_1(X_t,\bZ)$. 

In Section \ref{sec:resolution}, we 
instantiate
explicitly 
the functorial resolution of \cite[Thm.~4.4]{alt},
in the case of a base change of a nodal morphism
(Thm.~\ref{can-res}). 
We apply this resolution algorithm to 
transversely shifted matroidal
degenerations in Section \ref{sec:res-shifted}. 
Refinements of these results will play an 
important role in \cite[Sec.~3 and 4]{companion}.

\subsubsection*{The second Voronoi 
fan and Alexeev's theorem}

Finally, we review the work of Alexeev, Nakamura,
Namikawa and Faltings--Chai 
on the extension of the universal family
$\cX_g\to \cA_g$ over the toroidal compactification
$\overline{\cA}_g^{\rm vor}$ associated 
to a distinguished fan
$\fF_{\rm vor}$ (Defs. 
\ref{vor-cell}, \ref{vor-fan}), 
whose support is the cone $\cP_g^+$ 
of positive semi-definite $g\times g$
matrices with rational null space. 

We sketch a proof of Alexeev's theorem that 
$\overline{\cA}_g^{\rm vor}$ 
is the normalization of the
KSBA compactification of the moduli space of
abelian torsors with theta divisor $(X,\epsilon \Theta)$,
paying particular attention to the subtle differences
between Alexeev's universal family 
$\cX_g^\star\to \cA_g$ and the universal family of 
abelian varieties $\cX_g\to \cA_g$ 
(Constr.~\ref{torsor} and Rem.~\ref{torsor-remark}). 

 We also provide a brief review of the extensive 
 literature on the cones of the 
 second Voronoi fan, for $g\leq 6$ 
 (Ex.~\ref{voronoi-ex}).

 \subsection{Index of constructions}
 \label{list}

The various forms of the Mumford construction presented
in this paper are thus:

\begin{enumerate}
    \item[(\ref{mumford}):] Via fans, 
    over a $1$-parameter base (i.e., a disk), 
    and over a family of $1$-parameter bases,
    complete with respect to a fixed monodromy operator $T\colon H_1(X_t,\bZ)\to H_1(X_t,\bZ)$,
    encoded by a symmetric bilinear form
    $B\in {\rm Sym}^2 ({\rm gr}^W_0 V_\bZ)^\vee$.\smallskip
    \item[(\ref{mumford-multi-fan}):] 
    Via fans, over a polydisk $\Delta^k$, and 
    over a family of such polydisks,
    complete with respect to a fixed collection
    of monodromy bilinear forms $B_i\in 
    {\rm Sym}^2 ({\rm gr}^W_0 V_\bZ)^\vee$, 
    $i=1,\dots,k$. \smallskip
    \item[(\ref{singular-base}):] Via fans,
    over the toroidal extension 
    $\cA_g\hookrightarrow \cA_g^{\bB}$ 
    associated to a rational 
    polyhedral cone $\bB=\bR_{\geq 0}\{B_1,\dots,B_k\}\subset \cP_g^+$.\smallskip
    \item[(\ref{mumford-polytope}):] Via polytopes,
    over a polydisk $\Delta^k$,
    associated to a collection $\{b_1,\dots, b_k\}$ of convex piecewise linear
    functions $\bR^g\to \bR$ with appropriate $\bZ^g$-periodicity, and over a family of polydisks, 
    complete with respect to the associated
    monodromy cone $\bB$. \smallskip
    \item[(\ref{mumford-polytope-2}):] Via polytopes,
    over the toroidal extension 
    $\cA_g\hookrightarrow \cA_g^{\bB}$ 
    associated to a rational 
    polyhedral cone $\bB=\bR_{\geq 0}\{B_1,\dots,B_k\}\subset \cP_g^+$.\smallskip
    \item[(\ref{veronese}):] Via polytopes,
    as in (\ref{mumford-polytope}), 
    but where only $d$ times the principal polarization extends
    to a relatively ample line bundle on the family.\smallskip
    \item[(\ref{shifted-matroidal-construction}):] 
    As special cases of
    (\ref{mumford-polytope}) and (\ref{veronese}), 
    associated to a regular matroid $\underline{R}$
    of rank $g$, and a hyperplane arrangement inducing
    this regular matroid from the set of normal vectors.\smallskip
    \item[(\ref{voronoi-mumford}):] Associated
    to a ``tautological'' version of (\ref{mumford-polytope-2})
    for a cone $\bB\in \fF_{\rm vor}$ of the second Voronoi fan,
    giving a local 
    analytic extension of the universal family
    of abelian varieties.\smallskip
    \item[(\ref{torsor}):] As in (\ref{voronoi-mumford}), but
    giving an extension $\overline{\cX^\star_g}^{\rm vor}\to \overline{\cA}_g^{\rm vor}$ of the universal family
    of abelian torsors with theta divisor.
\end{enumerate}

An extensive collection of examples (Exs.~\ref{ex:tate},
\ref{theta-1-param}, \ref{theta-3-param}, \ref{tate-base},
\ref{cographic-ex}, \ref{r10-ex}, \ref{shifted-matroidal},
\ref{shifted-ex-2}, \ref{ex:mon-sep}, \ref{voronoi-ex}), 
with figures, is also
provided in the text, 
see especially Section \ref{sec:examples}.
The first of these 
(Ex.~\ref{ex:tate}) is the prerequisite
ur-example of the Mumford construction:
the {\it Tate curve}, i.e.~the extension
of the family $\bC^*/u^\bZ\to \Delta_u^*$ 
of elliptic curves by an irreducible nodal curve. 
 
\subsection{Acknowledgments} 
We thank Valery Alexeev, Younghan Bae, David Holmes,
and Mirko Mauri for useful discussions and comments.

PE was partially supported by NSF grant DMS-2401104. 
OdGF and StS have received funding from the European Research Council (ERC) under the
European Union’s Horizon 2020 research and innovation programme under grant agreement
N\textsuperscript{\underline{o}} 948066 (ERC-StG RationAlgic). 
OdGF has also received funding from the ERC Consolidator Grant FourSurf N\textsuperscript{\underline{o}} 101087365. 
The research was partly conducted in the framework of the DFG-funded research training group RTG 2965: From Geometry to Numbers.

\section{Preliminary material} 
\label{sec:prelim}

\subsection{Algebraic and analytic stacks} By a 
\emph{DM algebraic stack}, or simply \emph{DM stack}, 
we will mean a separated Deligne--Mumford stack of finite 
type over $\bC$. Similarly, a \emph{DM analytic stack} will 
be a separated Deligne--Mumford analytic stack $X$ in the 
sense of \cite[Def.~5.2]{toen}. Thus, $X$ is a stack on 
the site of complex analytic spaces such that the diagonal 
is representable and finite and there exists an analytic 
space $Y$ and a surjective \'etale morphism $Y \to X$. 
It follows that $X$ is locally modeled as a finite 
quotient of an analytic space, see \cite[Prop.~5.4]{toen}, 
and that the analytification of a DM algebraic stack is 
a DM analytic stack, see \cite[Lem.~5.5]{toen}.

\subsection{Principally polarized abelian varieties}

Let $\cA_g$ denote the DM stack
of principally polarized abelian varieties (PPAVs) 
of dimension $g$, over $\bC$.
Since a PPAV $X$ 
is uniquely determined by
its polarized weight $-1$ 
Hodge structure on $H_1(X,\bZ)$, 
the period map defines an isomorphism 
$\cA_g\simeq \Sp_{2g}(\bZ)\backslash \cH_g$ to
an arithmetic quotient of a 
Hermitian symmetric domain of Type III. 
We review this construction now.

\begin{definition} A {\it $\bZ$-polarized 
Hodge structure $(V_\bZ, H^{\bullet,\bullet}, L)$ 
of weight $k$} 
is a $\bZ$-module $V_\bZ$ together 
with an integral, non-degenerate,
$(-1)^k$-symmetric bilinear form 
$L \colon V_\bZ\otimes V_\bZ\to \bZ$,
and a Hodge decomposition $$V_\bC\coloneqq V_\bZ\otimes \bC= 
\bigoplus_{p+q=k} H^{p,q}$$
satisfying the following conditions:
\begin{enumerate}
\item $H^{q,p}=\overline{H^{p,q}}$ for all $p+q=k$,
\item $L(H^{p,q}, H^{p',q'})=0$ unless $p=q'$, $q=p'$,
\item $(-1)^{k(k-1)/2}i^{p-q}L(\bar v, v)>0$ 
for all $0\neq v\in H^{p,q}$.
\end{enumerate}
\end{definition}

\begin{definition} A $\bZ$-polarized 
Hodge structure $(V_\bZ, H^{\bullet,\bullet}, L)$
is {\it principally polarized} 
if the pairing $L $ is unimodular.
\end{definition}
 
Let $(V_\bZ, L)$ be a unimodular
symplectic lattice, and
consider the Lagrangian Grassmannian 
${\rm LGr}(V_\bC, L)$.
It is the projective flag variety of isotropic $g$-dimensional
subspaces of $V_\bC$.
The polarized weight $-1$ 
Hodge structures on $(V_\bZ,L)$ with a Hodge decomposition
of the form $V_\bC=H^{-1,0}\oplus H^{0,-1}$ define 
an analytic open subset of ${\rm LGr}(V_\bC, L)$, given by 
\begin{align}
\label{lag-gr}
\{H^{-1,0}\subset V_\bC \,:\,L\vert_{H^{-1,0}} = 0
\textrm{ and }
iL(\bar v, v)>0 \textrm{ for } 0\neq v\in H^{-1,0}\}.
\end{align} 
Given $[H^{-1,0}]$ in (\ref{lag-gr}),
we may define a complex torus
$$X\coloneqq V_\bC/(V_\bZ+H^{-1,0}).$$ We have canonical
isomorphisms $H_1(X,\bZ)\simeq V_\bZ$ and 
$H^1(X,\bZ)\simeq V_\bZ^\vee$.
Thus, the symplectic form
$L \in V_\bZ^\vee\wedge V_\bZ^\vee$ defines an element 
$$L \in \wedge^2 H^1(X,\bZ)\simeq H^2(X,\bZ).$$ 
The condition that $H^{-1,0}$ is Lagrangian for $L$
amounts to the property that $L \in H^{1,1}(X)$ is a 
Hodge class. Hence
$L$ determines  
a holomorphic line bundle $\cL\to X$, 
unique  up to translation by ${\rm Pic}^0(X)$.
Finally, the 
condition $iL(\bar v, v)>0$ ensures that any
lift $\cL$ is ample, and so
in fact, $X$ is an abelian variety (i.e.~projective). 

Choosing a standard symplectic basis of $V_\bZ$
produces an isometry $(V_\bZ, L)\simeq (\bZ^{2g},\cdot)$,
where $\bZ^{2g}$ is generated by vectors 
$e_i, f_i$ for $i=1,\dots,g$ and the unimodular 
symplectic form $\cdot$ 
satisfies $e_i\cdot e_j = f_i\cdot f_j = 0$ and $e_i\cdot f_j=\delta_{ij}$.

\begin{definition}\label{def:siegel}
The {\it Siegel upper half-space} $\cH_g$ is
the space of symmetric $g\times g$ matrices 
with positive-definite
imaginary part.
\end{definition}

A choice of symplectic basis of $(V_\bZ, L)$ identifies
(\ref{lag-gr}) with $\cH_g$.
In a standard symplectic basis, the Lagrangian 
$H^{-1,0}\subset V_\bZ\otimes \bC$
is the span of the columns of some 
$2g\times g$ {\it period matrix} 
\begin{align*}
\begin{pmatrix} \sigma \\ I \end{pmatrix} 
\in {\rm Mat}_{2g \times g}(\bC),
\end{align*}
which we write in $2\times 1$ block form. 
The condition that $L\in H^{1,1}(X)$ is 
Hodge amounts to the symmetry
of $\sigma$, while the positivity condition 
$iL(\bar v, v)>0$ amounts
to the fact that the imaginary part
${\rm Im}(\sigma)>0$ is positive-definite. 
Hence $\sigma\in \cH_g$.
Changes of symplectic basis, 
i.e.~elements of $\Sp_{2g}(\bZ)$,
act on the left, by 
$2\times 2$ block matrices, 
$$\twobytwo{A}{B}{C}{D}\begin{pmatrix} \sigma \\ 
I \end{pmatrix} = 
\begin{pmatrix} A\sigma +B \\ C\sigma +D \end{pmatrix}.$$ 
Renormalizing the generators of our Lagrangian subspace
to be dual to the $e_i$ with respect to the symplectic form
corresponds to right multiplication by $(C\sigma+D)^{-1}$.
So we get the Lagrangian corresponding to
$(A\sigma+B)(C\sigma+D)^{-1}\in \cH_g$, which is
the standard action of ${\rm Sp}_{2g}(\bZ)$ on $\cH_g$.

\begin{definition}  \label{ppav}
The pair $(X,L)$ is called a 
{\it principally polarized 
abelian variety},
or PPAV.
\end{definition}

For any representative $\cL\in {\rm Pic}(X)$ of $L$,
we have $h^0(X,\cL)=1$, and so there is a unique divisor 
$\Theta\in |\cL|$ called the {\it theta divisor}. 

It follows
from the above discussion that 
the moduli stack of PPAVs $(X,L)$
is given by the orbifold 
(i.e.~smooth DM
analytic stack)
$\cA_g \simeq \Sp_{2g}(\bZ)\backslash \cH_g$.
Furthermore, the universal family $\cX_g\to \cA_g$
of PPAVs  is uniformized by $\bC^g\times \cH_g$ and can 
be presented as a quotient, too:
$$\cX_g \simeq  \left(\bZ^{2g} \rtimes
\Sp_{2g}(\bZ) \right)\backslash \left( \bC^g\times \cH_g\right).$$

\subsection{Degenerations of PPAVs} \label{sec:monodromies}

In the following sections, 
we discuss the monodromy of 
degenerations of PPAVs, especially
in relation to toroidal extensions 
of $\cA_g$. See \cite{prym_ref} 
for reference.

Let $f\colon (X,L)\to \Delta^k$ be a degeneration of 
PPAVs of dimension $g$ over a polydisk $\Delta^k$ with 
coordinates $u_1, \dotsc, u_k$, such that the discriminant
locus is the union of the coordinate hyperplanes 
$V(u_i)=\{u_i=0\}$, for $i=1,\dots,k$. Fix a 
base point $t\in (\Delta^*)^k$ and 
let $V_\bZ\coloneqq H_1(X_t,\bZ)$.
Suppose that the monodromy transformation
$T_i\colon V_\bZ\to V_\bZ$ associated to 
the simple, oriented loop 
$\gamma_i\in \pi_1((\Delta^*)^k,t)\simeq \bZ^k$
is unipotent---for instance, by a result of Clemens 
\cite[Thm.~7.36]{clemens}, this holds if the general fiber
over $V(u_i)$ has reduced normal crossings.

Choosing a symplectic basis $(V_\bZ, L)\simeq (\bZ^{2g},\cdot)$,
we may view $T_i$ as acting on the reference lattice $\bZ^{2g}$. 
Let $N_i = \log(T_i) = T_i - I$ be its logarithm,
where $I$ denotes the identity matrix of size
$2g\times 2g$. Note that $N^2_i=0$
and $N_i \circ N_j = N_j \circ N_i$ commute. Let 
$N = \sum_{i=1}^k r_iN_i$, $r_i\in \bN$, 
be any strictly positive linear combination. 
Then $N$ is the monodromy 
of the restriction of $f$ to the cocharacter 
$\Delta\to \Delta^k$ defined by 
$u\mapsto (u^{r_1},\dots, u^{r_k})$. 
By \cite[Thm.~3.3]{cattani-kaplan}, 
all $(r_1,\dots,r_k)\in \bN^k$
define the same increasing 
{\it weight filtration} 
\begin{align*} W_{-2}&\coloneqq ({\rm im}\,N)^{\rm sat} \\ W_{-1}&\coloneqq \ker N \\ W_0&\coloneqq V_\bZ.\end{align*}
More generally, for any 
$(r_1,\dots,r_k)\in (\bZ_{\geq 0})^k$,
the filtration so defined depends only on the polyhedral
face of $(\bR_{\geq 0})^k$ containing $N$.

The above weight filtration may also be 
described as follows:
\begin{align}\label{eq:weight-filtration}
\textstyle W_{-2}= \left(\sum_ {i=1}^k
{\rm im}\,N_i\right)^{\rm sat} 
\quad \text{and}\quad W_{-1}=
\bigcap_{i=1}^k \ker(N_i).
\end{align}
Indeed, by the saturatedness of the above 
filtration, it suffices to prove this for 
rational, and, in fact, for real coefficients. 
The inclusions ${\rm im}(\sum N_i)\subset 
\sum {\rm im}\,N_i$ and 
$\bigcap \ker(N_i)\subset \ker (\sum N_i)$ 
are clear.
To prove the converse, let $(r_1,\dots ,r_k)
\in (\bR_{>0})^k$ and note that 
${\rm im}(\sum r_i N_i)$ and $\ker (\sum r_i N_i)$ 
do not depend on the choice of  $r_i>0$, see 
\cite[Thm.~3.3]{cattani-kaplan}.
The inclusions in question thus follow from a limit
argument where $r_j=1$ and $r_i\to 0$ for 
$i\neq j$, applied to ${\rm im} (\sum r_iN_i)\subset 
{\rm im}(\sum N_i)$ and 
$\ker(\sum N_i)\subset \ker(\sum r_iN_i)$.

\begin{definition}\label{unipotent-parabolic}
Consider the standard Lagrangian subspace
$\bZ e_1\oplus \cdots \oplus \bZ e_g\subset (\bZ^{2g},\cdot)$. 
Its stabilizer is 
the {\it parabolic group}
$$P_\bZ\coloneqq \left\{\twobytwo{A}{B}{0}{A^{-T}}
\in \Sp_{2g}(\bZ)\right\}$$ with $A\in \GL_g(\bZ)$
and $BA^T=AB^T$. We define
the {\it unipotent subgroup} of 
$P_\bZ$ to be
$$U_\bZ\coloneqq \left\{ 
\twobytwo{I}{B}{0}{I}\colon B
\in {\rm Sym}_{g\times g}(\bZ)\right\}$$
and the {\it Levi quotient} to be 
$P_\bZ/U_\bZ\simeq \GL_g(\bZ)$, which can also
be lifted into ${\rm Sp}_{2g}(\bZ)$ as the block
diagonal matrices (i.e.~matrices with
$B=0$ in $P_\bZ$).
\end{definition}

The collection of commuting unipotent matrices 
$T_i\in {\rm Sp}_{2g}(\bZ)$ 
 can be simultaneously conjugated into 
 the unipotent subgroup $U_\bZ$
as they fix a coisotropic space 
 (given by $W_{-1}$) and hence fix a Lagrangian
 subspace. 
Thus, choosing a basis appropriately, 
we may assume that the monodromies
$T_i$ are all of the form 
\begin{align}\label{simultaneous}T_i = \twobytwo{I}{B_i}{0}{I}\end{align}
for symmetric matrices $B_i$.

\begin{definition} \label{def:Bi}
Let $f\colon (X,L)\to \Delta^k$ be 
a degeneration
of PPAVs with unipotent monodromies 
about the coordinate
hyperplanes. We define the 
{\it monodromy bilinear
forms} 
$B_i\in {\rm Sym}^2 ({\rm gr}^W_0 V_\bZ)^\vee$ 
for $i=1,\dots,k$ by the formula
\begin{align} \label{b-formula}
B_i(x,y)= L(N_ix,y).\end{align}
 Observe that $B_i$ depends only
 on the punctured family 
 $f^*\colon (X^*,L^*)\to (\Delta^*)^k$
 and so the definition
 extends naturally to families of PPAVs
 over the punctured polydisk. 
\end{definition}

This provides a coordinate-free definition of
the matrices $B_i$ 
from above. 
Implicit in the above definition is the claim that $N_i$ 
contains $W_{-1}V_{\mathbb Z}$ in its kernel, which
follows from (\ref{eq:weight-filtration}).

\begin{definition}\label{monodromy-cone} 
Choose a symplectic basis of $V_\bZ$ 
such that $T_i$ has the form \eqref{simultaneous} 
for each $i$, which identifies
each 
$B_i$ with a symmetric matrix 
$B_i\in {\rm Sym}_{g\times g}(\bZ)$.
The span 
$$\bB_f\coloneqq 
\bR_{\geq 0}\{B_1,\,\dots,\,B_k\} 
\subset {\rm Sym}_{g\times g}(\bR)$$ 
is the {\it monodromy cone} 
associated to the degeneration 
$f\colon (X,L)\to \Delta^k$. 
This definition extends to any
family of $g$-dimensional PPAVs 
$f^\ast \colon (X^\ast, L^\ast) \to (\Delta^\ast)^k$ 
with unipotent monodromy.
\end{definition}

Note that the collection $(B_i)_{i=1,\dots,k}$ 
of monodromy matrices, and hence the 
monodromy cone $\bB_f$, is unique, 
up to the simultaneous action of $A\in \GL_g(\bZ)$
by $B_i\mapsto AB_iA^T$. This action
corresponds to the conjugation action 
$T_i \mapsto M T_i M^{-1}$
of $P_\bZ \subset {\rm{Sp}}_{2g}(\bZ)$, 
$M \in P_\bZ$, which descends
to the Levi quotient 
$P_\bZ/U_\bZ\simeq \GL_g(\bZ)$ because
$U_\bZ$ is commutative.

The symplectic basis of $V_\bZ$
determines a lift of the classifying map 
$\Phi\colon (\Delta^*)^k\to \cA_g$ to a 
holomorphic period map $\widetilde{\Phi}\colon 
\bH^k \to \cH_g$, where $\bH^k\to (\Delta^*)^k$
is the universal cover,
$\bH \coloneqq \set{\tau \in \bC \mid \im(\tau)>0}$.
Take coordinates
$(\tau_1,\dots,\tau_k)\in \bH^k$, with the universal
covering map given by $u_i={\rm exp}(2\pi i \tau_i)$.
This lifted map satisfies the equivariance property
$$\widetilde{\Phi}(\tau_1,\dots,\tau_i+1,\dots,\tau_k)
=T_i\cdot \widetilde{\Phi}(\tau_1,\dots,\tau_k) 
= \widetilde{\Phi}(\tau_1,\dots,\tau_k)+B_i$$
corresponding to the deck transformation for the
generator $\gamma_i\in \pi_1((\Delta^*)^k, t)$.

\begin{definition}
Define a holomorphic map to the flag variety 
$\bD^\vee\coloneqq {\rm LGr}(V_\bZ\otimes \bC, L)$ 
by
\begin{align*} \widetilde{\Psi}\colon \bH^m
& \to \bD^\vee \\ \tau&\mapsto 
\widetilde{\Phi}(\tau)- (\tau_1B_1+\dots+\tau_k B_k)\end{align*}
and denote by $\Psi\colon (\Delta^*)^k\to \bD^\vee$ 
its descent to $(\Delta^*)^k$. 
\end{definition}

Note that $\widetilde{\Psi}$ descends 
because the $-(\tau_1 B_1+\cdots +\tau_kB_k)$ 
term cancels the equivariance
of $\widetilde{\Phi}$, and so makes the map invariant 
under the action of $\bZ^k$.
We now recall Schmid's multivariable nilpotent 
orbit theorem \cite[Thm.~4.12]{schmid},
applied to our setting:

\begin{theorem}\label{schmid} $\Psi$ 
extends to a holomorphic map 
$\Delta^k\to \bD^\vee$.
Let $\Psi(0)$ denote the extension to the origin, 
and consider $\widetilde{\Phi}_{\rm nilp}(\tau) \coloneqq  
\Psi(0)+(\tau_1B_1+\dots+\tau_k B_k)$.
Then 
\begin{enumerate}
\item $\widetilde{\Phi}_{\rm nilp}(\tau)\in \cH_g
\subset \bD^\vee$ for 
all sufficiently large ${\rm Im}\,\tau_i$ and
\item\label{metric-decay} the distance
$d(\widetilde{\Phi}(\tau),
\widetilde{\Phi}_{\rm nilp}(\tau))$ 
decays exponentially in ${\rm Im}\,\tau_i$.
\end{enumerate}
\end{theorem}

Here distance is measured in the natural left
$\Sp_{2g}(\bR)$-invariant metric on $\cH_g$.

\begin{definition} Let 
$\cP_g\coloneqq \{B\in 
{\rm Sym}_{g\times g}(\bR)\,:\,B>0\}$
be the cone of positive-definite matrices 
and let $\cP_g^+$ be its {\it rational closure},
consisting of all positive semi-definite 
matrices whose kernel is a rational subspace
of $\bR^g$.
\end{definition} 

There is a natural stratification
$$\cP_g^+ = \textstyle\cP_g \sqcup 
\bigsqcup_{V_1} \cP_{g-1} \sqcup \bigsqcup_{V_2}
\cP_{g-2} \sqcup\cdots \sqcup \bigsqcup_{V_g} \cP_0$$ where the $V_i\subset \bZ^g$ range
over all primitive integral 
sublattices of $\bZ^g$
of dimension $i$, and the
relevant copy of $\cP_{g-i}$ is the cone of positive-definite
bilinear forms on $\bR^g/(V_i\otimes \bR)$. See Figure 
\ref{fig:second-voronoi} to visualize
the projectivization of $\cP_2^+$, which is a cusped
hyperbolic disk.

It follows from item (1) of Theorem \ref{schmid} that:

\begin{corollary}\label{pos-def-cor} 
The monodromy cone $\bB_f$
is contained in the rational closure $\cP_g^+$ of 
the positive-definite 
$g\times g$ matrices, uniquely
up to the action of $\GL_g(\bZ)$. \end{corollary}

We have that $W_{-2}^\perp=W_{-1}$ where 
the perpendicular is taken with respect to $L$. 
Thus $L$ descends to a unimodular 
symplectic form on $W_{-1}/W_{-2}$. In fact, by 
\cite[Thm.~6.16]{schmid}, we have the following fundamental
theorem on the existence of the limit MHS:

\begin{theorem}
The tuple $(V_\bZ, L, \Psi(0) , W_\bullet)$ defines a 
graded-polarized  $\bZ$-mixed Hodge structure,
the {\rm limit mixed Hodge structure}. 
In particular, the filtration of $W_{-1}/W_{-2} \otimes \bC$ 
induced by the Lagrangian subspace
$\Psi(0)$ defines a pure, principally 
polarized Hodge structure of weight $-1$ on 
$W_{-1}/W_{-2}\simeq (\bZ^{2h},\cdot)$.
Here $h$ is the rank of the null space 
of a general element
 of $\bB_f$.
\end{theorem}

 \begin{definition}\label{def:max}
We say that $f\colon X\to \Delta^k$ 
is {\it maximally degenerate}
if $\bB_f\cap \cP_g\neq \emptyset$, i.e.~the general element
of $\bB_f$ is positive-definite. 
Equivalently,  
$W_{-2}=W_{-1}$, i.e.~$h=0$. 
\end{definition}

 \begin{definition}\label{ag-trop}
 We define the 
 {\it tropical moduli space of abelian varieties} to be
 $$(\cA_g)_{\rm trop}\coloneqq \GL_g(\bZ)\backslash \cP_g^+$$ 
 where the action is via $B\mapsto ABA^T$.
 \end{definition} 

We can view $(\cA_g)_{\rm trop}$ as the tropical
moduli space of abelian varieties, because, as we will see
in Section \ref{sec:toroidal-extensions}, a fan defining
a toroidal extension of $\cA_g$ lives naturally in
$(\cA_g)_{\rm trop}$. But more deeply, $(\cA_g)_{\rm trop}$
(or at least, the image of $\cP_g$ in it)
is itself a moduli space of ``tropical abelian varieties''
\cite{mz, chan, trop-torelli}: 
It parametrizes isometry classes of full rank
lattices in $\bR^g$, where the action of $\GL_g(\bZ)$
serves the role of forgetting the basis of the lattice
in which the Gram matrix of the corresponding real
intersection form has been expanded.

Then $\bB_f$ defines, canonically, an immersed 
polyhedral cone in $(\cA_g)_{\rm trop}$.

\begin{example}[Degenerations of Jacobians]\label{nodal}
Let $\pi \colon C\to \Delta^k$ be a family of nodal curves, 
which is smooth over the complement of the coordinate hyperplanes
$V(u_1\cdots u_k)$. 

Let $\{p_{ij}\}$ denote the nodes of the 
general fiber over $V(u_i)$.
The local equation of the smoothing 
of the node $p_{ij}$ is given
by $x_{ij}y_{ij}=u_i^{r_{ij}}$ for
some positive integer $r_{ij}$. 
It follows from
the Picard-Lefschetz formula that 
the logarithm of monodromy 
on $V_\bZ=H_1(C_t,\bZ)$ about the 
$i$-th coordinate hyperplane is 
given by \begin{align}\label{pic-lef} 
\textstyle N_i\colon x\mapsto -\sum_{j} 
r_{ij}(x\cdot \gamma_{ij})\gamma_{ij}
\end{align}
where $\gamma_{ij}\in H_1(C_t,\bZ)$ is 
(either orientation of)
the vanishing cycle of the node $p_{ij}$ and $\cdot$ 
is the intersection form on $V_\bZ$. 
Observe that the total space $C$ 
is smooth if and only if all $r_{ij}=1$ and no 
node of $C_0$ is the specialization of a node over 
both $V(u_i)$ and $V(u_j)$ for $i\neq j$. 
Taking the relative Jacobian fibration 
$$J\pi^\circ\colon JC^\circ
\to (\Delta^*)^k$$
of the smooth family, 
the monodromy is given by the same formula,
since $H_1(JC_u, \bZ)\simeq H_1(C_u,\bZ)$ 
for $u\in (\Delta^*)^k$.
The weight filtration is 
$$W_{-2}=\bZ\textrm{-span}\{\gamma_{ij}\} 
\textrm{ and } W_{-1}=(W_{-2})^\perp.$$
Computing the monodromy bilinear forms 
on ${\rm gr}^W_0 V_\bZ$ we get 
$$\textstyle B_i(x,y)\coloneqq  N_i(x)\cdot y = 
\sum_j r_{ij} (x\cdot \gamma_{ij})(y\cdot \gamma_{ij}).$$
Suppose now that $\pi\colon C\to \Delta^k\subset {\rm Def}(C_0)$ 
is a slice of the universal deformation of $C_0$
which is transversal to the 
equisingular/locally trivial deformations.
Then there is only one node over each $V(u_i)$, the corresponding
integer $r_i=1$, and $k$ is the number of nodes of $C_0$.
Thus, $$B_i(x,x) = (\gamma_i\cdot x)^2\in {\rm Sym}^2 
({\rm gr}^W_0 V_\bZ)^\vee$$
is a rank $1$ quadratic form,
given by the square of the linear form which is
pairing with the vanishing cycle
$\gamma_i\in {\rm gr}^W_{-2}V_\bZ$.

We have canonical isomorphisms
\begin{align*} {\rm gr}^W_0V_\bZ&\simeq H_1(\Gamma(C_0),\bZ), \\
{\rm gr}^W_{-2}V_\bZ&\simeq H^1(\Gamma(C_0),\bZ), \end{align*}
where $\Gamma(C_0)$
is the dual graph of the central fiber;
indeed, by duality it suffices to prove the first 
isomorphism, which follows e.g.\ from 
Proposition \ref{prop:gr-H_lim=gr-H-Gamma} below. 
The space
$W_{-2}V_\bZ= {\rm gr}_{-2}^WV_\bZ$ is generated 
by the vanishing cycles
$\gamma_i$ which are in bijection with the edges of $\Gamma(C_0)$.
The relations between the vanishing cycles $\gamma_i$ are given
by the boundaries of the subsurfaces they bound. In terms of graphs,
these are the coboundaries of the vertices of $\Gamma(C_0)$, 
so that $${\rm gr}_{-2}^WV_\bZ \simeq 
{\rm coker}(\cC^0(\Gamma(C_0),\bZ)\xrightarrow{\partial} 
\cC^1(\Gamma(C_0),\bZ)) \eqqcolon H^1(\Gamma(C_0),\bZ).$$
In turn, 
$H_1(\Gamma(C_0),\bZ)\simeq {\rm gr}^W_0V_\bZ$ 
with the quadratic form 
$B_i=(\gamma_i\cdot x)^2$
evaluating on a $1$-cycle 
$\sum c_ie_i\in H_1(\Gamma(C_0),\bZ)$ 
to the square $c_i^2$ of the coefficient 
of the edge $i$ in it.
See Figure \ref{fig:theta_curve}.

It follows from the Mayer--Vietoris sequence that 
${\rm gr}^W_{-1}V_\bZ \simeq H_1(C_0^\nu, \bZ)$ where
$C_0^\nu\to C_0$ is the normalization, with its natural
polarized $\bZ$-Hodge structure.

\begin{figure}[ht]
\includegraphics[height=2.7in]{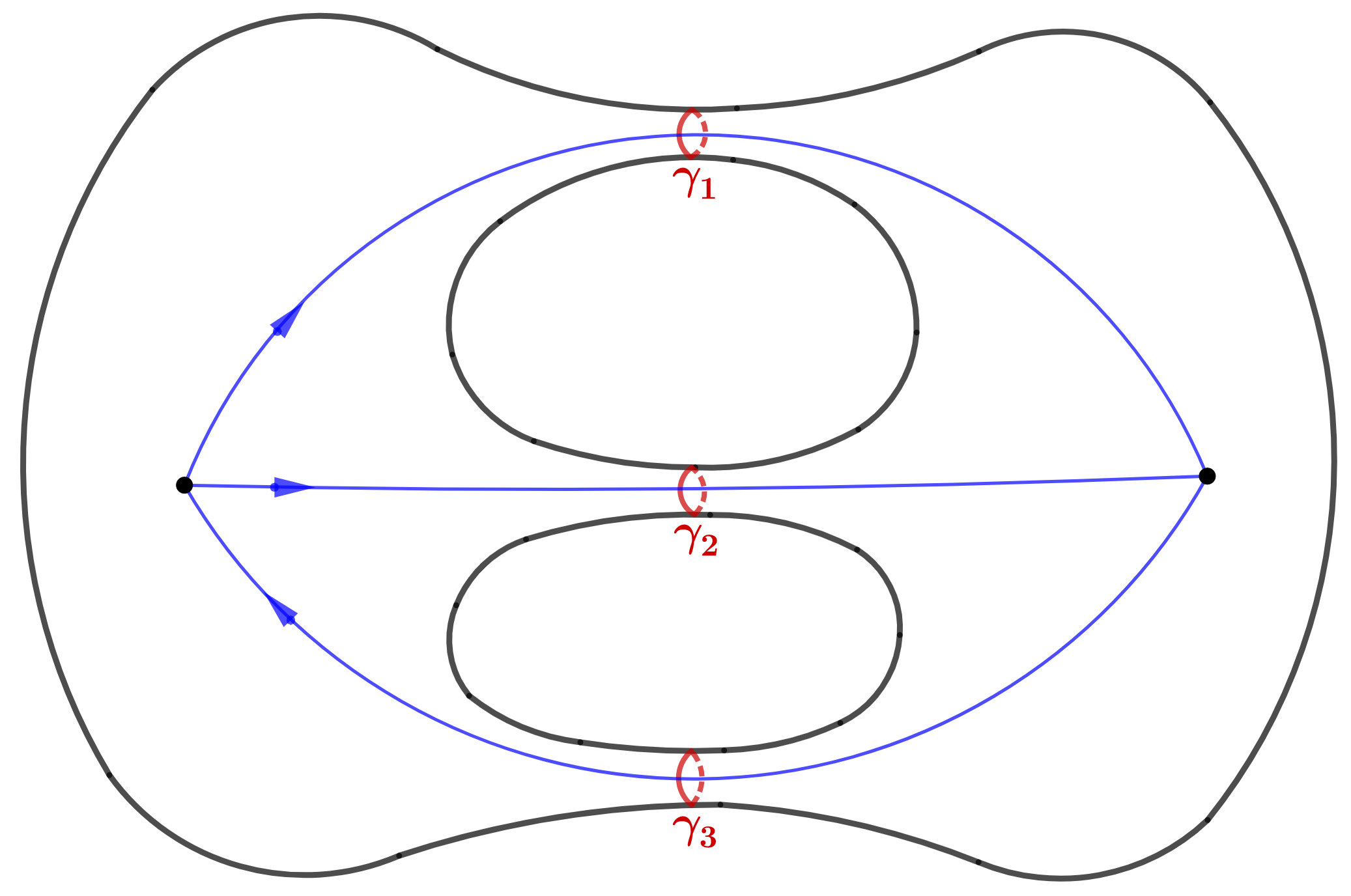}
\caption{Nearby fiber $C_t$ 
to the universal deformation
$\pi \colon C\to \Delta^3$ of a nodal genus $2$
curve $C_0$ with three nodes. 
Vanishing cycles 
$\{\gamma_1,\gamma_2,\gamma_3\}\in {\rm gr}^W_{-2}V_\bZ$
shown in red, and dual graph $\Gamma(C_0)$
of the central fiber, shown in blue.}
\label{fig:theta_curve}
\end{figure}
\end{example}

\begin{example}[Degenerations of 
intermediate Jacobians]\label{gwena-ex}
This example is due to Gwena \cite{gwena}.
Let $Y_0$ be the unique cubic threefold with $10$ isolated
$A_1$ singularities,
the {\it Segre cubic}. Concretely, it is 
defined by the equations
$$Y_0\coloneqq \textstyle \left\{\sum_{i=0}^5 x_i =
\sum_{i=0}^5x_i^3 =0\right\}\subset \bP^5$$
with the $10$ singularities given by the $S_6$-orbit
of the point $[1: 1:1:-1:-1:-1]\in \bP^5$. Then, the 
universal deformation (whose existence follows from 
\cite[Cor.~4.2]{friedman}, for example)
$$\pi \colon Y\to {\rm Def}_{Y_0}\simeq \Delta^{10}$$ of $Y_0$ 
is a degeneration of smooth cubic threefolds, whose discriminant
locus is the union of the $10$ coordinate hyperplanes
$V(u_1\cdots u_{10})\subset \Delta^{10}$. Similar to the universal
deformation of a nodal curve, if we label the nodes of $Y_0$
and coordinate hyperplanes appropriately, then 
$V(u_i)\simeq \Delta^9\subset \Delta^{10}$
is exactly the locus where the $i$-th node is not smoothed.

We now define the {\it intermediate Jacobian fibration} 
$$IJ\pi^\circ\colon IJY^\circ\to (\Delta^*)^{10}.$$ 
The rank $10$, polarized, unimodular $\bZ$-local system
$\bV_\bZ\coloneqq  (R^3\pi^\circ_*\bZ(1))^\vee$
underlies a variation of Hodge structure of 
weight $-1$ and type $(-1,0)$, $(0,-1)$.
It is polarized by the negation of the 
intersection form on $H_3(Y_u,\bZ)$ for $u \in (\Delta^\ast)^{10}$. 
Then, we define $IJY^\circ\coloneqq \bV_\bC/(\cF^0+\bV_\bZ)$.

Let $p_i\in Y_i$ be the unique node on the general fiber
of $Y_i\to V(u_i)$. Then associated to $p_i$ is the cycle of 
a vanishing $3$-sphere
$\gamma_i \in V_\bZ = H_3(Y_t, \bZ(-1))$. 
With an appropriate labeling 
and orientations, the $10$ cycles $\gamma_i$ 
satisfy the following collection
of linear relations
\begin{align}\label{relations}
\begin{split}
\gamma_6 &= \gamma_5-\gamma_1+\gamma_2 \\
\gamma_7 &= \gamma_1-\gamma_2+\gamma_3 \\
\gamma_8 &= \gamma_2-\gamma_3+\gamma_4 \\
\gamma_9 &= \gamma_3-\gamma_4+\gamma_5 \\
\gamma_{10} &= \gamma_4-\gamma_5+\gamma_1,
\end{split}
\end{align}
see \cite[Sec.~7.1.10]{gwena}, 
and generate a primitive integral
Lagrangian subspace
$W_{-2}\subset H_3(Y_t,\bZ(-1))$ of rank $5$. The Dehn
twist about the three-sphere $\gamma_i$ gives the formula
$$N_i\colon x\mapsto -(x\cdot \gamma_i)\gamma_i$$
for the logarithm of monodromy about the $i$-th coordinate
axis. By the same computation as Example \ref{nodal}, 
the monodromy bilinear forms on ${\rm gr}^W_0 V_\bZ$ 
satisfy $B_i(x,x)=(\gamma_i\cdot x)^2$. 
As we will see in Remark 
\ref{R10notcographic} in Section \ref{sec:matroid}, 
there is no graph $G$ for which
the $\gamma_i$ correspond to the edges 
of $G$ and such that the relations \eqref{relations} 
between the $\gamma_i$ are given by the image of the map 
$C^0(\Gamma(G),\bZ) \to C^1(\Gamma(G),\bZ)$ defined by 
some choice of orientation, cf.~Example \ref{nodal}.
\end{example}

We will give explicit extensions of
the families $J\pi^\circ$ and $IJ\pi^\circ$ 
over $\Delta^k$ and $\Delta^{10}$
in Examples \ref{cographic-ex}, \ref{r10-ex}
and in Construction \ref{shifted-matroidal-construction}.

\subsection{Toric varieties}\label{sec:toric-varieties}

We recall here some of the basic theory of toric
varieties. We refer to \cite{fulton, cox}
for the standard notions.

Toric geometry will be used both to extend
$\cA_g$ toroidally 
(Section \ref{sec:toroidal-extensions}), and
to build Mumford degenerations 
(Section \ref{sec:mumford}). 
We employ the standard 
toric notations of 
$\bfN\simeq \bZ^g$
for a free abelian group of rank $g$ and 
$\bfM\coloneqq {\rm Hom}(\bfN ,\bZ)$,
for constructions concerning the abelian and
degenerate abelian fibers. 
These lattices
play, respectively, the roles of the 
cocharacter and the character
lattices in toric geometry. Thus, fans lie
in $\bfN$ while polytopes lie in $\bfM$.

\begin{definition} A {\it fan} $\fF$ in $\bfN$
is a set of strongly convex, rational polyhedral
cones $\tau\subset \bfN_\bR$ for which every face
of a cone is a cone, and the intersection
of two cones is a face of each.
\end{definition}

We do not impose the hypothesis that fans
have finitely many cones, and indeed almost
none of the fans in this paper satisfy this hypothesis.
For each cone $\tau\in \fF$ in a fan, 
we may consider the $\bC$-algebra $\bC[\tau^\vee\cap \bfM]$ associated to the commutative semi-group $\tau^\vee\cap \bfM$,
where $\tau^\vee\subset \bfM_\bR$ is the collection
of all linear functionals evaluating non-negatively
on $\tau\subset \bfN_\bR$.
We form
the corresponding affine toric variety 
\begin{align} \label{affinetoric}Y(\tau)\coloneqq {\rm Spec}\, \bC[\tau^\vee\cap \bfM].\end{align}
Then the gluing 
$Y(\fF)\coloneqq \bigcup_{\tau\in \fF} Y(\tau)$ of these
affine schemes along the natural
open immersions corresponding face inclusions
gives the {\it toric variety} $Y(\fF)$, 
see \cite[Sec.~1.4]{fulton}.

\begin{notation} If $\fF$ is a polyhedral fan, 
we notate its toric variety by $Y(\fF)$.
\end{notation}

For all constructions in this paper, $Y(\fF)$ will be an
analytic space (and even a $\bC$-scheme)
which is locally of finite type.

As in the usual theory of toric varieties, 
the torus orbits of dimension $d$ in $Y(\fF)$, 
isomorphic necessarily to $(\bC^*)^d$, 
correspond bijectively
to cones of codimension $d$ in $\fF$, 
see \cite[Sec.~3.1]{fulton}.

\begin{definition} \label{polytope:1}A {\it polytope} $H$ 
is a convex set in $\bfM_\bR$ defined by the intersection
of a (possibly infinite) number of closed 
rational half-spaces,
such that $H$ is {\it locally of finite type}, 
i.e.~locally about every point $p\in H$, it is
defined by a finite number of half-spaces. 
  \end{definition}

  \begin{definition} \label{polytope:2}
 A {\it face}
 $F\subset H$ of a polytope
 is a non-empty intersection of $H$ with a 
 (possibly empty) collection of supporting 
 hyperplanes, and the {\it local monoid}
 $\bfM_F$ of this face
 is the intersection of the lattice $\bfM$ 
 with the finitely
 many (possibly empty) 
 closed half-spaces which define $H$ 
 in the neighborhood of
 a general point of $F$,
 translated to the origin.

 The {\it normal fan} of $H$ is the fan formed
 from the dual cones of the local monoids $\bfM_F$ ranging
 over all faces $F\subset H$ (including the open face $H$).
 \end{definition}

The toric variety $Y=Y_H$ associated to the polytope
$H$ is the union of ${\Spec}\,\bC[\bfM_F]$ ranging
over all faces, see also \cite[Sec.~1.5]{fulton}. 
If, furthermore, all faces
of $H$ are integral polytopes,
then there is a canonically defined
torus-equivariant holomorphic line bundle
$\cL$ on $Y$, given by gluing together line
bundles on each affine chart ${\Spec}\,\bC[\bfM_F]$,
in a manner which locally agrees with the recipe
in \cite[Sec.~3]{fulton}.

\begin{remark}\label{toricpolytope}
When the polytope $H$ is compact, we may
construct $Y=Y_H$ 
directly as the projective variety 
$Y_H = {\rm Proj}\,
\bC[{\rm Cone}(H)\cap (\bfM\times \bZ)]$,
where ${\rm Cone}(H)$ is the cone over $H$,
put at height $1$ in $\bfM_\bR\times \{1\}\subset \bfM_\bR\times \bR$; the above line bundle 
is $\cL = \cO_Y(1)$. 
The lattice points ${\bf m}\in H\cap \bfM$ 
define a basis of torus-equivariant
sections of $H^0(Y, \cO(1))$. 
More generally, lattice points 
$({\bf m},w) \in 
{\rm Cone}(H) \cap (\bfM \times \set{w})$ 
of height $w$ correspond to 
torus-equivariant sections 
$\theta_{{\bf m}/w} \in H^0(Y, \cO(w))$.
The multiplication map 
$$H^0(Y, \cO(w_1)) \otimes H^0(Y, \cO(w_2)) \to 
H^0(Y, \cO(w_1+w_2))$$ corresponds to 
$(({\bf m}, w_1), ({\bf m'}, w_2)) 
\mapsto ({\bf m+m'}, w_1+w_2)$, which
we may equivalently write as
a multiplication rule 
$\theta_{{\bf m}/w_1}\cdot \theta_{{\bf m}'/w_2}= \theta_{({\bf m}+{\bf m}')/(w_1+w_2)}$, 
cf.~(\ref{mult-rule}) below.
\end{remark}

Suppose
now that there is a subgroup
$A\subset \GL(\bfM)\ltimes \bfM$ of the 
integral-affine group (possibly infinite), 
acting on the polytope $H$. Then there is a natural
action of $A$ on $Y$, which in a 
basis $\bfM\simeq \bZ^g$ acts on the open
torus orbit $\bfN\otimes \bC^*\simeq (\bC^*)^g$
by the map $$c_i\mapsto c_1^{a_{i1}}\cdots c_g^{a_{ig}}$$
where $(a_{ij})_{1\leq i,j\leq g}\in \GL_g(\bZ)$ 
is the linear part of $A$ in the chosen basis.
Moreover, there is a linearization
of $\cL=\cO_Y(1)$ 
with respect to the $A$-action on $Y$,
which acts on sections by 
$a\cdot (z^m)=z^{a\cdot m}$.
The $A$-action is properly discontinuous on a tubular
neighborhood of all toric strata corresponding
to faces $F$, or cones $\tau\in \fF$,
whose $A$-stabilizer has finite order.

\begin{definition} A fan $\fF$ is {\it regular}
if every cone $\tau\in \fF$
is {\it standard affine}; that is,
the primitive integral
generators of the extremal
rays of $\tau$ form a subset
of a basis of $\bfN$.\end{definition}

Equivalently, if $\fF$ is a normal fan,
the polytope $H$ should be {\it Delzant}. 
The toric variety
$Y_H=Y(\fF)$ is a smooth analytic space if and only
if $\fF$ is regular, in which case for each 
$\tau \in \fF$, the affine toric variety $Y(\tau)$, 
see \eqref{affinetoric}, is isomorphic to a product of an 
affine space with a torus. 

\begin{definition} The {\it support} (in $\bfN_\bR$) 
of a fan $\fF$ is the union of all 
cones $\tau\in \fF$. 
\end{definition}

Finally, a {\it morphism} of fans $\fF\to \fG$ 
is a linear map between the corresponding cocharacter
lattices, which sends cones into cones.
It induces a torus-equivariant map
$Y(\fF)\to Y(\fG)$.

\subsection{Toroidal extensions}\label{sec:toroidal-extensions}
We now outline the 
construction of toroidal extensions of $\cA_g$ associated
to a monodromy cone $\bB_f$ and more generally a fan $\fF$.
See \cite{namikawabook, mumford-amrt, mumford1975new, fc} 
for references on toroidal compactifications of 
Siegel spaces. 

\begin{definition}
\label{def:fan} A {\it fan} $\fF$ for $\cA_g$ 
is a rational polyhedral fan, whose support
is contained in 
$\cP_g^+\subset {\rm Sym}_{g\times g}(\bR)$, and
such that $\fF$ is $\GL_g(\bZ)$-invariant
under the action $B\mapsto ABA^T$,
with finitely many orbits of cones.
\end{definition}

\begin{example}\label{embed}
Let $f\colon X\to \Delta^k$ 
be an abelian fibration with unipotent 
monodromies about the coordinate hyperplanes.
Then $\fF_f\coloneqq \GL_g(\bZ)\cdot \bB_f$ defines a fan,
when $\bB_f$ injects 
into $(\cA_{g})_{\rm trop}=\GL_g(\bZ)\backslash \cP_g^+$.
\end{example}

Let $\bB$ be a polyhedral cone which embeds
into $(\cA_g)_{\rm trop}$. In what follows, we will 
consider the fan $\fF=\GL_g(\bZ)\cdot \bB$.
Consider the coordinate-wise exponential mapping:
\begin{align*} E\colon {\rm Sym}_{g\times g}(\bC)& \to 
 {\rm Sym}_{g\times g}(\bC^*) \\ (\sigma_{ij})_{i,j=1}^g&\mapsto (\exp(2\pi i 
 \sigma_{ij}))_{i,j=1}^g. \end{align*}
The map $E$ is the quotient by the action
of translation by $U_\bZ\simeq 
{\rm Sym}_{g\times g}(\bZ)$, so $E$ 
defines an open embedding of the quotient
of Siegel space (Def.~\ref{def:siegel}) 
into a torus 
\begin{align}\label{torus-embedding}
\overline{E}\colon U_\bZ 
\backslash \cH_g\hookrightarrow 
U_\bZ \backslash {\rm Sym}_{g \times g}(\bC) 
\simeq \Sym_{g \times g}(\bC^\ast) \simeq  
{\rm Sym}_{g\times g}(\bZ)\otimes 
\bC^*\simeq U_\bZ\otimes \bC^*.\end{align}
Here $U_\bZ\subset P_\bZ$
is the unipotent radical of the parabolic,
as in Definition \ref{unipotent-parabolic}, 
and the isomorphism 
$\Sym_{g \times g}(\bZ) \otimes \bC^\ast 
\simeq \Sym_{g \times g}(\bC^\ast)$ is given by
$(n_{ij})_{1\leq i,j\leq g}\otimes \lambda = 
(\lambda^{n_{ij}})_{1\leq i,j\leq g}$.
This is called the 
{\it first} or {\it unipotent partial quotient} 
of $\cH_g$ in the theory of 
toroidal compactifications. \smallskip

Compose the period maps 
$\widetilde{\Phi}$ and $\widetilde{\Phi}_{\rm nilp}$ 
of any degeneration with $\bB_f=\bB$
with the quotient map $\cH_g \to 
\Sym_{g \times g}(\bZ) \backslash \cH_g$.  
They descend to single-valued maps 
$\Phi, \Phi_{\rm{nilp}} \colon (\Delta^*)^k 
\to \Sym_{g \times g}(\bZ) \backslash \cH_g.$ 
Composing $\Phi_{\rm nilp}$ with the map 
$\overline{E}$ gives rise to a map 
\begin{align}\label{nilpotent}
(\Delta^*)^k\to {\rm Sym}_{g\times g}(\bZ)\otimes \bC^* \simeq \Sym_{g \times g}(\bC^\ast) \end{align} 
whose image is (an analytic open subset of) 
a translate of the
subtorus $\langle \bB\rangle \otimes \bC^*$
where $\langle \bB\rangle\coloneqq (\bR\bB)
\cap {\rm Sym}_{g\times g}(\bZ)$. 
Thus, Theorem \ref{schmid} can be rephrased as 
saying that the period mapping
is approximated by a translate of a subtorus, 
near $0\in \Delta^k$. 

Since the fan $\fF=\GL_g(\bZ)\cdot \bB$ sits in the 
co-character lattice ${\rm Sym}_{g\times g}(\bZ)$,
the associated toric variety $Y(\fF)$ is a toroidal 
extension of ${\rm Sym}_{g\times g}(\bC^*)$.
Consider the quotient $\GL_g(\bZ)\backslash Y(\fF)$.
This quotient is not globally well-behaved, e.g.~the action
of $\GL_g(\bZ)$ fixes the origin point of the open torus orbit.
But there is a tubular analytic neighborhood 
$T(\fF)\subset Y(\fF)$ of the union of the toric 
boundary strata of $Y(\fF)$
corresponding to cones intersecting $\cP_g$ on which
the $\GL_g(\bZ)$-action is properly discontinuous.
Let $T^*(\fF)$ be the intersection of $T(\fF)$ with
the open torus orbit $U_\bZ\otimes \bC^*$.
Then, we have open embeddings
\begin{align}
    \label{embedding}
\GL_g(\bZ)\backslash T(\fF)
\hookleftarrow \GL_g(\bZ)\backslash T^*(\fF) 
\hookrightarrow
P_\bZ \backslash \cH_g.
\end{align} 
By \cite[Thm.~4.9(iv)]{BB66}, 
the boundary of $\Sp_{2g}(\bZ)\backslash \cH_g$ is locally 
modeled near the Baily-Borel cusp (associated to the 
Lagrangian subspace $\bZ e_1\oplus \cdots 
\oplus \bZ e_g$) by $P_\bZ\backslash \cH_g$. 
Thus, by (\ref{embedding}), we may glue
$\GL_g(\bZ)\backslash T(\fF)$ to 
$\cA_g = \Sp_{2g}(\bZ)\backslash \cH_g$ along their
common analytic open subset $\GL_g(\bZ)\backslash T^*(\fF)$.

More generally, the same construction applies
to any $\GL_g(\bZ)$-invariant fan $\fF$, 
and the resulting toroidal extension
is relatively proper over the $0$-dimensional
cusp of the Baily-Borel \cite{BB66}
compactification $\overline{\cA}_g^{\rm BB}=
\cA_g\sqcup \cA_{g-1}\sqcup \cdots 
\sqcup \cA_1\sqcup \cA_0$
if and only if ${\rm Supp}(\fF) = \cP_g^+$. 

When $\fF$ contains cones 
supported in $\cP_g^+\setminus \cP_g$, 
the gluing defined by (\ref{embedding})
further extends along the 
intermediate-dimensional strata of the 
Baily-Borel compactification.
This is the {\it toroidal extension}
$\cA_g\hookrightarrow \cA_g^{\fF}$
of the orbifold $\cA_g$.

When ${\rm Supp}(\fF) = \cP_g^+$, we notate
the toroidal extension by 
$\cA_g\hookrightarrow \overline{\cA}_g^\fF$;
it is proper, and we call it a
{\it toroidal compactification}.
It follows
from \cite[Thm.~5.2]{mumford-amrt} 
that $\cA_g^\fF$ or $\overline{\cA}_g^\fF$ 
is a DM algebraic stack,
in the former case
by refining and extending $\fF$ to a 
fan with full support $\cP_g^+$.

\begin{notation} For simplicity,
we will write $\cA_g\hookrightarrow
\cA_g^\bB$ to notate the toroidal
extension associated to the fan 
$\fF=\GL_g(\bZ)\cdot \bB$ consisting
of the orbit of a polyhedral cone $\bB\subset \cP_g^+$.\end{notation}

\begin{proposition}\label{monodromy-extend}
    For any degeneration of PPAVs $f^\ast
    \colon X^\ast \to (\Delta^*)^k$
    with monodromy cone $\bB$, the period map
    $(\Delta^*)^k\to \cA_g$ admits a unique
    holomorphic extension 
    $\Delta^k\to \cA_g^{\bB}$.
\end{proposition}

\begin{proof} 
The proposition follows from Theorem \ref{schmid}:~We have shown
that the nilpotent orbit $\Phi_{\rm nilp}$ maps,
analytically-locally near the boundary of $\cA_g$,
into the translate of 
the subtorus $\langle \bB\rangle\otimes \bC^*\subset 
{\rm Sym}_{g\times g}(\bC^*)$, see \eqref{nilpotent}.  
Any cocharacter
$B\otimes (\bC\setminus 0)
\subset \langle \bB\rangle \otimes \bC^*$,
for $B\in \bB$, admits a completion over $0\in \bC$ 
to the toroidal extension 
$\cA_g^{\bB}$, sending $0$
into the toroidal boundary stratum corresponding to the 
cone of $\fF$ containing $B$ in its relative interior.
We deduce that $\Phi_{\rm nilp}\colon (\Delta^*)^k\to \cA_g$ 
extends holomorphically to a map 
$\Delta^k\to \cA_g^{\bB}$.

Next, it follows
from the exponential convergence (\ref{metric-decay}) 
of $\Phi$ towards $\Phi_{\rm nilp}$ (in
the invariant metric on $\cH_g$) that
$\Phi\colon (\Delta^*)^k\to \cA_g$ 
admits a continuous extension $\Delta^k\to 
\cA_g^{\bB}$. Since $\Delta^k$
is normal, Riemann's removable singularities theorem
implies that this continuous extension is holomorphic.
It is unique because the toroidal extension is a separated
analytic stack.
\end{proof}

\begin{remark}\label{monodromy-fit} 
In general, one may wish to consider
monodromy cones $\bB_f$ for which the 
$\bB_f$ does not
embed into $(\cA_g)_{\rm trop}$ (Def.~\ref{ag-trop}).
The issue here is that
the $\GL_g(\bZ)$-orbit of such a cone may intersect itself.
This problem is resolved by rather quotienting 
 $U_\bZ\backslash \cH_g 
\subset U_\bZ\otimes \bC^*$
by a finite index subgroup $\Gamma\subset \GL_g(\bZ)$.

For instance, consider the $n$-torsion subgroup
$X[n]\simeq (\bZ/n\bZ)^{2g}$ of a PPAV $X$. 
The principal polarization
defines a natural non-degenerate symplectic pairing 
on $X[n]$. We define the moduli space
of abelian varieties $X$ with 
{\it full Lagrangian level $n$ structure}
by adding the data of:
\begin{enumerate}
    \item a Lagrangian subspace 
    $\Xi\simeq (\bZ/n\bZ)^g\subset X[n]$ and
    \item a $\bZ/n\bZ$-basis of $\Xi$.
\end{enumerate}

All degenerations we consider in this paper admit
some full Lagrangian level $n$ structure, because
there is a distinguished Lagrangian subspace 
${\rm gr}^W_{-2}H_1(X,\bZ)$ on which the monodromy
acts trivially. The moduli stack $\cA_g[n]$ 
of abelian varieties
with full Lagrangian level $n$ structure is
an \'etale cover $\widetilde{\cA}_g=\cA_g[n]\to \cA_g$. 

At an appropriate $0$-cusp of the Baily--Borel compactification
of $\widetilde{\cA}_g$ (the cusp corresponding to a 
Lagrangian subspace
$\widetilde{\Xi}\subset H_1(X,\bZ)$ for which 
$\tfrac{1}{n}\widetilde{\Xi}/\widetilde{\Xi}= \Xi$), 
the parabolic stabilizer $P_\bZ$ has the following structure: 
$$0\to U_\bZ\to P_\bZ \to \Gamma(n)\to 0,$$
where the unipotent subgroup
$U_\bZ\simeq {\rm Sym}_{g\times g}(\bZ)$
is the same as without level structure, but 
the Levi quotient 
$\Gamma(n)\coloneqq \ker(\GL_g(\bZ)\to \GL_g(\bZ/n\bZ))$ 
is the full level $n$ subgroup. 

Then, a toroidal extension of $\widetilde{\cA}_g$
at this Baily--Borel cusp has the ``advantage'' that
a fan need only be $\Gamma(n)$-invariant. Furthermore,
the fundamental domain for the action of $\Gamma(n)$
is larger than the fundamental domain for that action
of $\GL_g(\bZ)$. In particular, given any polyhedral
cone $\bB\subset \cP_g^+$, there exists some $n$
(depending on $\bB$) for which $\bB$ embeds into the
quotient $\Gamma(n)\backslash \cP_g^+$.
The preceding results also apply at this cusp, since
$\Gamma(n)\cdot \bB$ now defines a fan.
\end{remark}

\begin{proposition} \label{monodromyrealization}
Let $\tau\subset \bR^k$ be a strictly
convex, rational polyhedral cone, and 
consider any homomorphism
$\phi\colon 
\bZ^k\to {\rm Sym}_{g\times g}(\bZ)$ 
for which 
$\bB\coloneq \phi_\bR(\tau)\subset \cP_g^+$.
There exists a quasiprojective
variety $Y$, divisor $D\subset Y$,
point $0\in D$, and projective
family $f^*\colon X^*\to Y^*$, 
$Y^*=Y\setminus D$, 
of PPAVs of dimension $g$, in
the algebraic category, whose
monodromy cone at $0$ is given by $\phi$
in the following sense:
\begin{enumerate}
    \item\label{smooth-star} 
    $Y$ admits, near $0\in Y$, an \'etale-local 
    isomorphism to the toric variety $Y(\tau)$,
    sending $D$ to the toric boundary, and $0$ 
    to the torus fixed point.
    \item the monodromy representation of $Y\setminus D$ near $0$ is given, under this isomorphism, by $$ \pi_1(\bZ^k\otimes \bC^*, *)\simeq \bZ^k \xrightarrow{\phi} {\rm Sym}_{g\times g}(\bZ)=U_\bZ\subset {\rm Sp}_{2g}(\bZ).$$  Here $\bZ^k\otimes \bC^*\subset Y(\tau)$ is
    the open torus orbit.
\end{enumerate}
\end{proposition}

Observe that it follows 
from (\ref{smooth-star}) that $Y^*$
is smooth near $0$.  

\begin{proof}
We first prove the proposition
under the hypothesis that $\phi$ is
injective.

Consider the \'etale cover 
$\cA_g[n]\to \cA_g$
given by full Lagrangian level $n\geq 3$
structure as in Remark \ref{monodromy-fit}.
As noted, there is a Baily--Borel cusp
of $\cA_g[n]$ whose parabolic
stabilizer $P_\bZ$ has unipotent subgroup
$U_\bZ\subset {\rm Sym}_{g\times g}(\bZ)$ 
which is the same as for $\cA_g$ but whose
Levi quotient is 
$P_\bZ/U_\bZ=\Gamma(n)\subset \GL_g(\bZ)$.
Since $\Gamma(n)$ is neat for $n\geq 3$, 
its action on the toroidal extension 
$\cA_g[n]^\fF$ is free
in the tubular neighborhood
$T(\fF)$ of (\ref{embedding}). 
We deduce that there is a Zariski open
$V\subset \cA_g[n]^\fF$, 
containing
all maximally degenerate strata, 
over which
the universal family
$\cX_g\vert_{V^*}\to {V^*}$ of PPAVs, 
$V^*\coloneqq V\cap \cA_g[n]$,
exists as a scheme (rather than just
a DM stack). 

Then, choose $n\geq 3$ so that
$\bB$ lies in the strict interior of a 
fundamental domain for the action of
$\Gamma(n)$ on $\cP_g^+$ and define
$\fF=\Gamma(n)\cdot \bB$.
By the 
construction of toroidal extensions,
the open set
$V$ is \'etale-locally isomorphic
to the toric variety 
$Y_{{\rm Sym}_{g\times g}(\bR)}(\bB)$.

Choose
a finite index sublattice 
$\Lambda\subset {\rm Sym}_{g\times g}(\bZ)$
for which $\Lambda\cap \bB={\rm im}(\phi)$. Then the finite \'etale cover
$\Lambda\otimes \bC^*\to 
{\rm Sym}_{g\times g}(\bZ)\otimes \bC^*$ 
of algebraic tori
induces a branched cover of toric
varieties $Y_{\Lambda\otimes \bR}(\tau)\to Y_{{\rm Sym}_{g\times g}(\bR)}(\bB)$.
Take an algebraic branched
cover $V'\to V$ with the same
branching over the toric boundary,
under the \'etale-local identification
of $\cA_g[n]^\fF$
with $Y_{{\rm Sym}_{g\times g}(\bR)}(\bB)$.
Then, $V'$ is \'etale-locally 
isomorphic to $Y_{\Lambda\otimes \bR}(\tau)$,
over the deepest toroidal stratum.

Finally, to construct $Y$,
we slice $V'$ by 
${g+1\choose 2}- \dim \tau$ 
generic hyperplanes, which under 
the \'etale-local identification of $V'$
with $Y_{\Lambda\otimes \bR}(\tau)$ are
transversal to the deepest toroidal
boundary stratum. We set $Y^*$
to be the inverse image of $V^*$,
$D=Y\setminus Y^*$, and $0\in D$
as an intersection point 
of the hyperplanes with the deepest
stratum. We set the family
of abelian varieties 
$f^*\colon X^*\to Y^*$ to be
the pullback of the universal family 
$\cX_g\vert_{V^*}\to V^*$ (which exists
by the above discussion)
along 
the map $Y^*\to V^*$.

To verify that the monodromy
representation of $f^*$ is as specified,
consider a $1$-parameter 
family $\Delta^*\to Y^*$ 
which extends to a map 
$\Delta\to Y$ sending $0\mapsto 0$. 
The monodromy over $\Delta^*$
can be computed in the \'etale-local
model as the monodromy of the family
of PPAVs over a co-character.
Such co-characters correspond
to lattice points, in $\tau\cap \bZ^k$.
Since $f^*\colon X^*\to Y^*$ 
is pulled back from $V^*$, 
the monodromy representation is pulled back along
the morphism of cocharacter
lattices, given by the inclusion $\Lambda\hookrightarrow 
{\rm Sym}_{g\times g}(\bZ)$ which, in particular, restricts
$\phi \colon \bZ^k\to {\rm Sym}_{g\times g}(\bZ)$. 
Finally, it suffices
to observe that the monodromy over 
the cocharacter $B\otimes \bC^*\subset {\rm Sym}_{g\times g}(\bZ)\otimes \bC^*$ is canonically identified
with $B\in \cP_g^+\cap 
{\rm Sym}_{g\times g}(\bZ)$,
see e.g.~(\ref{mono-eq}).

Finally, we address the case
where $\phi$ is not injective.
Define
$\overline{\tau}:=\phi_\bR(\tau)\subset 
\cP_g^+$, viewed as a polyhedral cone
in the lattice $\phi(\bZ^k)$.
We have a descended, injective homomorphism
$\overline{\phi}\colon 
\phi(\bZ^k) \hookrightarrow 
{\rm Sym}_{g\times g}(\bZ)$
satisfying the hypotheses of 
the proposition. Thus, there is a 
family $\overline{f}^*\colon \oX^*\to \oY^*$
over a base $\oY^*=\oY\setminus \oD$,
a point $\overline{0}\in \oD$,
and an \'etale-local isomorphism
$\alpha \colon 
\oY\to Y(\overline{\tau})$ near 
$\overline{0}$ satisfying
the conclusion of the proposition.
Then the base change of the morphism
$\alpha\circ \overline{f}^*\colon 
\oX^*\to Y(\overline{\tau})$ along
the morphism of
toric varieties $Y(\tau)\to Y(\overline{\tau})$
produces the desired family
$f^*\colon X^*\to Y^*:=
\oY^*\times_{Y(\overline{\tau})}
Y(\tau) \subset \oY\times_{Y(\overline{\tau})}
Y(\tau)\eqqcolon Y$. 
Here we take $0\in Y$ to be the fiber
product of the points $\overline{0}\in \oY$
and the torus fixed point of $Y(\tau)$.
Since the desired monodromy
map factors through $\overline{\phi}$,
the condition on monodromy follows.
\end{proof}

\begin{corollary} \label{corollary:monodromyrealization}
Let $B_i\in \cP_g^+\cap 
{\rm Sym}_{g\times g}(\bZ)$ 
for $i=1,\dots,k$.
There is a smooth quasiprojective
variety $Y$, an snc divisor
$D\subset Y$, a zero-stratum 
$0\in D$, and 
a projective algebraic family
$f^*\colon X^*\to Y^*$ of PPAVs,
such that the local monodromy 
bilinear form (Def.~\ref{def:Bi})
about the component $D_i\ni 0$ is 
$B_i$ for all $i=1,\dots,k$. 
\end{corollary}

Note that $B_i$ need not be primitive.

\begin{proof}
The corollary follows
from Proposition \ref{monodromyrealization},
by taking 
$\tau:=\bR_{\geq 0}^k\subset \bR^k$, 
and the homomorphism
$\phi\colon \bZ^k\to 
{\rm Sym}_{g\times g}(\bZ)$ 
sending $e_i\mapsto B_i$. 
Here one can assume that $Y$ is smooth 
and $D$ is snc because
they are \'etale-locally isomorphic
to $Y(\tau) = \bC^k$ and the union
of the coordinate hyperplanes, respectively.
\end{proof}

\begin{remark}
    Given any two families $X^*_1$, $X^*_2$ of 
    PPAVs over $(\Delta^*)^k$, with the same integral monodromies
    about each coordinate axis, there is an analytic deformation 
    $\cX^*\to (\Delta^*)^k\times Z$ of such families,
    over a connected base $Z$, and points $1,2\in Z$ for which
    $\cX^*_i\simeq X^*_i$ for $i=1,2$. Indeed, we may
    first deform each $X_i^*$ to the nilpotent orbit
    (\ref{nilpotent}) passing
    through the same point of the deepest toroidal stratum
    of $\widetilde{\cA}_g^{\,\bB}$, 
    and then relate the two 
    translates of subtori $\langle \bB\rangle \otimes \bC^*$ 
    by a translation in ${\rm Sym}_{g\times g}(\bC^*)$.
\end{remark}

\section{The Mumford construction}\label{sec:mumford}

In Section \ref{sec:toroidal-extensions}, we have shown
how to extend $\cA_g$ (or  $\widetilde{\cA}_g)$ toroidally, to obtain $\cA_g\hookrightarrow \cA_g^{\,\bB}$ (or $\widetilde{\cA}_g\hookrightarrow\widetilde{\cA}_g^{\,\bB}$),
so that
the period map $(\Delta^*)^k\to \cA_g$ 
of any degeneration $f\colon X\to \Delta^k$
with monodromy cone $\bB_f=\bB$ admits an extension
of the period map
over the punctures $\Delta^k\to \cA_g^{\,\bB}$.
We will now describe how to extend the
universal family of PPAVs over $\cA_g$ so that we may
pull back this extension to produce particularly
nice birational models of degenerations.
It is useful to have examples in mind;
many are provided in Section \ref{sec:examples}. 
All degenerations we consider 
in this section are maximal, in the sense
of Definition \ref{def:max}.

\subsection{Mumford construction, fan version}\label{sec:fan}
Let  $B\in {\rm Sym}^2\bfM^\vee$ be a 
positive-definite, symmetric, integral,
bilinear form on a lattice $\bfM\simeq \bZ^g$.
Then $B$ defines a homomorphism 
\begin{align*} N\colon \bfM &\to \bfM^\vee=\bfN, \\
{\bf m}&\mapsto B({\bf m},-).\end{align*}

Define $\Lambda_B\subset \bfN$ to be ${\rm im}(N)$.
In terms of symmetric $g\times g$ matrices, $\Lambda_B$
is the span of the rows (or columns) of $B$, and so
defines a finite index sublattice of $\bfN\simeq \bZ^g$.

\begin{definition}
$X_{\rm trop}(B) \coloneqq  \bfN_\bR/\Lambda_B$ 
is the {\it tropical abelian variety} 
associated to $B$.
\end{definition}

\begin{definition}\label{tiling}
A (resp.~$\bQ$-){\it tiling} 
of $X_{\rm trop}(B)=\bfN_\bR/\Lambda_B$
is a decomposition into convex polytopes with 
integer (resp.~rational) vertices,
or equivalently, a $\Lambda_B$-periodic 
(resp.~$\bQ$-)polytopal tesselation of $\bfN_\bR$.
A {\it complete triangulation} of $X_{\rm trop}(B)$ 
is a tiling, all of
whose polytopes are lattice simplices 
of minimal volume $(1/g!)$. 
\end{definition}

\begin{construction}[$1$-parameter case]\label{mumford} 
Let $B\in {\rm Sym}^2 \bfM^\vee$ be 
positive-definite and let 
$\cT$ be a $\bQ$-tiling of $X_{\rm trop}(B)$.
We define the $1$-parameter
Mumford degeneration associated to $\cT$. 

Embed $\bfN_\bR\simeq  \bfN_\bR\times \{1\}
\hookrightarrow \bfN_\bR \times \bR$ 
as an affine hyperplane at height $1$ 
in a space of one dimension higher. 
Then, the cone over the tiling 
${\rm Cone}(\cT)$ defines a rational 
polyhedral fan in $\bfN_\bR \times \bR \simeq \bR^{g+1}$. 
See Figure \ref{fig:tate6}
for an example of a tiling
$\cT$ of $\bR^2/\Lambda_B$, 
where $$B=\twobytwo{4}{1}{1}{3}.$$ 

Let 
$Y({\rm Cone}(\cT))$ denote the corresponding
infinite type toric variety. The action of 
$\Lambda_B$ by translations on $\cT$ lifts to a 
linear action $\Lambda_B\hookrightarrow \GL_{g+1}(\bZ)$
on $\bfN_\bR\times \bR$ 
acting on the fan ${\rm Cone}(\cT)$, 
and hence induces an action of $\Lambda_B$ 
on $Y({\rm Cone}(\cT))$ by automorphisms.

Observe that the height function $\bfN_\bR \times \bR
\to \bR$, given by
projecting to the final coordinate,
defines a morphism of fans 
${\rm Cone}(\cT)\to \bR_{\geq 0}$ to the fan of $\bC$
(which is simply the positive ray in $\bR$). 
Hence, there is an induced
map of toric varieties 
$$Y({\rm Cone}(\cT))\to Y(\bR_{\geq 0})=\bC.$$ 
Since the action of $\Lambda_B$ preserves the height function, 
this morphism descends to the quotient 
$\Lambda_B\backslash Y({\rm Cone}(\cT))\to \bC$, though
this full quotient is poorly behaved. Let $u$
be the monomial coordinate about $0\in \bC$. 
By standard toric geometry, we have:
\begin{enumerate}
    \item The fiber of $Y({\rm Cone}(\cT))\to \bC$ over 
    $u\in \bC^*=\bC\setminus \{0\}$ is 
    $\bfN\otimes \bC^*\simeq (\bC^*)^g$, and the fiber
    over $u=0\in \bC$ (i.e.~the toric boundary) 
    is an infinite  quilt of complete toric varieties,
    whose dual complex is the original tiling $\cT$. 
    \smallskip
    \item $\Lambda_B$  acts on the dual complex $\cT$
    of the toric boundary by translation, and 
    acts on the fiber over $u\in \bC^*$ by translations by 
    the rank $g$ subgroup 
    $u^B\coloneqq \Lambda_B\otimes u\subset \bfN\otimes \bC^*$ 
    which in coordinates is the subgroup
$$\langle (u^{B_{11}},\dots, u^{B_{1g}}), \,\dots,\,  (u^{B_{g1}},\dots, u^{B_{gg}})\rangle\subset (\bC^*)^g,
\quad B = (B_{ij})\in {\rm Sym}_{g\times g}(\bZ).$$
\end{enumerate}

Then, since $B$ is positive-definite, 
$u^B\subset \bfN\otimes \bC^*$ is a discrete
subgroup for all $u\in \Delta^*$.
Hence the action of $\bfM\simeq 
\Lambda_B$ is properly 
discontinuous over the unit disk 
$\Delta$, because
the action is clearly properly discontinuous
over $u=0$ since $\bfM$ acts freely on the tiles of $\cT$. 
Setting $X({\rm Cone}(\cT))\coloneqq \Lambda_B\backslash 
Y({\rm Cone}(\cT))^{u\in \Delta}$, we get
a proper, complex-analytic degeneration 
$$f\colon X({\rm Cone}(\cT))\to \Delta$$
of complex tori, called the {\it standard $1$-parameter
Mumford degeneration} associated to the $\bQ$-tiling
$\cT$ of $X_{\rm trop}(B)=\bfN_\bR/\Lambda_B$.

The general fiber is principally polarized,
by a symplectic form $L$ defined as follows: 
Let $u \in \Delta^\ast$. 
Then, noting that we have a canonical isomorphism
$H_1(\bfN\otimes \bC^*,\bZ)\simeq \bfN$, we 
have a canonical exact sequence
\begin{align}\label{weight-exact} 
0\to \bfN\to H_1(X_u,\bZ)\xrightarrow{\sigma} 
\bfM\to 0\end{align}
induced by the long exact sequence of homotopy 
groups associated to the fibration 
$$u^B \hookrightarrow \bfN \otimes \bC^\ast 
\to \bfN \otimes \bC^\ast/u^B = X_u.$$ Here we use that we 
have canonical isomorphisms 
$u^B \simeq \Lambda_B \simeq \bfM$.
Then, we may define a unimodular
symplectic form on $H_1(X_u,\bZ)$ by
choosing a splitting of $\sigma$ and then
using the canonical pairing between
$\bfM$ and $\bfN$. The splitting of $\sigma$
we choose is specified by a 
choice of logarithm $2 \pi i \tau = \log u$.  
Such a choice
gives 
a presentation 
$$X_u
=
\bfN_\bC/(\bfN\oplus \tau N(\bfM))
\simeq 
\bC^g/(\bZ^g\oplus \tau B ( \bZ^g))
;$$
the resulting
symplectic form $L$ is then independent 
of the choice of logarithm $\tau$ 
because the transformation 
$\tau\mapsto \tau+1$
defines a symplectomorphism of 
$H_1(X_u,\bZ)$. Concretely, the symplectic form $L$
is given as $L(({\bf n},{\bf m}), ({\bf n}',{\bf m}'))= 
{\bf m}'({\bf n})-{\bf m}({\bf n}')$, for 
$({\bf n},{\bf m}), ({\bf n}', {\bf m}') 
\in \bfM \oplus \bf N$. 

More generally, given a symmetric $g\times g$ matrix
$a=(a_{ij})\in {\rm Sym}_{g\times g}(\bC^*)$, we may 
perform the same construction, but instead quotient
by the subgroup
$$au^B\coloneqq \langle (a_{11}u^{B_{11}},\dots, a_{1g}u^{B_{1g}}), 
\,\dots,\, 
(a_{g1}u^{B_{g1}},\dots, a_{gg}u^{B_{gg}})\rangle.$$
We must take $|u|$ small enough that
$-B\log |u|-\log|a_{ij}|>0$.
Then this construction may be performed relatively,
choosing $a$ in
some subtorus of ${\rm Sym}_{g\times g}(\bC^*)$ 
forming coset representatives 
for the natural action of $u\in \bC^*$. 
We get an analytic degeneration
$$f^{\rm univ}_\circ \colon X^{\rm univ}_\circ ({\rm Cone}(\cT))
\to \Delta^{\rm univ}$$ 
where $\Delta^{\rm univ} \to 
{\rm Sym}_{g\times g}(\bC^*)/\bC^*$ is a holomorphic disk
bundle over ${\rm Sym}_{g\times g}(\bC^*)/\bC^*$, 
such that $f^{\rm univ}_\circ$ 
restricts to a family of PPAVs over the punctured 
disc bundle 
$(\Delta^\ast)^{\rm{univ}}
\subset \Delta^{\rm{univ}}$. 

In terms of the toroidal
extensions of Section \ref{sec:toroidal-extensions},
$\Delta^{\rm univ}$
maps to the toroidal extension $\cA_g^{\fF}$ 
whose fan is (the orbit of) a single
ray $\fF=\GL_g(\bZ)\cdot \bR_{\geq 0}B$. 
So every $1$-parameter
family
of PPAVs with monodromy
$B$ admits an extension pulled back from 
$f^{\rm univ}_\circ$
along an arc transverse to the boundary divisor
$\{0\}\times 
{\rm Sym}_{g\times g}(\bC^*)/\bC^*=
\Delta^{\rm univ}\setminus (\Delta^*)^{\rm univ}$.
\hfill $\clubsuit$ 
\end{construction}

\begin{remark}\label{non-reduced}
Allowing the tiling
$\cT$ to have strictly rational vertices corresponds
to allowing the components of the central
fiber to have non-reduced components.
More precisely, the irreducible components 
$V_i\subset X_0({\rm Cone}(\cT))\coloneqq f^{-1}(0)$ 
are in bijection with the $0$-cells 
$v_i\in \cT$, 
and the multiplicity $d_i$ of $V_i$ 
is the smallest positive integer
for which $d_iv_i\in \bfN$.
\end{remark}

\begin{remark}\label{kulikov} When $\cT$ has 
integral vertices, the total space $X({\rm Cone}(\cT))$ or 
$X^{\rm univ}_\circ({\rm Cone}(\cT))$ of the Mumford 
degeneration is smooth if 
and only if $\cT$ is a complete triangulation, 
i.e.~$\cT$ is a tiling by standard lattice simplices.
This condition ensures that ${\rm Cone}(\cT)$ 
is a regular fan---its cones are 
all standard affine cones.
Then $X({\rm Cone}(\cT))\to \Delta$ is a 
semistable, 
$K$-trivial degeneration (with smooth total space),
see Proposition \ref{k-triv} below. 
These are sometimes called Kulikov models,
in analogy to K3 degenerations. 
\end{remark}

\begin{remark} \label{imprimitive}
By passing to an 
intermediate cover $\cH_g\to \widetilde{\cA}_g\to \cA_g$
as in Remark \ref{monodromy-fit}, we may assume 
that the Levi quotient 
$\Gamma\subset \GL_g(\bZ)$ 
of the parabolic stabilizer 
at some Baily--Borel $0$-cusp of $\widetilde{\cA}_g$
acts freely on $\cP_g$.
Then, when $B$ is primitive,
the boundary divisor of 
$\widetilde{\cA}_g^{\,\bR_{\geq 0}B}$
is isomorphic to the discriminant 
divisor of $f^{\rm univ}_\circ$, rather
than a further finite quotient of it. In this case, the base
of $f^{\rm univ}_\circ$ glues onto $\widetilde{\cA}_g$ and thus,
we may extend 
$f^{\rm univ}_\circ$ to a family 
$f^{\rm univ}\colon X^{\rm univ}({\rm Cone}(\cT))
\to \widetilde{\cA}_g^{\,\bR_{\geq 0}B}$
extending the universal family
over $\widetilde{\cA}_g$.
A priori,
this extension exists
only in the category of analytic spaces. 
But it is always an algebraic space
(Prop.~\ref{algebraicity}),
and for certain choices of tiling $\cT$, 
we may ensure it is relatively projective,
see Section \ref{sec:polytope}.

When $B$ is not primitive, the base $\Delta^{\rm univ}$
of $f^{\rm univ}_\circ$
rather maps to $\widetilde{\cA}_g^{\,\bR_{\geq 0}B}$
by a map ramified over the toroidal boundary divisor,
to order the imprimitivity of $B$.
\end{remark}

\begin{proposition}\label{monodromy-1-param-mumford} 
The standard $1$-parameter
Mumford degeneration $f\colon
X({\rm Cone}(\cT))\to \Delta$
corresponding to a tiling $\cT$ of 
$\bfN_\bR/\Lambda_B$ has monodromy invariant $B$.
The period map on the universal cover of $\Delta^*$ 
is the nilpotent orbit through the origin: 
The maps $$\widetilde \Phi, \widetilde \Phi_{\rm{nilp}} \colon \bH \to {\rm Sym}_{g \times g}(\bZ) \backslash \cH_g $$ 
satisfy
$\widetilde{\Phi}(\tau)=
\widetilde{\Phi}_{\rm nilp}(\tau)=\tau B,$
where $u=e^{2\pi i \tau}$ is the coordinate
on $\Delta^*$. More generally, the period map
on the restriction of $f^{\rm univ}_\circ$ to 
the universal cover of
$(\Delta^*)^{\rm univ}\vert_{\{a\}}$ 
for $a\in {\rm Sym}_{g\times g}(\bC^*)$,
is given by
$\widetilde{\Phi}(\tau)=
\widetilde{\Phi}_{\rm nilp}(\tau) 
= \tfrac{1}{2\pi i}(\log a_{ij}) +\tau B$.
\end{proposition}

\begin{proof} On the one hand, the coordinatewise 
exponential $E(\tau B)$ is given by \begin{align}\label{mono-eq}
	\begin{pmatrix}
		\exp(2\pi i B_{11}\tau) & \cdots & \exp(2\pi i B_{1g}\tau) \\
		\cdots & & \cdots \\
		\exp(2\pi i B_{g1}\tau) & \cdots & \exp(2\pi i B_{gg}\tau) \\
	\end{pmatrix}=
	\begin{pmatrix}
		u^{B_{11}} & \cdots & u^{B_{1g}} \\
		\cdots & & \cdots \\
		u^{B_{g1}} & \cdots & u^{B_{gg}} \\
	\end{pmatrix}.\end{align}
 On the other hand, the fiber of the universal family
 $\cX_g\to \cA_g$ over the period matrix $\sigma \in \cH_g$
 is simply the complex torus 
 $\bC^g/(\bZ^g\oplus \bZ^g\sigma) \simeq (\bC^*)^g/
 \langle \textrm{rows of } E(\sigma)\rangle$ 
 so the proposition follows.\end{proof}

\begin{construction}[$k$-parameter case]\label{mumford-multi-fan} 
We now extend Construction \ref{mumford} to the multivariable
setting. Consider a collection of positive semi-definite
bilinear forms $B_i\in {\rm Sym}^2\bfM^\vee$ for $i=1,\dots,k$
for which $\sum_{i=1}^k B_i$ is positive-definite.
These bilinear forms define 
a collection of symmetric homomorphisms
$N_i\colon \bfM\to \bfN$. 

We consider the quotient of $\bfN\otimes \bC^*\simeq (\bC^*)^g$ 
 by the subgroup \begin{align*} u_1^{B_1}\cdots u_k^{B_k}\coloneqq 
 \langle &(u_1^{(B_1)_{11}}\cdots u_k^{(B_r)_{11}} , \dots, 
 u_1^{(B_1)_{1g}}\cdots u_k^{(B_k)_{1g}} ), \,\dots \,, \\
 &(u_1^{(B_1)_{g1}}\cdots u_k^{(B_k)_{g1}} , \dots, 
 u_1^{(B_1)_{gg}}\cdots u_k^{(B_k)_{gg}} )\rangle\simeq \bZ^g.\end{align*} 
If $|u_i|<1$ for all $i$, the resulting action of 
$u_1^{B_1}\cdots u_k^{B_k}$ on $(\bC^*)^g$ is 
properly discontinuous. Thus,
the quotient is a fibration 
\begin{align}\label{align:fibration:monodromy}f^\ast\colon X^*(B_1,\dots,B_k) \to (\Delta^*)^k
\end{align}
of PPAVs, over a punctured polydisk, 
where the base has coordinates $u_i$. 

To define
an extension over $\Delta^k$, we require a fan $\cS$
inside of $\bfN_\bR\times \bR^k\simeq \bR^{g+k}$ 
which is $\bfM$-periodic
for an action respecting the projection to $\bR^k$. 
More precisely, declare ${\bf m}\in \bfM$
to act linearly on $\bfN_\bR \times \bR^{k}$ 
by 
\begin{align}\label{fan-action} 
({\bf n},\,\vec{r}) \mapsto ({\bf n} + 
(\vec{r}\cdot \vec{N})({\bf m}), \,\vec{r})
\in \bfN_\bR \times \bR^k,
\end{align}
where $\vec{r}\cdot \vec{N}\coloneqq r_1N_1+\cdots+r_kN_k$
and $N_i\colon \bfM\to \bfN$ 
are the symmetric homomorphisms
associated to $B_i\in {\rm Sym}^2\bfM^\vee$ as above. 

The fan $\cS$ must then be $\bfM$-periodic
(with respect to the action (\ref{fan-action})) 
and the projection
to $\bR^k$ must induce a morphism of fans to
$(\bR_{\geq 0})^k$, 
such that
${\rm Supp}(\cS)$ contains $\bR^g\times (1,\dots,1)$.
Furthermore, we usually require that the fan
morphism 
$\cS\to (\bR_{\geq 0})^k$ is {\it flat}, 
that is, the image 
of any cone of $\cS$ intersected
with $\bfN\times \bZ^k$ is a
cone of $(\bR_{\geq 0})^k$ intersected
with $\bZ^k$.

Then the multivariable Mumford construction is the result
of quotienting by $\bfM$ the inverse 
image $Y(\cS)^{u\in \Delta^k}$ of $\Delta^k$ in
the infinite type toric variety $Y(\cS)$. 
We call this quotient
\begin{align}\label{mumford:fan}
f \colon X(\cS)\to \Delta^k.\end{align}
It is a proper, analytic,
flat extension of 
$f^\ast\colon  X^*(B_1,\dots,B_k) \to (\Delta^*)^k$, 
with flatness guaranteed by the
flatness of the fan map. 
As in Construction \ref{mumford}, 
the fibration $\bfM 
= u_1^{B_1} \cdots u_k^{B_k} \hookrightarrow 
\bfN \otimes \bC^\ast \to X_u \coloneqq f^{-1}(u)$ 
defines an exact sequence
\begin{align} \label{NH1M}0\to 
\bfN\to H_1(X_u,\bZ)\xrightarrow{\sigma} 
\bfM\to 0, \quad \quad  u = (u_1, \dotsc, u_k)
\in (\Delta^\ast)^k,\end{align}
and by choosing a section of $\sigma$ by taking
logarithms of $u_i$, the canonical 
pairing between $\bfM$ and $\bfN$ induces a 
well-defined principal polarization on $X_u$. 

Over the co-character $\Delta\to \Delta^k$ defined
by $u\mapsto (u^{r_1},\dots,u^{r_k})$, the construction
specializes to the $1$-variable Mumford Construction \ref{mumford}
associated to $B=r_1B_1+\cdots +r_k B_k$ and the relevant
fan is the restriction of $\cS$ to the inverse image
of $\bR_{\geq 0}\vec{r}\subset (\bR_{\geq 0})^k$; here
$\cS\vert_{\bR_{\geq 0}\vec{r}}\simeq {\rm Cone}(\cT)$
for a tiling $\cT$ depending on $\vec{r}$.
In particular, 
by Proposition \ref{monodromy-1-param-mumford},
the monodromy cone 
(see Definition \ref{monodromy-cone}) 
of the degeneration 
$f^\ast\colon  X^*(B_1,\dots,B_k) \to (\Delta^*)^k$ 
is given by $\bB \coloneqq \bR_{\geq 0}\{B_1,\dots,B_k\}$.

More generally, we may, as in Construction \ref{mumford},
perform the multivariable construction 
relatively over the torus 
$({\rm Sym}_{g\times g}(\bZ)/
\langle \bB\rangle) \otimes \bC^*$ 
by twisting the $\bfM$-action by some elements 
$a=(a_{ij})\in {\rm Sym}_{g\times g}(\bC^*)$.
Here, we have quotiented by 
$\langle \bB\rangle\otimes  \bC^*$, 
so as to reduce redundant 
moduli of the general fiber as much as possible.
We denote the resulting fibration
by $$f^{\rm univ}_\circ\colon
X^{\rm univ}_\circ(\cS)\to (\Delta^k)^{\rm univ}$$
where $(\Delta^k)^{\rm univ}\to 
({\rm Sym}_{g\times g}(\bZ)/
\langle \bB\rangle) \otimes \bC^*$ is a polydisk
bundle with $$\dim (\Delta^k)^{\rm univ} = g(g+1)/2+k-\dim \langle \bB\rangle.$$ Note that
$f^{\rm univ}_\circ \colon X^{\rm univ}_\circ(\cS)\to (\Delta^k)^{\rm univ}$ is a locally trivial
deformation of $f \colon X(\cS)\to \Delta^k$. Indeed, as analytic germs about the zero section
of $(\Delta^k)^{\rm univ}$,
the universal cover of the former 
is the product  of the universal cover of the latter
with the torus of 
twists $({\rm Sym}_{g\times g}(\bZ)/
\langle \bB\rangle) \otimes \bC^*$.
\hfill $\clubsuit$ 
\end{construction}

\begin{remark}\label{section-location} 
In terms of fans, the (possibly rational) 
origin section of $X(\cS)\to \Delta^k$ 
is declared to be the image of the 
subtorus whose cocharacter lattice is
$\{0\}\times \bZ^k\subset \bfN\times \bZ^k$. \end{remark}

\begin{remark}
Let $B_1, \dotsc, B_k \in {\rm Sym}_{g \times g}(\bZ)$ 
and let $f^\ast\colon X^*(B_1,\dots,B_k) \to (\Delta^*)^k$ 
be the family of $g$-dimensional PPAVs defined in 
\eqref{align:fibration:monodromy}, 
with fiber $X_t$, $t \in (\Delta^\ast)^k$. 
Then, as we have seen  
above, 
the monodromy bilinear form about $\set{u_i=0}$ (Def.~\ref{def:Bi}) equals $B_i$ 
for a suitable symplectic basis of $H_1(X_t,\bZ)$, 
cf.~Corollary \ref{corollary:monodromyrealization}.
\end{remark}

\begin{construction}[multiparameter case, cone of a fan for $\cA_g$]
\label{singular-base}
Suppose that the cone $$\bB\coloneqq \bR_{\geq 0}\{B_1,
\cdots, B_k\}\subset  \cP_g^+$$
is not standard affine, or not even simplicial.
Recall from Remark \ref{monodromy-fit} that
for some \'etale cover 
$\widetilde{\cA}_g\to \cA_g$, we have a
toroidal extension
$\widetilde{\cA}_g\hookrightarrow \widetilde{\cA}_g^{\,\bB}$
whose monodromy
cone is $\bB$.
For an appropriate choice of fan $\cS$,
Construction \ref{mumford-multi-fan} gives
a degeneration of abelian varieties over 
a $k$-dimensional polydisk $X(\cS)\to 
\Delta^k\subset Y(\bR_{\geq 0}^k)$.
In fact, we may descend 
the construction to an analytic open
neighborhood of the torus fixed point in the affine toric
variety $Y_{\bR \bB}(\bB)$, 
when $\cS$ is supported
rather in the vector space $\bfN_\bR\times \bR\bB$, and 
periodic with respect to the same action (\ref{fan-action}).
Here, the subscript $\bR \bB$ of $Y_{\bR\bB}$ 
refers to the fact that 
we take the toric variety of the polyhedral cone 
$\bB\subset \bR\bB$, sitting
inside the vector space $\bR \bB$, rather than
inside ${\rm Sym}_{g\times g}(\bR)$.

Taking the universal twist,
we produce a degeneration
$$X^{\rm univ}_\circ(\cS)\to T(\bB)\subset Y(\bB)$$ 
where now $Y(\bB)$ is the toric
variety of the polyhedral cone 
$\bB\subset {\rm Sym}_{g\times g}(\bR)$
and $T(\bB)$ is an analytic tubular neighborhood
of the deepest toric boundary stratum.
Then, following Section \ref{sec:toroidal-extensions},
we may analytically glue $T(\bB)$ along the complement $T^*(\bB)$
of the toric boundary to $\widetilde{\cA}_g$.
Taking
the gluing to respect the zero sections, we produce a degeneration
$$X^{\rm univ}_+(\cS)\to 
\widetilde{\cA}_g\cup_{T^*(\bB)} T(\bB)\eqqcolon \widetilde{\cA}_g^+\subset \widetilde{\cA}_g^{\,\bB}$$
which
analytically extends
the universal family 
$\widetilde{\cX}_g\to \widetilde{\cA}_g$.

One may unify 
the set-up of toroidal 
extensions of $\cA_g$ laid out in Section
\ref{sec:toroidal-extensions}, with the toroidal
construction of $X(\cS)$, by 
viewing $\cS$ as a fan
$\cS^{\rm univ}$ in
the larger space $\bR^g\oplus \cP_g^+ \subset  \bR^g
\oplus {\rm Sym}_{g\times g}(\bR)$ admitting
a morphism of fans $\cS^{\rm univ}\to \bB$ via projection to the second factor. 
Namikawa was the first to introduce
such ``mixed cone decompositions'',
see \cite[Sec.~3]{namikawa76} and 
\cite[Sec.~9]{namikawabook} 
(though in these texts,
the focus is one particular Mumford construction,
similar to Construction \ref{voronoi-mumford}).

The family $X^{\rm univ}_+(\cS)$ 
further extends
to a proper, flat family
\begin{align}\label{funiv}f^{\rm univ}\colon X^{\rm univ}(\cS)
\to \widetilde{\cA}_g^{\,\bB}\end{align}
surjecting onto the base $\widetilde{\cA}_g^{\,\bB}$. 
Indeed, this follows from \cite[Ch.~VI.1]{fc}. Defining $\cS^{\rm univ}$ as above, and
taking $T_{\rm max}(\bB)\subset Y(\bB)$ 
as the maximal analytic open neighborhood
of the toric boundary over 
which the action of $\bfM$ on the inverse 
image of $T_{\rm max}(\bB)$ in $Y(\cS^{\rm univ})$
is properly discontinuous, the base $T_{\rm max}(\bB)$ 
surjects onto the toroidal extension
$\widetilde{\cA}_g^{\,\bB}$. Concretely,
$T_{\rm max}(\bB)$ is the relative closure
of $U_\bZ\backslash \cH_g\subset 
{\rm Sym}_{g\times g}(\bC^*)$,
see (\ref{torus-embedding}), in the toric
variety $Y(\bB)$.
Then, 
the family descends from a partial 
uniformization of toroidal strata 
fibering over Baily--Borel strata
of intermediate dimension, to yield \eqref{funiv}.

The precise relation to Faltings--Chai 
\cite[Ch.~VI.1]{fc} is that the fan 
$\cS^{\rm univ}$
is an instance of a 
``$\GL(X)\ltimes X^s$-admissible 
polyhedral cone decomposition''
(for $s=1$) in the terminology of 
{\it loc.cit.}~up to the following
three 
minor modifications: We do not
demand as in \cite[Ch.~VI, Def.1.3.(iii)]{fc}
that $\cS^{\rm univ}$ 
defines 
a complete fan 
(i.e.~a compactification of $\cX_g$),
when working on an \'etale cover
$\widetilde{\cA}_g\to \cA_g$ we only demand
admissibility for a finite index subgroup,
and we do not require the existence of a 
polarization function, 
see \cite[Ch.~VI.1, Rem.~1.6.(a)]{fc}.
The constructions used to prove
\cite[Ch.~VI.1, Thm.~1.13]{fc} generalize 
to this setting, to yield an extension $\widetilde{\cX}_g^{\,\fG}\to 
\widetilde{\cA}_g^{\,\fF}$ 
of the universal family, for any morphism
$\fG\to \fF$ from a mixed cone decomposition
(for $\widetilde{\cX}_g$) 
to a cone decomposition (for $\widetilde{\cA}_g$).
The family \eqref{funiv} is a special case. 

It is crucial
for our applications in \cite{companion}
to achieve both a smooth total 
space {\it and} flatness
over $\widetilde{\cA}_g^{\,\bB}$ 
(see Proposition \ref{X-smooth} and Theorem \ref{theorem:extension}). 
As mentioned in 
\cite[Ch.~VI, Def.~1.3.(v), Rem.~1.4]{fc},
achieving both of these properties is a 
hard combinatorial problem, closely related
to the main result of \cite{alt}; in general,
it is only possible after modifying the base 
$\widetilde{\cA}_g^{\,\bB}$. 

As additional 
historical notes, the first example of a complete
fan for $\cX_g$ (equidimensional but
not regular) was provided by Namikawa,
see e.g.~\cite[Sec.~13, Prop.~13.5, Thm.~13.6]{namikawa76}. The generalization
of Faltings--Chai to the more general
setting of mixed Shimura varieties 
is Pink's dissertation, see especially
\cite[Ex.~2.25, Ch.~6, Ch.~10]{pink}
for discussion relevant to $\cX_g$.
\hfill $\clubsuit$ 
\end{construction}

As an example of the above construction,
the Tate curve extends the family of elliptic curves
$\bC^*/u^\bZ$ over the unit
disk $\Delta_u = \{u\in \bC\,:\,|u|<1\}$, see Example \ref{ex:tate}.
The maximally extended base, on which $u^\bZ$
acts properly discontinuously, is 
$\Delta_u \supset 
\Delta_u^\ast\simeq \bZ\backslash \bH$. 
 The family over 
$\Delta_u^\ast$ descends
(as an orbifold)
 along the infinite degree surjection $\bZ\backslash \bH\to \SL_2(\bZ)\backslash \bH$, to the universal family over 
the orbifold $\cA_1 = \SL_2(\bZ)\backslash \bH$.
Only the punctured disc $\Delta_u(e^{-2\pi})^*$ 
of the smaller analytic disk 
$\Delta_u(e^{-2\pi}) \coloneqq 
\{u\in \bC\,:\,|u|<e^{-2\pi}\}\subset \Delta_u$
embeds, on the level of coarse spaces,
into $\SL_2(\bZ)\backslash \bH$ 
(as stacks, the degree 
of the map from $\Delta_u(e^{-2\pi})^*$
onto its image is,
rather, equal to two,
because of the $\bZ/2$-gerbe on the target). 

\begin{remark}\label{casual}
An additional subtlety is that, to glue
$f^{\rm univ}$ to the universal family, requires
in general, a cover of $\widetilde{\cA}_g^{\,\bB}$ 
which is \'etale in the punctured neighborhood
of the boundary, but branched over the boundary divisors.
For instance, when $\bB = \bR_{\geq 0}B$ is a ray,
this branched cover is necessary if and only if
$B$ is imprimitive (Rem.~\ref{imprimitive}). 
Such a branched cover of 
$\widetilde{\cA}_g^{\,\bB}$ is guaranteed to exist
in an affine open neighborhood of the 
deepest toroidal stratum, as shown
in Proposition \ref{monodromyrealization}.
But it is unclear, in general, whether there  
exists a {\it global} \'etale cover of 
$\widetilde{\cA}_g$ achieving the desired branching
behavior. For example, supposing 
the monodromy cone were of the form
$\bB = \bR_{\geq 0}\{3B_1, 5B_2\}$ 
for $B_1, B_2$ primitive---is then the local, toroidal 
branched cover of $\widetilde{\cA}_g^{\,\bB}$ which
is branched to orders $3$, $5$ over the two toric
boundary divisors $\partial_1,\partial_2\subset \widetilde{\cA}_g^{\,\bB}$ induced by
passing to a further finite index
subgroup of $\Sp_{2g}(\bZ)$?

Regardless, this finite cover does 
exist in an
affine open neighborhood of the relevant 
boundary stratum. By a slight 
abuse of notation, we continue
to notate the resulting cover and
its toroidal extension
by $\widetilde{\cA}_g$ and $\widetilde{\cA}_g^{\,\bB}$,
even though $\widetilde{\cA}_g$ is only a finite \'etale
over a Zariski open subset of $\cA_g$.
\end{remark}

To summarize Constructions \ref{mumford-multi-fan} and
\ref{singular-base}:

\begin{proposition}\label{extension-prop}
Let $\cS$ be a fan in 
$\bfN_\bR\times \bR\bB$ 
satisfying the properties described in Constructions
\ref{mumford-multi-fan} and \ref{singular-base}. 
In particular, $\cS$ is 
$\bfM$-invariant under the action (\ref{fan-action}),
for a collection of symmetric
bilinear forms $\{B_1,\dots,B_k\}\subset 
{\rm Sym}^2 \bfM^\vee$ which are the 
rays generating a polyhedral cone in
the space of positive semidefinite bilinear forms on $\bfM$,
with $\sum_{i=1}^k B_i$ positive-definite.

Then Construction \ref{singular-base}
produces a flat, proper extension of 
the universal family over
$\widetilde{\cA}_g$, 
in the category of complex analytic spaces, 
$$X^{\rm univ}(\cS)\to 
\widetilde{\cA}_g^{\,\bB}.$$
\end{proposition}

We now examine when a Mumford degeneration
is $K$-trivial:

\begin{proposition}\label{k-triv}
    If $\cS_{(1,\dots,1)}$ is a tiling
    (as opposed to a $\bQ$-tiling) of $\bfN_\bR$,
    then the multivariable Mumford construction
     $f\colon X(\cS)\to \Delta^k$ 
     (see \ref{mumford-multi-fan})
     is $K$-trivial: $K_{X(\cS)}\sim 0$.
\end{proposition}

\begin{proof}
    The universal cover of $X(\cS)$ 
    admits an analytic open embedding
    into the toric variety $Y(\cS)$, whose anticanonical
    divisor is the reduced toric boundary. This toric boundary
    in turn is the reduced inverse image of the union
    of the coordinate hyperplanes 
    $V(u_1\cdots u_k)\subset \bC^k$ under
    the toric
    morphism $Y(\cS)\to Y(\bR_{\geq 0}^k)\simeq \bC^k$.
    Thus, if the inverse image of $V(u_1\cdots u_k)$
    equals its reduced inverse image, we conclude
    that $Y(\cS)$ and in turn $X(\cS)$ are 
    relatively $K$-trivial.

    To check that the divisors contained in the 
    inverse image of $V(u_i)$ are reduced, it suffices
    to restrict to the arc 
    $\Delta\to \Delta^k$, $u\mapsto (u,\dots,u)$. 
    We now apply Remark \ref{non-reduced}.
\end{proof}

\subsection{Weight filtration and dual complex}\label{sec:wf} 

\begin{proposition}\label{mn-identify}
Let $X(\cS)\to \Delta^k$ be a fan
Mumford Construction \ref{mumford-multi-fan} 
and let $X_t$ be the fiber over a point 
$t\in (\Delta^*)^k$.
Consider the exact sequence
$$0\to \bfN \to H_1(X_t,\bZ)\to \bfM\to 0,$$
see \eqref{NH1M}. 
Then $\bfN = \bfN \subset H_1(X_t,\bZ)$ is the weight 
filtration $W_{-2} = W_{-1} \subset W_0$ 
of the limiting mixed Hodge structure on $H_1(X_t,\bZ)$.
That is, there are integral isomorphisms
\begin{align*} 
\bfM &\simeq {\rm gr}^W_0H_1(X_t,\bZ), \\
\bfN &\simeq {\rm gr}^W_{-2}H_1(X_t,\bZ).
\end{align*}
Furthermore, 
$N_i\simeq N_i^{\rm mon}$ and $B_i\simeq B_i^{\rm mon}$
where
$N_i^{\rm mon}\colon {\rm gr}^W_0H_1(X_t,\bZ)\to 
{\rm gr}^W_{-2}H_1(X_t,\bZ)$ and 
$B_i^{\rm mon}\in 
{\rm Sym}^2 ({\rm gr}^W_0H_1(X_t,\bZ))^\vee$ 
are the monodromy
operators and bilinear forms of Section \ref{sec:monodromies}.
\end{proposition}

\begin{proof} The identification of $\bfM$ and $\bfN$
with the stated graded pieces of the weight filtration follows
by construction, see e.g.~(\ref{NH1M})---the homology group 
$H_1(\bfN\otimes \bC^*, \bZ)\simeq \bfN$ 
is spanned by
the vanishing cycles, 
which are null-homologous in the neighborhood
of any $0$-dimensional toric stratum of $X(\cS)$.
Proposition \ref{monodromy-1-param-mumford}
shows that the monodromy operator $N_i^{\rm mon}$ agrees
with $N_i\colon \bfM\to \bfN$ and the hypothesis that 
$B= \sum_{i=1}^kB_k>0$ ensures that $W_{-2}=({\rm im}\,N)^{\rm sat}$
for $N=\sum_{i=1}^k N_i$ agrees with $\bfN$.
\end{proof}

\begin{proposition}
Let $X(\cS)\to \Delta^k$ be a fan
Mumford Construction \ref{mumford-multi-fan}. Then there is a canonical isomorphism 
${\bf M} \simeq H_1(\Gamma(X_0), \bZ)$ where $\Gamma(X_0)$
is the dual (polyhedral) complex of the central fiber $X_0$. 
\end{proposition}
\begin{proof}
    The dual polyhedral complex of the $\bfM$-prequotient
is the infinite periodic polyhedral decomposition 
of $\bfN_\bR$ given by the preimage of $(1,\dots,1)\in \bR^k$
under the morphism of fans $\cS\to (\bR_{\geq 0})^k$, see
Construction \ref{mumford-multi-fan}. It follows
that $\Gamma(X_0)\simeq \cS_{(1,\dots,1)}/N(\bfM)=
\bfN_\bR/\Lambda_B$. Thus, there is a canonical
isomorphism $H_1(\Gamma(X_0),\bZ)\simeq \bfM$.
\end{proof}
 
\subsection{Mumford construction, polytope version} \label{sec:polytope}
We discuss now a polytopal version of 
the Mumford degeneration, which outputs
a relatively projective degeneration, together
with a relatively ample line bundle.
Furthermore, it is isomorphic to the fan construction 
as in Section \ref{sec:fan}, for an appropriate
choice of fan $\cS$ in $\bfN_\bR \times \bR^k$. 
Our approach is, in part, 
inspired by Gross--Siebert \cite[Sec.~2]{gross},
and their construction of canonical theta functions, 
building on the classical
theory of theta functions, see 
e.g.~\cite[Prop.~II.1.3 and Thm.~II.1.3]{tata}. 
It is primarily based on a ``PL version'' 
of the classical theory,
in line with Alexeev--Nakamura 
\cite{alexeev-nakamura}. 

Let $A = \bC^g/(\bZ^g \oplus \bZ^g\sigma)$ be an 
abelian variety with principal polarization $L$. 
The classical theory of theta functions
studies explicit sections of the powers of $\cL$, a 
lift of $L$, by pulling back to
the universal cover $\pi\colon \bC^g\to A$. Since
$\pi^*\cL\simeq \cO_{\bC^g}$, such sections can be understood
via holomorphic functions on $\bC^g$, with appropriate
factors of automorphy under the deck action
of the periods $\bZ^g\oplus \bZ^g\sigma$. 
Such holomorphic functions are called 
\emph{theta functions}. We do the same here, for
the intermediate cover $(\bC^*)^g\to A$ discussed in the introduction, of the fibers $A=X_t$ of a degenerating family
of PPAVs $X^*\to (\Delta^*)^k$. These theta functions
extend as holomorphic sections of a line
bundle over a toric extension of $X^*$ over $\Delta^k$.

Consider the standard torus 
$\bT^g\coloneqq  \bR^g/\bZ^g=\bfM_\bR/\bfM$.
We define $\bZ {\rm PL}/\bZ {\rm L}$ to
be the sheaf on $\bT^g$ of 
$\bZ$-piecewise linear functions 
modulo the subsheaf of $\bZ$-linear
functions. On an open set 
$U\subset \bfM_\bR/\bfM$, the sections 
$\bZ{\rm PL}(U)\coloneqq\{f\colon U\to \bR\}$
consist of continuous, piecewise linear
functions, which in a domain $D\subset U$ of linearity
are of the form 
\begin{align}\label{loc-zpl}
    f\vert_D({\bf m})= a_1m_1+\cdots+a_gm_g+a_{g+1}
\end{align}
with $a_j\in \bZ$. Here $m_i$ are integral coordinates on 
$\bfM_\bR$ which define local coordinates on $D$.
Similarly, $\bZ{\rm L}$ is the sheaf
of locally $\bZ$-linear functions on $U$, of the same shape.
A section $\bZ{\rm PL}/\bZ{\rm L}(U)$ 
can be understood globally on $U$ in terms
of its ``bending locus'' (Definition \ref{bendinglocus}).

The pull-back of a section
$\overline{b}_i\in H^0(\bT^g, \bZ{\rm  PL}/\bZ {\rm L})$ to
the universal cover $\bfM_\bR \to \bT^g$ 
lifts to a $\bZ$-piecewise linear function 
$b_i\colon \bfM_\bR \to \bR$, since $H^1(\bfM_\bR,\bZ {\rm L})=0$.
It is integer-valued
on $\bfM$ and, more generally, $\tfrac{1}{w}\bZ$-valued
on $\tfrac{1}{w}\bfM$ for any positive integer $w\in \bN$.
The function $x \mapsto b_i(x+{\bf m})-b_i(x)$
is a linear function on $\bfM_\bR$ for all 
${\bf m}\in \bfM$, because $b_i$ is lifted from $\bT^g$.
Conversely, any 
$\bZ$-piecewise linear function $b_i \colon M_\bR \to \bR$, 
such that $b_i(x+{\bf m}) - b_i(x)$ is linear, 
for all ${\bf m}\in \bfM$, 
descends to a section 
$\overline{b}_i \in 
H^0(\bT^g, \bZ{\rm  PL}/\bZ {\rm L})$ that determines 
the equivalence class $[b_i]$ uniquely 
(where $b_i \sim b_i'$ if the difference 
$b_i - b_i'$ is linear). 

The domains
of linearity of $b_i$ are rational polyhedra.
The sections $\overline{b}_i$ for which a lifted
function $b_i$ is convex
form a convex polyhedral subcone 
of $H^0(\bT^g, \bZ{\rm  PL}/\bZ {\rm L})$. 

\begin{definition} \label{bendinglocus}
Associated to $\overline{b}_i\in 
H^0(\bT^g, \bZ{\rm  PL}/\bZ {\rm L})$ is a weighted 
polyhedral complex in $\bT^g$, called the 
{\it bending locus} ${\rm Bend}(\overline{b}_i)$. 
Its faces are the codimension $1$
polytopes in $\bT^g$ along which 
$\overline{b}_i$ is non-linear,
and the {\it bending parameter} 
(positive when $\overline{b}_i$ is convex)
defining the weight on
a codimension $1$ polytope, 
is the change in slope of 
the restriction of $\overline{b}_i$ 
an integral, complementary segment
to the hyperplane containing the face.
\end{definition}

\begin{definition} \label{dicing}
    Let
$\{\overline{b}_1,\dots,\overline{b}_k\}
\in H^0(\bT^g, \bZ{\rm  PL}/\bZ {\rm L})$ 
be a collection of convex sections, with
nontrivial bending in every direction, i.e.~for every non-zero
${\bf m}\in \bfM$, there exists some
$i\in \{1,\dots,k\}$ for which 
$b_i(x+{\bf m})-b_i(x)\not\equiv 0$ is not
identically zero. 
Equivalently, 
$\bigcup_i {\rm Bend}(\overline{b}_i)$ cuts
$\bT^g$ into polytopes.
We say that the $\overline{b}_1,\dots,\overline{b}_k$ 
are {\it dicing} if the polyhedral
decomposition 
$\bigcup_i {\rm Bend}(\overline{b}_i)$ 
has integral
vertices. 
\end{definition}

The dicing condition is quite restrictive, since
only the origin of $\bT^g$ may appear as a vertex
of $\bigcup_i {\rm Bend}(\overline{b}_i)$. We will
relax this hypothesis in Construction \ref{veronese}.

\begin{example}
Let $g=1$ and $\bfM = \bZ$, so that ${\bf M}_\bR/{\bf M}= \bR/\bZ$. 
Define a PL function
$b\colon \bR\to \bR$ which is linear on each interval
$[m,m+1]$, $m\in \bZ$, and has values on $\bZ$ equal to 
$b(m) = \tfrac{1}{2}(m^2-m)$. The graph of $b$ is depicted
in the left of Figure \ref{fig:tate}. The locus where $b$ 
is non-linear is $\bZ$, hence ${\rm Bend}(\overline{b}) 
= \set{0} \in \bR/\bZ = \bT^1$, with weight one, 
see Figure \ref{fig:tate3}. 
\end{example}

\begin{remark} \label{completion}
Any projective morphism $X \to \Delta^k$ of 
analytic spaces gives rise to an algebraic family 
$\wh X \to \Spec \bC[[u_1, \dotsc, u_k]]$, 
the \emph{formal completion} of $X \to \Delta^k$. 
Indeed, the projectivity of $X\to \Delta ^k$ implies that 
there is a positive integer $N$ such that 
$X\subset \Delta^k\times \mathbb P^N$ is cut out 
by homogeneous polynomials 
whose coefficients are convergent power series.
The completion $\wh X$ is then cut out by the 
same equations, viewing the convergent 
power series as formal power series in 
$\bC[[u_1, \dotsc, u_k]]$.
\end{remark}

\begin{construction}\label{mumford-polytope}
Let
$\{\overline{b}_1,\dots,\overline{b}_k\}
\in H^0(\bT^g, \bZ{\rm  PL}/\bZ {\rm L})$ 
be a collection of convex sections, with
nontrivial bending in every direction, 
and assume that the 
$\overline b_1, \dotsc, \overline b_k$ are dicing.
Let $\overline{v}\in \tfrac{1}{w}\bfM/\bfM\in \bT^g$
be a $\tfrac{1}{w}$-integral point of $\bT^g$, for some 
positive integer $w\in \bN$. 
We define the {\it weight $w$ theta function 
associated to $\overline{v}$} to be 
\begin{align}\label{theta-def}
\Theta_{\overline{v}}(z_1,\dots,z_g, u_1,\dots,u_k)\coloneqq 
\sum_{v\,\in\, \overline{v}+\bfM} 
(z_1^{x_1(v)}\cdots z_g^{x_g(v)}
u_1^{b_1(v)}\cdots u_k^{b_k(v)})^w\end{align}
where $x_i$ is the $i$-th coordinate function,
cf.~\cite[Sec.~4.5]{alexeev-nakamura}.

The condition that the $\overline{b}_i$ have bending
in every direction ensures that this power series
converges in an appropriate power series ring 
(which notably involves both negative and positive
powers of $z_i$). Consider the $\bC[[u_1,\dots,u_k]]$-module
$$R({\overline{b}_1,\dots,\overline{b}_k})\coloneqq 
\bigoplus_{w=0}^\infty 
\,\bigoplus_{\overline{v}\,\in\, \frac{1}{w}\bfM/\bfM}
\bC[[u_1,\dots,u_k]]\cdot \Theta_{\overline{v}}.$$ 
Expanding the product of two theta functions
$\Theta_{\overline{v}_1}$, $\Theta_{\overline{v}_2}$
of weights $w_1$, $w_2$  
by collecting all monomial terms into $\bfM$-orbits
(see (\ref{polytope-action}) below), 
we see that there is an
expansion
\begin{align}
    \label{mult-rule}
\Theta_{\overline{v}_1}\Theta_{\overline{v}_2}=
\sum_{\overline{v}_3\in \frac{1}{w_1+w_2}\bfM/\bfM} 
c_{\overline{v}_1\overline{v}_2}^{\overline{v}_3}
(u_1,\dots,u_k) \Theta_{\overline{v}_3}
\end{align}
where the coefficients
$c_{\overline{v}_1\overline{v}_2}^{\overline{v}_3}
(u_1,\dots,u_k)\in \bZ[[u_1,\dots,u_k]]$ are 
integral power series, as opposed to simply Laurent series, 
by the convexity of the $b_i$, see e.g.~\cite[Eqn.~(2.5)]{gross}.
Note that
to get a nonzero coefficient,
there must be a lift of $\overline{v}_3$ of the form 
$$v_3 = \frac{w_1v_1+w_2v_2}{w_1+w_2}.$$ 
Hence
$R({\overline{b}_1,\dots,\overline{b}_k})$ is 
closed under multiplication. It is, furthermore,
a finitely generated, graded ring over $\bC[[u_1,\dots,u_k]]$.
Consider the resulting projective 
$\bC[[u_1,\dots,u_k]]$-scheme 
\begin{align}\label{projfamily}
{\rm Proj}_{\bC[[u_1,\dots,u_k]]} \,
R(\overline{b}_1,\dots,\overline{b}_k) 
\to {\rm Spec} \,\bC[[u_1,\dots,u_k]].\end{align}
It is a degeneration of PPAVs of dimension $g$,
with the theta functions providing the projective
embedding, 
which is the 
completion (in the sense of Remark \ref{completion}) 
of a relatively projective 
complex analytic degeneration 
\begin{align}\label{mumford:polytope}f\colon 
X({\overline{b}_1,\dots, \overline{b}_k})\to 
\Delta^k\end{align} over a polydisk. This 
can be proven e.g.~by observing
that $\Theta_{\overline{v}}$ are analytically convergent
power series on a Mumford fan construction
when all $|u_i|<1$, a fact which
is justified in the course of the proof
of Theorem \ref{poly-is-fan}.
We call $f$ the {\it Mumford degeneration}
associated to $\{\overline{b}_1,\dots,\overline{b}_k\}$. 

\begin{remark} \label{genericfiberabelian}
That the generic fiber of \eqref{projfamily} is 
an abelian variety also
follows from the classical theory of theta functions. 
\end{remark}

We have assumed that the $\overline{b}_i$ are dicing, 
see Definition \ref{dicing}.
Define $\Gamma\subset \bfM_\bR\times \bR^k$ 
as the {\it overgraph} of the collection of functions
$(b_1,\dots,b_k)\colon \bfM_\bR\to \bR^k$, that is,
\begin{align}\label{overgraph}\Gamma =
\Gamma(b_1,\dots,b_k)+(\bR_{\geq 0})^k
\subset \bfM_\bR\times \bR^k.\end{align}
Then, $\Gamma$ is an 
infinite convex, locally finite 
polytope in $\bfM\times \bR^k$,
whose faces are integral polytopes.
We may think of the lattice 
points, in $\Gamma\cap (\bfM\times \bZ^k)$, as
the monomial sections of $\cO(1)$ 
on the corresponding infinite type 
toric variety $Y=Y_\Gamma$,
see Section \ref{sec:toric-varieties}.
Similarly, we may think of the 
$\tfrac{1}{w}$-integral points of 
$\Gamma$ as the 
monomial sections of $\cO(w)$, 
cf.~Remark \ref{toricpolytope}.

Then,
the theta function $\Theta_{\overline{v}}$ for 
$\overline{v}\in \tfrac{1}{w}\bfM/\bfM$ is the result
of summing such monomials, over an $\bfM$-orbit,
where ${\bf m}\in \bfM$ acts on $\bfM_\bR\times \bR^k$ by 
the affine-linear action
\begin{align}\label{polytope-action}
(x,\vec{r})\mapsto 
(x+{\bf m}, \,
\vec{r} + \vec{b}(x+{\bf m})-\vec{b}(x))\end{align}
and $\vec{b} = (b_1,\dots,b_k)$.
Note that this action of $\bfM$ preserves $\Gamma$.

Form the normal fan $\cS$ to $\Gamma$.
As (\ref{polytope-action}) gives an action
of $\bfM$ on $\Gamma$, it induces an $\bfM$-action
on $\cS$. This action agrees with the 
action (\ref{fan-action})
for a multivariable Mumford fan construction,
associated to the bilinear forms $B_1,\dots, B_k\in 
{\rm Sym}^2\bfM^\vee$ 
defined by the equations
\begin{align}\label{Bi}
B_i({\bf m},{\bf m}')\coloneqq b_i({\bf m}+{\bf m}')-
b_i({\bf m})-b_i({\bf m}')+b_i(0),\quad 
{\bf m},{\bf m}'\in \bfM.\end{align}

By the convexity of the $b_i\colon \bfM_\bR\to \bR$,
the normal fan 
$\cS\subset (\bfM_\bR \times \bR^k)^\vee$
admits a canonical morphism to the fan 
$\bR_{\geq 0}^k\subset \bR^k\simeq (\bR^k)^\vee$, 
given by restricting
linear functionals in 
$(\bfM_\bR\times \bR^k)^\vee$ to
$\bR^k$. Then, the data of $\cS$, together with
the projection to $\bR_{\geq 0}^k$, defines the data
of a Mumford fan Construction \ref{mumford-multi-fan}.
We will prove that
$X(\overline{b}_1,\dots,\overline{b}_k)\simeq X(\cS)$
 in Theorem \ref{poly-is-fan}. 

To ``twist'' the construction, as
in Constructions \ref{mumford}, \ref{mumford-multi-fan},
\ref{singular-base}, by
some continuous parameters 
$a= (a_{ij})\in {\rm Sym}_{g\times g}(\bC^*)$, and produce
a universal degeneration which
represents all possible continuous moduli of degenerations
of the specified combinatorial type, we must introduce
appropriate coefficients 
$$ \Theta_{\overline{v}}^a
(z_1,\dots,z_g, u_1,\dots,u_k)\coloneqq 
\sum_{v\,\in\, \overline{v}+\bfM} d_v(a)
(z_1^{x_1(v)}\cdots z_g^{x_g(v)}
u_1^{b_1(v)}\cdots u_k^{b_k(v)})^w$$
for $d_v(a)\in \bC^*$. This twists the structure
constants to give a graded ring 
$R^a({\overline{b}_1,\dots,\overline{b}_k})$
and ranging over the moduli of $a$,
produces a relatively projective
multivariable Mumford degeneration over the base
which is a
${\rm Spec}\,\bC[[u_1,\dots,u_k]]$-bundle
over $(\bC^*)^D$,
$D = \dim \cA_g - {\rm rank}\,\bR\{B_1,\dots, B_k\}$.
It agrees on the general fiber
with the quotient by the family of subgroups
$au_1^{B_1}\cdots u_k^{B_k}$.
These constants $d_v(a)$ form part of the
so-called ``degeneration data'' of \cite{fc}.
\hfill $\clubsuit$ 
\end{construction}

\begin{notation} The construction
of the ring $R(\overline{b}_1,\dots,\overline{b}_k)$
depends only on the $\overline{b}_i$ and not the lifts
$b_i$ to PL functions on $\bfM_\bR$. But
it is usually easiest to specify 
$\overline{b}_i$ by providing the PL function 
$b_i\colon \bfM_\bR\to \bR$.
With this in mind, we will henceforth 
notate the Mumford degeneration 
$X({\overline{b}_1,\dots, \overline{b}_k})\to 
\Delta^k$ by $X(b_1,\dots,b_k)\to \Delta^k$, 
${\rm Bend}(\overline{b}_i)$
by ${\rm Bend}(b_i)$, etc. \end{notation}

\begin{remark}\label{d-dicing}
The number of theta functions
$\Theta_{\overline{v}}$ of weight $w$
is exactly $w^g$. These functions form the {\it theta basis},
a canonical (up to scaling) basis of sections of 
$H^0(X({b_1,\dots,b_k}),\,
 \cL^{\otimes w})$ where $\cL=\cO(1)$ is a lift
 of the relative principal
 polarization. In particular, 
 $\Theta = V(\Theta_{\overline{0}})$ extends as a
 Cartier divisor over the degenerating family 
 $X({b_1,\dots,b_k})\to \Delta^k$.

 Since the $b_i$ are dicing, 
 $\Gamma$ is an integer polyhedron, 
 which is why $\cO_Y(1)$ is Cartier on $Y=Y_\Gamma$.
 It also admits a natural linearization with respect
 to the $\bfM$-action. This is why the principal
 polarization extends, as a line bundle, 
 to $X(b_1,\dots,b_k)$.
 Absent the dicing condition, 
 one may consider the least positive
 integer $d$ for which the overgraph $\Gamma$
 is a $\tfrac{1}{d}(\bfM\times \bZ^k)$-integral polyhedron.
 Then $\cO_Y(d)$ defines an integral polyhedron
 and so descends as a line bundle on $X(b_1,\dots,b_k)$
 which is a lift of $d$ times a principal 
 polarization on the smooth fibers.
\end{remark}

\begin{definition}\label{def:d-dicing}
We say that $\{b_1,\dots,b_k\}$
are {\it $\tfrac{1}{d}$-dicing} if
they are $\bQ$-piecewise linear,
the corresponding
overgraph $\Gamma$
of $\Gamma(b_1,\dots,b_k)\subset \bfM_\bR\times \bR^k$
is a $\tfrac{1}{d}(\bfM\times \bZ^k)$-integral 
polyhedron, 
and in the local form (\ref{loc-zpl}),
the slopes $a_1,\dots,a_g\in \bZ$ are still integral, but 
we allow $a_{g+1}\in \tfrac{1}{d}\bZ$. We denote sheaves
of functions with such a local form by $\tfrac{1}{d}\bZ{\rm PL}$
and $\tfrac{1}{d}\bZ{\rm L}$. \end{definition}

\begin{construction}\label{mumford-polytope-2}
    Like Construction \ref{singular-base} 
    vis-\`a-vis Construction
    \ref{mumford-multi-fan}, we 
    generalize Construction \ref{mumford-polytope} to the case
    where $\{\overline{b}_1,\dots,\overline{b}_k\}$ 
    are the extremal rays of a convex polyhedral cone 
    $\mathbbm{b}\subset H^0(\bT^g,\bZ{\rm PL}/\bZ{\rm L})$ mapping isomorphically to a convex polyhedral cone
    $\bB \subset \cP_g^+$ under the map 
    $\overline{b}_i \mapsto B_i$
    with $B_i \in \Sym^2\bfM^\vee$ defined in \eqref{Bi}.
    We replace $(b_1,\dots,b_k)$ by the PL function 
    \begin{align*}
    \bfM_\bR & \to (\bR\mathbbm{b})^\vee
    \simeq 
    \bR^{\dim \bR \mathbbm{b}} \\ 
    {\bf m} & \mapsto (b\mapsto b({\bf m})). 
    \end{align*}
    Otherwise, the
    details of Construction \ref{mumford-polytope} 
    are the same.
    The output is a relatively projective
    degeneration of abelian varieties
    $$X(\mathbbm{b})\to T(\mathbbm{b})\subset Y(\mathbbm{b})$$ 
    over an analytic tubular neighborhood
    of the torus fixed point of $Y(\mathbbm{b})$. Performing
    this construction with the universal twist by 
    $a\in {\rm Sym}_{g\times g}(\bZ)/\langle \bB\rangle\otimes \bC^*$,
    for $\bB=\bR_{\geq 0}\{B_1,\dots,B_k\}$, 
    and extending/descending over the toroidal
    extension $\widetilde{\cA}_g^{\,\bB}$ as in 
    Construction \ref{singular-base}, we may
    produce an analytically-locally relatively projective, 
    analytic 
    extension of the universal family 
    $$X^{\rm univ}(\mathbbm{b})\to  \widetilde{\cA}_g^{\,\bB},$$
    see Proposition \ref{extension-prop} 
    (we descend from an open set 
    $T_{\rm max}(\mathbbm{b})\subset Y(\mathbbm{b})$
    as in Construction \ref{singular-base} to verify
    the analytic-local projectivity). 
    We will show in Section \ref{sec:alexeev}
    that $X^{\rm univ}(\mathbbm{b})\to  \widetilde{\cA}_g^{\,\bB}$ 
    is an \'etale-locally projective morphism
    of algebraic spaces over $\widetilde{\cA}_g^{\,\bB}$. 
    \hfill $\clubsuit$ 
\end{construction}

\begin{remark} Our primary case of interest in 
\cite{companion} is where $\mathbbm{b}$ 
defines a simplicial cone $\bB\subset \cP_g^+$
and in this setting, Constructions \ref{mumford-multi-fan}
and \ref{mumford-polytope} over a polydisk will suffice. \end{remark}

The next proposition follows directly from Construction
\ref{mumford-polytope}:

\begin{proposition}\label{bending-read}
    Let $(\bT^g,b_1 ,\dots,b_k)$
    define a polarized, multivariable Mumford degeneration 
    $X(b_1,\dots,b_k)\to \Delta^k$. Then the following hold:
    \begin{enumerate}
        \item 
    The intersection complex
    of the fiber $X_I$ over the generic point
    of the coordinate subspace $V(u_i\,:\,i\in I)$ is 
    the polyhedral decomposition 
    $\bigcup_{i\in I}{\rm Bend}(\overline{b}_i)$
    of $\bT^g$.
    \item 
 The polytopes of this 
 polyhedral decomposition, when compact,
    are the polytopes of the polarized toric components, 
    in the sense of Remark \ref{toricpolytope}. 
    \item
    Non-compact faces $F$ of 
    $\bigcup_{i\in I}{\rm Bend}(\overline{b}_i)$
    are of the form $F\simeq F_0\times \bT^h$ with $F_0$ compact.
    The dimension of the abelian part 
    of the corresponding component
    of $X_I$ is $h$. This component is a toric variety
    bundle over an abelian $h$-fold, possibly self-glued, 
    where the toric variety has polytope $F_0$.
    \end{enumerate}
\end{proposition}

 \begin{proof}[Sketch]
The universal cover of the Mumford construction
is the toric variety $Y_\Gamma$ whose polytope is $\Gamma$,
and hence $\Gamma$ is the intersection complex
of this universal cover. Then the intersection complex
of the Mumford construction itself is the 
quotient by the $\bfM$ action, and the stated description 
follows---components with an abelian factor
of dimension $h>0$
arise from infinite
faces of $\Gamma$ 
stabilized by a rank $h$ subgroup of $\bfM$.
See also Theorem \ref{poly-is-fan}.
 \end{proof}

 \subsection{Comparison of polytope and fan constructions}\label{sec:comparison}

We explain why the polytope construction 
of the Mumford degeneration coincides with the fan 
construction.

\begin{theorem}\label{poly-is-fan}
Let
$\{\overline{b}_1,\dots,\overline{b}_k\}
\in H^0(\bT^g, \bZ{\rm  PL}/\bZ {\rm L})$ 
be a collection of convex sections, with
nontrivial bending in every direction, which are dicing. 
Define $\Gamma \subset \bfM_\bR \times \bR^k$ as in 
\eqref{overgraph}, and let $\mathcal S$ be the normal 
fan to $\Gamma$. 
Then, there is a canonical isomorphism of analytic spaces
\[
X({\overline{b}_1,\dots, \overline{b}_k}) \simeq X(\mathcal S)
\]
over $\Delta^k$, where $X({\overline{b}_1,\dots, \overline{b}_k}) 
\to \Delta^k$ is the polytope Mumford degeneration defined 
in \eqref{mumford:polytope} and where $X(\mathcal S) \to \Delta^k$ 
is the fan Mumford degeneration defined in \eqref{mumford:fan}. 
\end{theorem}

\begin{proof} Let $Y=Y_\Gamma$ be the locally finite
type toric variety defined by the 
polytope $\Gamma$, as in Section
\ref{sec:toric-varieties}. Then $Y=Y(\cS)$
by definition, and the $\bfM$-action 
(\ref{polytope-action}) on $\Gamma$, which we denote
${\bf m}\cdot -$, defines a linearization of the 
line bundle $\cO_Y(1)$ associated to
the polytope $\Gamma$.
A $\tfrac{1}{w}$-integral point
$(v,\vec{r})\in \Gamma\cap 
\tfrac{1}{w}(\bfM\times \bZ^k)$
defines an analytic section 
$(z,u)^{w(v,\vec{r})}\coloneqq
z^{wv} u^{w\vec{r}}\in H^0(Y, \cO_Y(w))$.
When this $\tfrac{1}{w}$-integral
point lies on the graph $\Gamma(b_1,\dots,b_k)$, 
we have the equality
\begin{align}\label{m-average}\Theta_{\overline{v}}(z,u) = 
\sum_{{\bf m}\in \bfM}(z,u)^{w({\bf m}
\cdot (v,\vec{r}))}\end{align} of analytic
functions, on the analytic open subset of $Y$
where all $|u_i|<1$. Convergence
holds because the $b_i$ having nontrivial
bending in all directions (see Definition \ref{dicing}),
and so with respect to an exhaustion
of $\bfM$, the powers of $u$ grow quadratically
while the powers of $z$ only grow linearly.
If, rather,
$(v,\vec{r})$ lies above the graph $\Gamma(b_1,\dots,b_k)$,
the corresponding sum over the $\bfM$-orbit
is simply a monomial in $u$ 
times $\Theta_{\overline{v}}(z,u)$.
We deduce that $\Theta_{\overline{v}}(z,u)$
descends, as an analytic section, to 
$H^0(X(\cS), \cL^{\otimes w})$ where
$\cL$ is the descent of the $\bfM$-linearized line 
bundle $\cO_Y(1)$ to the $\bfM$-quotient $X(\cS)$.

It suffices then to verify
that these descended sections define,
for some fixed $w$, a relatively very ample
line bundle on $X(\cS)$---in particular, 
that they separate points and tangents.
The argument is essentially the same as
\cite[Thm.~4.7]{alexeev-nakamura}, replacing
the Delaunay decomposition with the
more general decompositions 
$\bigcup_i {\rm Bend}(\overline{b}_i)$
that we consider. 

In fact, in the setting where all polytopal faces $F$
cut by $\bigcup_i {\rm Bend}(\overline{b}_i)$
are embedded as opposed to immersed in $\bT^g$, 
the multiplication rule (\ref{mult-rule}) for 
theta functions $\Theta_{\overline{v}}$ for
$\overline{v}\in F$, reduce, modulo the ideal 
$(u_1,\dots,u_k)$, to the usual multiplication
rule (Rem.~\ref{toricpolytope}) 
for the monomial sections of the powers of
the line bundle $\cL\vert_{Y_F}$ on the toric
stratum $Y_F\subset X(\cS)$ associated to $F$; 
see Lemma \ref{mult-rule-mod-u}.
Assuming $w$ is sufficiently large, we also ensure
that the non-normal union $X_0(\cS)=\lim_F Y_F$ 
is projectively
embedded via $\cL^{\otimes w}$. One deduces very ampleness
for all fibers, by the openness of very ampleness.
\end{proof}

For instance, we have the following special
case, for $1$-parameter degenerations:

\begin{corollary}
Let $\overline b \in H^0(\bT^g, \bZ{\rm  PL}/\bZ {\rm L})$ 
be dicing, with PL lift $b \colon \bfM_\bR \to \bR$ and define
$\Gamma \coloneqq \Gamma(b) + \bR_{\geq 0} \subset 
\bfM_\bR \times \bR$. Let $\mathcal S$ be the normal fan 
to $\Gamma$, and define $B \in {\rm{Sym}^2}\bfM^\vee$ by
$B({\bf m},{\bf m}') \coloneqq b({\bf m}+{\bf m}') - 
b({\bf m}) - b({\bf m}')+b(0)$. 
Define a $\Lambda_B$-invariant tiling $\mathcal T$ 
of ${\bfN}_\bR$ by slicing $\mathcal S$ at height $1$. 
Then 
$X({\rm{Cone}}(\mathcal T)) \simeq X(\mathcal S) 
\simeq X(\overline b).$ \hfill $\qed$
\end{corollary}

\subsection{Examples}\label{sec:examples} 
We now discuss some examples
of the Mumford construction.

\begin{example}\label{ex:tate}
The basic example is the Tate curve
$\bC^*/u^\bZ$. Here $g=1$, so 
${\bf M}_\bR/{\bf M}\simeq \bR/\bZ$. Define a PL function
$b\colon \bR\to \bR$ which is linear on each interval
$[m,m+1]$, $m\in \bZ$, and has values on $\bZ$ equal to 
$b(m) = \tfrac{1}{2}(m^2-m)$. The graph of $b$ is depicted
in the left of Figure \ref{fig:tate}.

\begin{figure}[ht]
\includegraphics[height=2.25in]{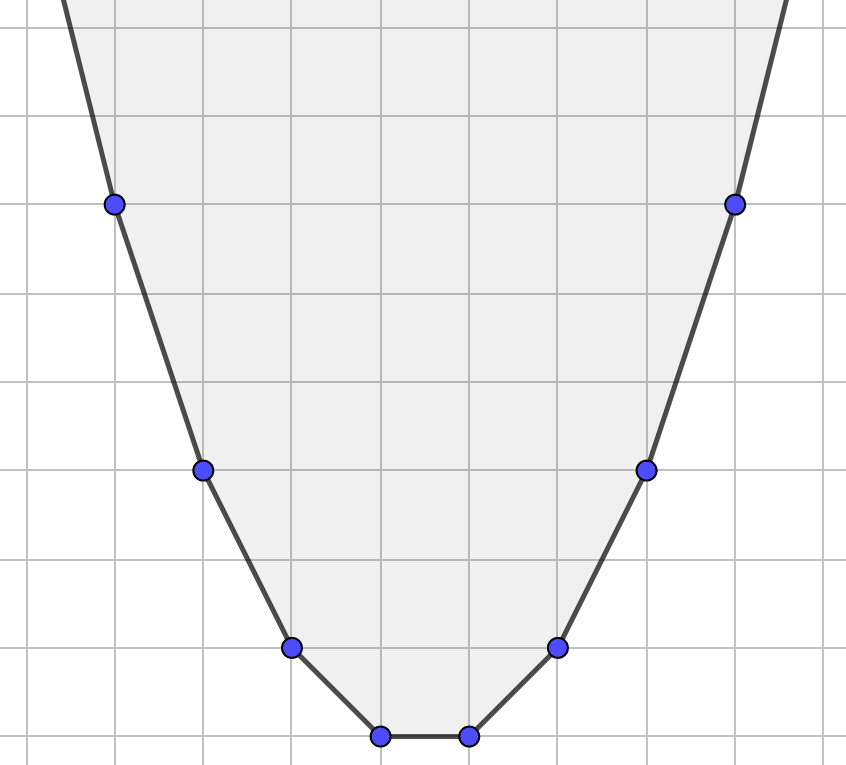}\hspace{10pt}
\includegraphics[height=2.25in]{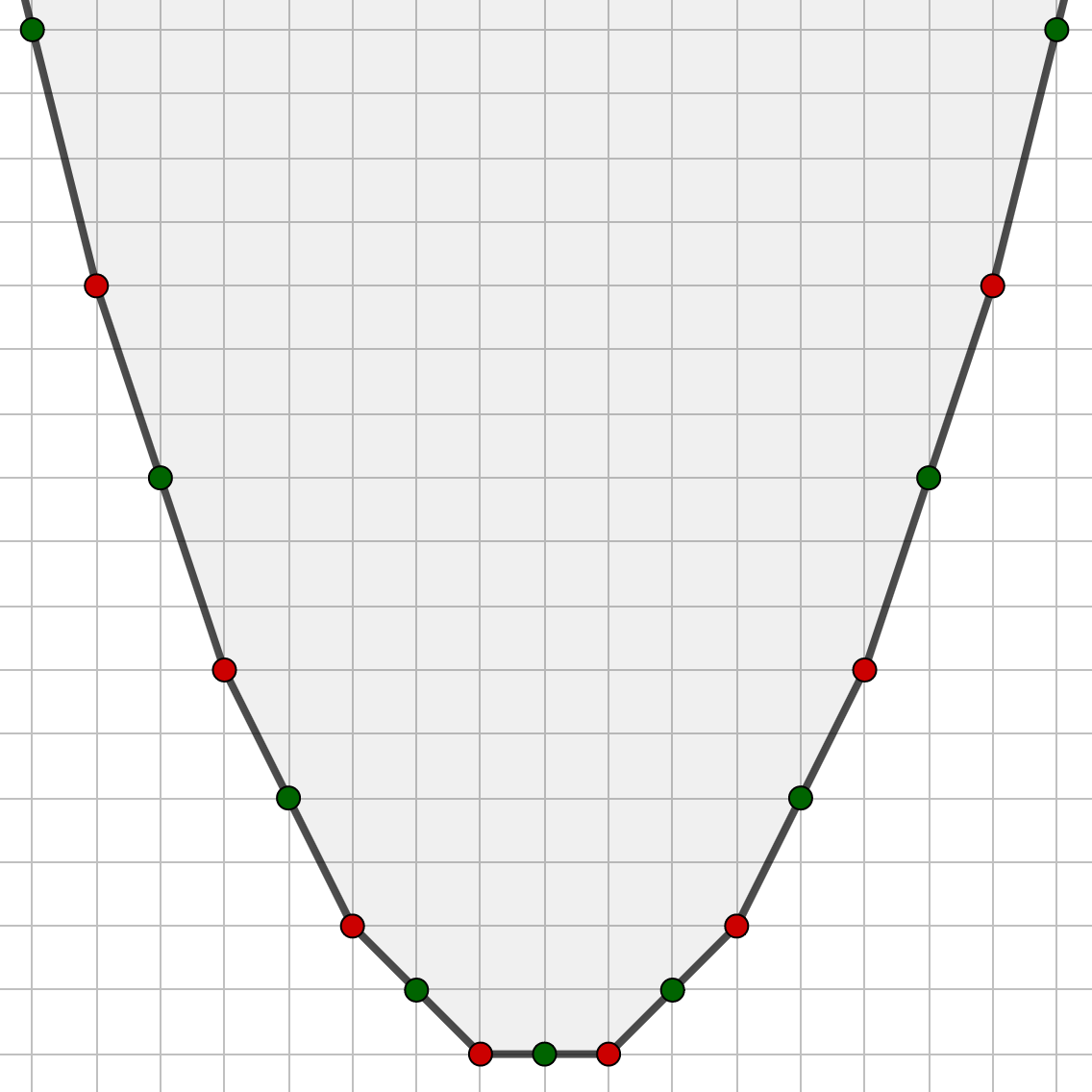}
\caption{Mumford polytope construction of the Tate curve,
$\Theta_{0/1}(z,u)$ in blue, $\Theta_{0/2}(z,u)$ in red, $\Theta_{1/2}(z,u)$ in green.}
\label{fig:tate}
\end{figure}

Then $\Gamma$ is the shaded region in Figure \ref{fig:tate}, and the convex hull of 
$\{(m, \tfrac{1}{2}(m^2-m)): m\in \bfM\}
\subset \bfM\times \bZ = \bZ^2$.
The theta functions of weight $1$, $2$, and $3$ are 
\begin{align*} \Theta_{0/1}(z,u)&= \cdots + z^{-3}u^{6} + z^{-2}u^{3}+ z^{-1}u^1+z^0u^0+ z^1u^0+z^2u^1+z^3u^3+\cdots \\
\Theta_{0/2}(z,u)&= \cdots + z^{-6}u^{12} + z^{-4}u^{6}+ z^{-2}u^2+z^0u^0+ z^2u^0+z^4u^2+z^6u^{6}+\cdots \\
\Theta_{1/2}(z,u)&= \cdots + z^{-5}u^{9} + z^{-3}u^{4}+ z^{-1}u^1+z^1u^0+ z^3u^1+z^5u^4+z^7u^9+\cdots \\
\Theta_{0/3}(z,u)&= \cdots + z^{-9}u^{18} + z^{-6}u^{9}+ z^{-3}u^3+z^0u^0+ z^3u^0+z^6u^{3}+z^9u^{9}+\cdots \\
\Theta_{1/3}(z,u)&= \cdots + z^{-8}u^{15} + z^{-5}u^{7}+ z^{-2}u^2+z^1u^0+ z^4u^1+z^7u^5+z^{10}u^{12}+\cdots \\
\Theta_{2/3}(z,u)&= \cdots + z^{-7}u^{12} + z^{-4}u^{5}+ z^{-1}u^1+z^2u^0+ z^5u^2+z^8u^7+z^{11}u^{15}+\cdots
\end{align*} 
with those of weight $1$ and $2$
depicted in Figure \ref{fig:tate} as the sum of the
blue, red, and green monomials. The normal
fan is depicted in Figure \ref{fig:tate2}
and the tiling $\cT$ (i.e.~the slice of the normal
fan at height $1$) is the tiling of $\bR^1/\Lambda_B$
by a segment of length $1$. Here 
$\Lambda_B\simeq \bZ$ 
because the bilinear form on $\bR^1$
defined by the formula 
$$B(x,y)= \tfrac{1}{2}(x+y)^2-\tfrac{1}{2}x^2-\tfrac{1}{2}y^2=xy$$ 
has $1\times 1$ Gram matrix $[1]\in {\rm Sym}_{1\times 1}(\bZ)$.

\begin{figure}[ht]
\includegraphics[width=4in]{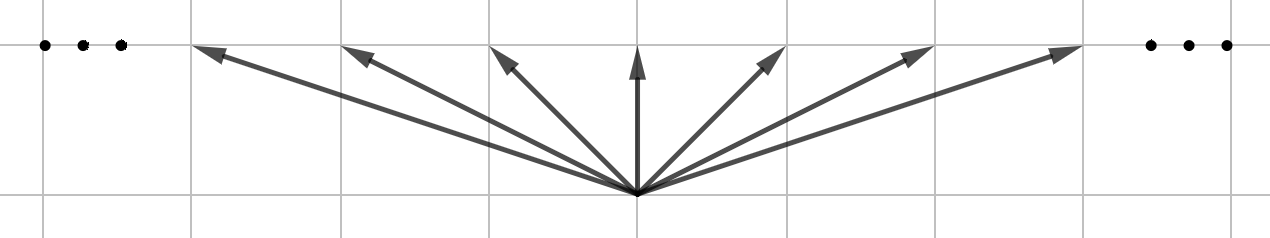}
\caption{Normal fan of the Tate curve.}
\label{fig:tate2}
\end{figure}
\end{example}

The torus $\bT^1=\bR/\bZ$ and the weighted
polyhedral complex ${\rm Bend}(b)$
inside it are depicted in Figure \ref{fig:tate3}
(the most condensed presentation of a 
Mumford construction). This figure happens to
be the same as the tiling $\cT$, but this
is a coincidence. By Proposition \ref{bending-read},
this polyhedral decomposition of $\bT^1$ is the 
intersection complex of the special fiber $X_0$ 
of the degeneration of elliptic curves 
$X = X(b) \to \Delta$, with strata formed
from the polytopes of the decomposition ${\rm Bend}(b)$.
Hence $X_0$ is $\bP^1$ glued to 
itself along two points, $0$ and $\infty$.

\begin{figure}[ht]
\includegraphics[height=1.2in]{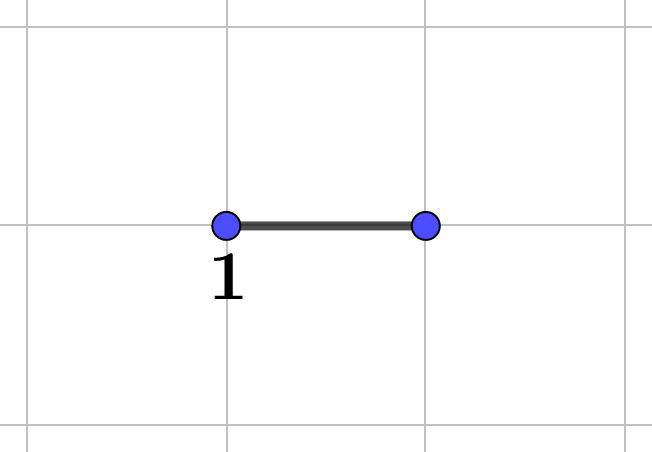}
\caption{Bending complex of the Tate curve in $\bR/\bZ$. 
The integer $1$ indicates the bending parameter, 
see Definition \ref{bendinglocus}.}
\label{fig:tate3}
\end{figure}

Figure \ref{fig:tate_quotient} is a visual
depiction of Construction \ref{mumford}. By considering
the maximal cones of the normal
fan of Figure \ref{fig:tate3}, we see that
the universal cover $Y({\rm Cone}(\cT))\to \Delta_u$ 
of the Tate curve
may be constructed as an infinite union of copies
of $\bC^2$: $$Y({\rm Cone}(\cT))=\bigcup_{n\in \bZ}\bC^2_{(x_n,y_n)}$$ where the gluings are $x_{n+1}=y_n^{-1}$ and $y_{n+1} = x_ny_n^2$.
The map to $\bC_u$ is given on local charts
by $u=x_ny_n$ and respects the gluings. Finally,
the $\bZ$-action is 
$(x_n,y_n)\mapsto (x_{n+1}, y_{n+1})$.

\begin{figure}[ht]
\includegraphics[width=6in]{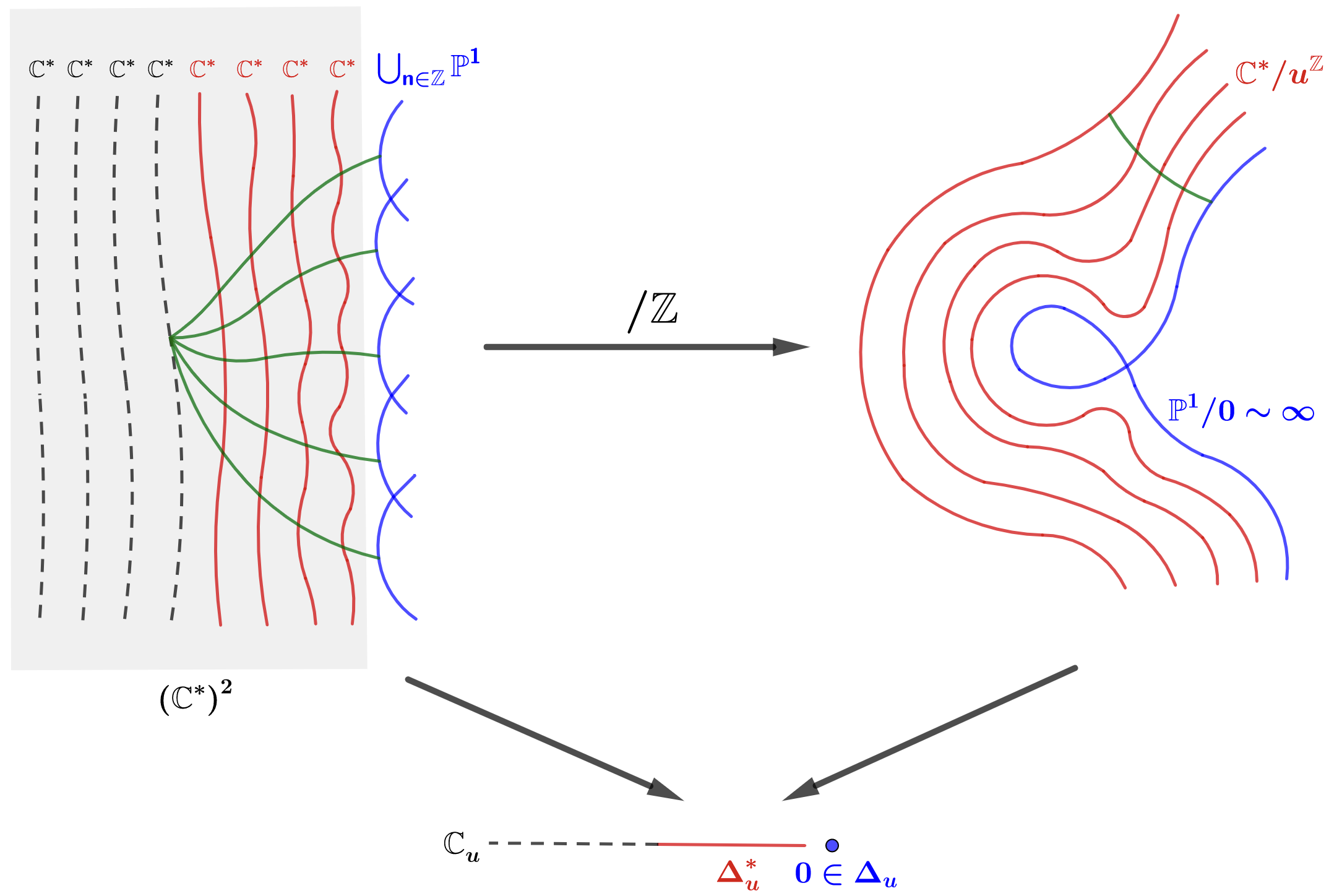}
\caption{Top left: Universal cover of the Tate curve.
Inverse image of $\Delta^*_u$ depicted in red, 
with embedding into $(\bC^*)^2$, in grey.
In the toroidal extension $Y({\rm Cone}(\cT))$,
the fiber over $0\in \Delta_u$ in blue is an
infinite $\bZ$-periodic quilt of toric varieties,
given by gluing an infinite chain of $\bP^1$s.
A $\bZ$-orbit of co-characters passing through
$(1,1)\in (\bC^*)^2$ and forming
sections over $\bC_u$ is depicted in green.
Top right: The Tate curve, with
general fiber $\bC^*/u^\bZ$ in red, central
nodal fiber in blue, and section in green.
}
\label{fig:tate_quotient}
\end{figure}

\begin{example}[Multiplication of theta functions]
In general the multiplication rule for theta functions
is quite complicated, but here we check that for
the central fiber $u = 0$ of the Tate curve, 
Example \ref{ex:tate}, the sections of $\cL^{\otimes 3}$
define an embedding of a nodal cubic $X_0(b)\hookrightarrow\bP^2$. 
In addition, the computation will show that the generic fiber 
of $\Proj_{\bC[[u]]} R(b) \to \Spec \bC[[u]]$, see \eqref{projfamily}, 
is a smooth cubic plane curve, as also follows from 
Theorem \ref{poly-is-fan}. See Remark \ref{genericfiberabelian}.

By quotienting by the ideal $(u)$, the multiplication
rule (\ref{mult-rule}) significantly simplifies.
Generally, a product of two monomials lies in the ideal
$(u_1,\dots,u_k)$, whenever they lie over distinct 
domains of linearity of the $b_i$ as then
their product, viewed as a lattice point in 
$\bfM\times \bZ^k$, lies strictly above the graph 
$\Gamma(b_1,\dots,b_k)$ by convexity. We deduce:

\begin{lemma}\label{mult-rule-mod-u} 
$\Theta_{\overline{v}_1}\cdot 
\Theta_{\overline{v}_2}=0\textrm{ mod }(u_1,\dots,u_k)$ 
whenever $\overline{v}_1$, $\overline{v}_2$ 
do not lie in any common polyhedral domain of 
$\bigcup_i {\rm Bend}(\overline{b}_i)\subset \bT^g$. 
Furthermore, if both $\overline{v}_1$, $\overline{v}_2$
lie in the interior
of a polyhedral domain of maximal dimension, then we have 
$$\Theta_{\overline{v}_1}\cdot \Theta_{\overline{v}_2} = 
\Theta_{\frac{w_1\overline{v}_1+w_2\overline{v}_2}{w_1+w_2}}
\,{\rm  mod}\,(u_1,\dots,u_k).$$
\end{lemma}

Otherwise, the multiplication rule mod $(u_1,\dots,u_k)$
must take into account the fact that there are might be
multiple representatives $v_1$, $v_2$ of 
$\overline{v}_1,\overline{v}_2$ which lie in the same
domain of linearity.
In any case, we may apply this comment
and Lemma \ref{mult-rule-mod-u}
to the weight $3$ theta functions of Example \ref{ex:tate}.
We deduce the following multiplication rules mod $u$:
\begin{align*} 
\Theta_{0/3}^3&=
\Theta_{0/9}+3\Theta_{3/9}+3\Theta_{6/9} 
&&\Theta_{1/3}^3= \Theta_{3/9} 
\\
\Theta_{0/3}^2\Theta_{1/3}&=
\Theta_{1/9}+2\Theta_{4/9}+\Theta_{7/9} 
&& \Theta_{1/3}^2\Theta_{2/3}=
\Theta_{4/9}
&&& \Theta_{0/3}\Theta_{1/3}^2=
\Theta_{2/9}+\Theta_{5/9}
\\
\Theta_{0/3}^2\Theta_{2/3}&=
\Theta_{2/9}+2\Theta_{5/9}+\Theta_{8/9} 
&& \Theta_{1/3}\Theta_{2/3}^2=
\Theta_{5/9}
&&&\Theta_{0/3}\Theta_{2/3}^2=
\Theta_{4/9}+\Theta_{7/9}
\\
\Theta_{0/3}\Theta_{1/3}\Theta_{2/3} &= 
\Theta_{3/9}+\Theta_{6/9}
&&\Theta_{2/3}^3=
\Theta_{6/9}
\end{align*}
as only $0/3\in \tfrac{1}{3}\bfM/\bfM\subset
\bR/\bZ$ lies in multiple
domains of linearity of $\overline{b}$.

These are the ten cubics 
in ${\rm Sym}^3H^0(X_0(b),\cL^{\otimes 3})$,
and there is indeed one linear relation:
$$\Theta_{0/3}\Theta_{1/3}\Theta_{2/3} = 
\Theta_{1/3}^3+\Theta_{2/3}^3\textrm{ mod }u,$$
which gives the projective equation 
$\{xyz=x^3+y^3\}\subset \bP^2$ of a cubic curve,
with a simple node at $[0:0:1]$. Computing
the expansion of theta products
rather over $\bC[[u]]/(u^2)$,
one additionally sees that the local equation
of the node in the total space is $xy=u$ and hence the general
fiber of $X(b)\to \Delta$ is a smooth elliptic curve.
\end{example}

\begin{example}[Theta graph, $1$-parameter]\label{theta-1-param}
Now consider the $1$-parameter Mumford degeneration
$(\bT^2, b)$
corresponding to the PL
function $$b(m_1,m_2) = r_1\frac{m_1^2-m_1}{2}+
r_2\frac{m_2^2-m_2}{2}+r_3\frac{(m_1+m_2)^2-(m_1+m_2)}{2}$$
on $\bfM=\bZ^2$ where $(r_1,r_2,r_3)\in \bN^3$ is some
fixed vector. Then $b$ is convex and the
boundary of $\Gamma$ is depicted in Figure
\ref{fig:tate4}.

\begin{figure}[ht]
\includegraphics[height=3.3in]{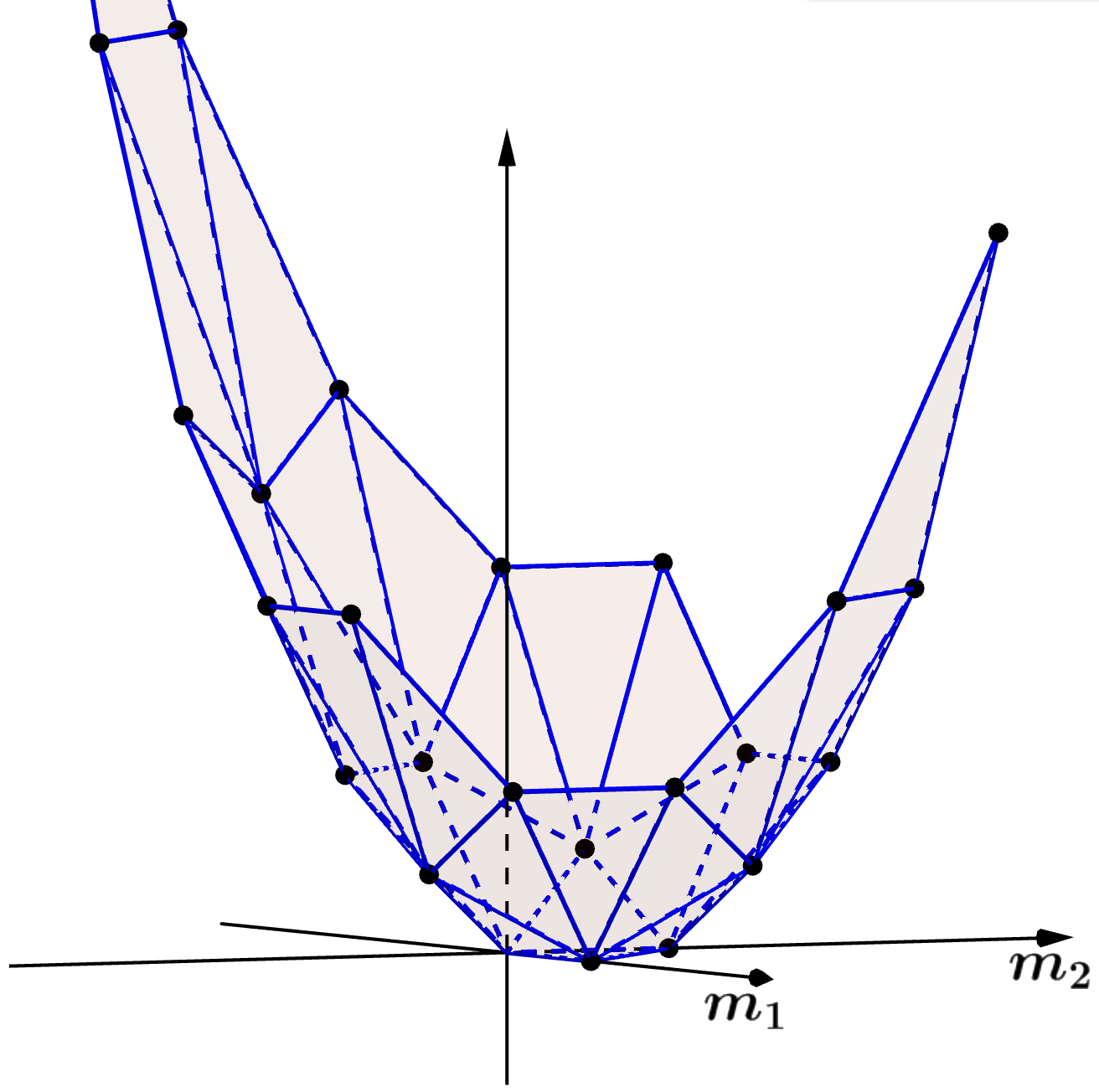}
\caption{The graph of $b$ over $\bfM_\bR= \bR^2$, 
for values
$r_1=r_2=r_3=1$.}
\label{fig:tate4}
\end{figure}

\begin{figure}[ht]
\includegraphics[height=1.5in]{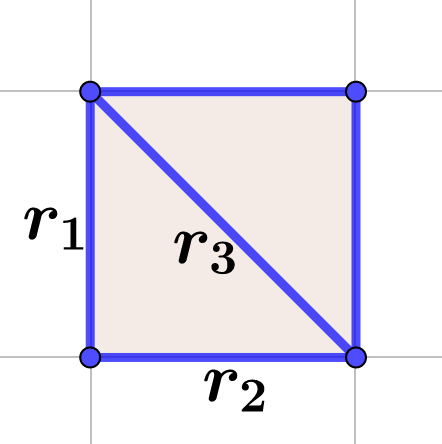}
\caption{Bending complex of $b$ 
in $\bT^2$. 
}
\label{fig:tate5}
\end{figure}

\begin{figure}[ht]
\includegraphics[height=3.5in]{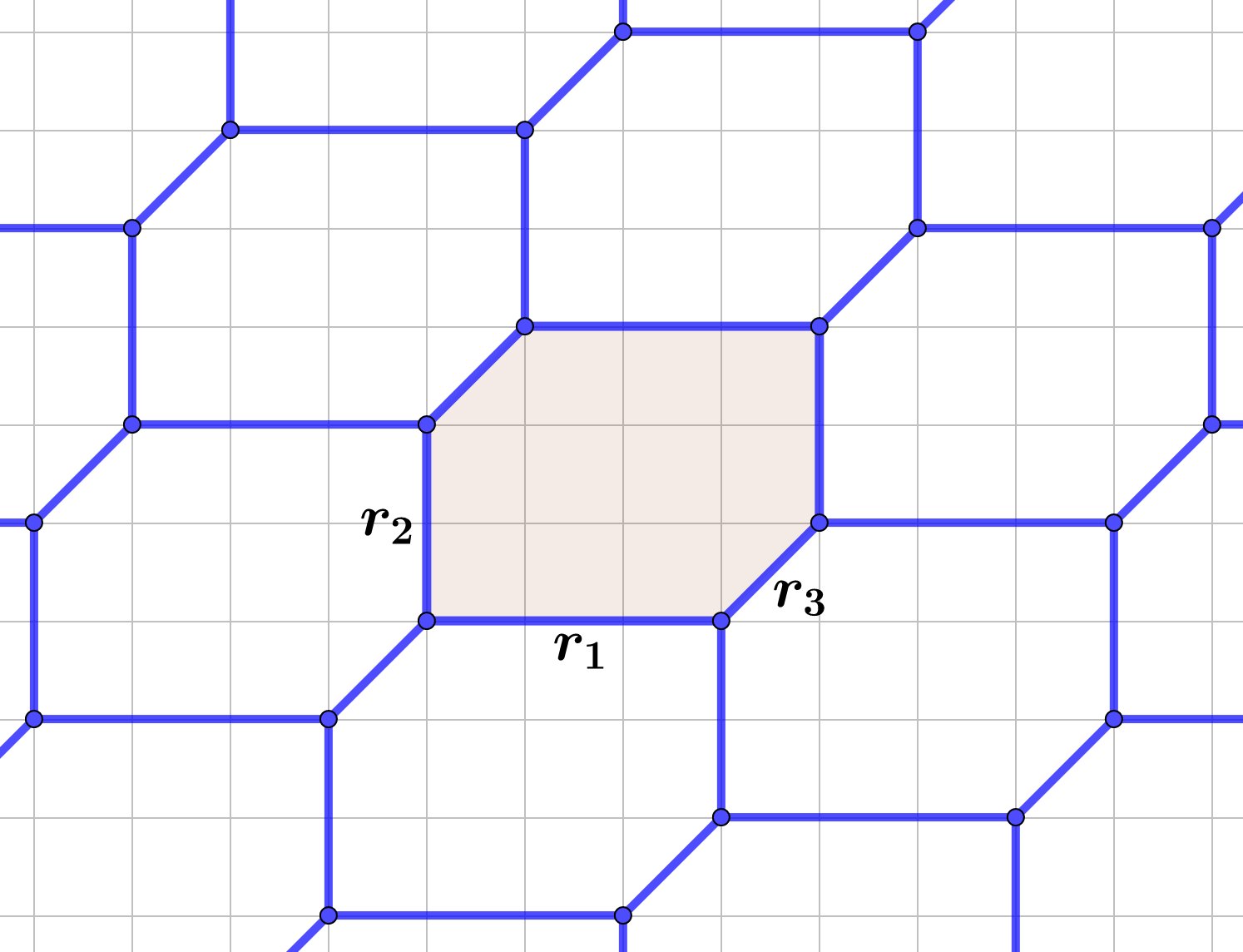}
\caption{Fundamental domain for the 
$\Lambda_B$-action on $\bR^2$, 
$(r_1,r_2,r_3)=(3,2,1)$,
and tiling
$\cT$, arising from slicing
the normal fan at height $1$.}
\label{fig:tate6}
\end{figure}

We have, for instance, the weight $1$ theta function
\begin{align*}\Theta_{(0/1,\,0/1)}(z_1,z_2,u)=\cdots &+ 
 z_1^{-1}z_2^{1}\,\,\,u^{r_1} &&+z_1^0z_2^1\,\,\,u^0 &&+z_1^{1}z_2^1\,\,\,u^{r_3} \\
&+z_1^{-1}z_2^0\,\,\,u^{r_1+r_3} &&+z_1^0z_2^0\,\,\,u^0 &&+z_1^1z_2^0\,\,\,u^0 \\
&+z_1^{-1}z_2^{-1}u^{r_1+r_2+3r_3} &&+z_1^0z_2^{-1}u^{r_2+r_3} &&+z_1^1z_2^{-1}u^{r_2}+\cdots \end{align*}
e.g.~since $b(-1,-1)=r_1+r_2+3r_3$. The bending complex 
of $b$ is depicted in Figure
\ref{fig:tate5}. The associated bilinear form 
has Gram matrix $$B = r_1\twobytwo{1}{0}{0}{0}+
r_2\twobytwo{0}{0}{0}{1}+
r_3\twobytwo{1}{1}{1}{1}$$
and the tiling $\cT$ 
is depicted in Figure \ref{fig:tate6}.

To produce the universal $1$-parameter
Mumford degeneration, requires twisting
the construction by 
$a\in {\rm Sym}_{2\times 2}(\bC^*)/u^B\simeq (\bC^*)^2$.
\end{example}

\begin{example}[Theta graph, $3$-parameter]\label{theta-3-param}
We now modify the previous example, by
instead taking a $3$-parameter Mumford degeneration
for $(\bT^2, b_1, b_2, b_3)$
where
$$b_1(m_1,m_2)\coloneqq \frac{m_1^2-m_1}{2},\quad 
b_2(m_1,m_2)\coloneqq \frac{m_2^2-m_2}{2},\quad 
b_3(m_1,m_2)\coloneqq \frac{(m_1+m_2)^2-(m_1+m_2)}{2}.$$
This example is originally attributed
to Deligne \cite[Sec.~7]{mumford1972}; 
called by Mumford the ``keystone'' of the 
compactification of $\cA_2$. 
The figure is similar to Figure \ref{fig:tate4}, 
but we now use three different colors, to indicate the different
bending loci ${\rm Bend}(b_i)$, for $i=1,2,3$.
See Figure \ref{fig:tate7}.

\begin{figure}[ht]
\includegraphics[height=1.5in]{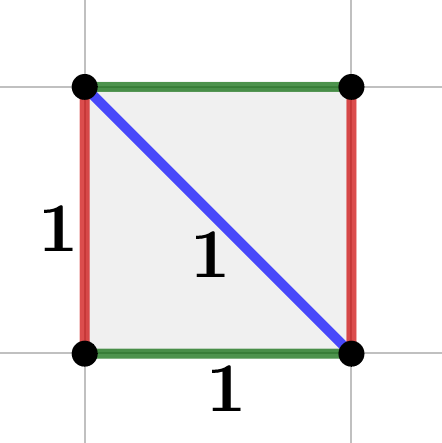}
\caption{Bending complexes of 
$b_1$, $b_2$,
$b_3$
in $\bT^2$, in red, green, and blue, respectively. 
The integers are the bending parameters of $b_1, b_2$, and $b_3$.}
\label{fig:tate7}
\end{figure}

The theta function of weight $1$ is
\begin{align*}\Theta_{(0/1,\,0/1)}(z_1,z_2,u_1, u_2, u_3)=\cdots 
& + z_1^{-1}z_2^{1}u_1^1u_2^0u_3^0 
 &&+z_1^0z_2^1u_1^0u_2^0u_3^0 
 &&+z_1^{1}z_2^1u_1^0u_2^0u_3^1 \\
&+z_1^{-1}z_2^0u_1^1u_2^0u_3^1 
&&+z_1^0z_2^0u_1^0u_2^0u_3^0 &&+z_1^1z_2^0u_1^0u_2^0u_3^0 \\
&+z_1^{-1}z_2^{-1}u_1^1u_2^1u_3^3 
&&+z_1^0z_2^{-1}u_1^0u_2^1u_3^1 &&+z_1^1z_2^{-1}u_1^0u_2^1u_3^0+\cdots.\end{align*}
Note that, upon restriction to the co-character 
${\rm Spec}\,\bC[[u]]\subset {\rm Spec}\,\bC[[u_1,u_2,u_3]]$
defined by $(u_1,u_2,u_3)=(u^{r_1},u^{r_2},u^{r_3})$ 
with $(r_1,r_2,r_3)\in \bN^3$, 
we get $\Theta_{(0/1,\,0/1)}(z_1,z_2,u)$ 
from Example \ref{theta-1-param},
and indeed, the multivariable Mumford 
construction restricts
to a $1$-parameter one along this 
co-character, and 
$B= r_1B_1+r_2B_2+r_3B_3$.

To understand the fibers of 
$X(b_1,b_2,b_3)=X\to \Delta^3$ 
over the various coordinate subspaces, we 
refer to Figure \ref{delaunayv2}.
By Proposition \ref{bending-read}, the 
polytopes of the components of
the fiber over $(u_1,u_2,u_3)\in \Delta^3$ 
of the Mumford construction can be read
off from the bending 
loci of the $b_i$
for which $V(u_i)=0$.
The fiber over the origin is the union 
$$X_{(0,0,0)}=
\bP^2\cup_{\triangle} \bP^2$$
of two copies of $\bP^2$ along a triangle of lines,
so that the intersection complex is the upper left
of Figure \ref{delaunayv2}.
The limit of the theta divisor is the union
of two lines $$V(\Theta_{(0/1,\,0/1)})\cap X_{(0,0,0)}=
\ell_1\cup \ell_2\subset \bP^2\cup_\triangle \bP^2.$$

The fiber $X_{(0,0,u_3)}$ over a point on the $u_3$-axis
is normalized by the square $(\bP^1\times \bP^1,\square)$ 
and results from gluing two
sections (the top and bottom of the square) and two fibers 
(the left and right of the square). The gluing isomorphisms
are $u_3\in \bC^*$, $u_3^{-1}\in \bC^*$. The theta divisor
$V(\Theta_{(0/1,\,0/1)})$ lies in the linear system of $\cO(1,1)$
on the normalization,
and glues to a Cartier divisor on the 
non-normal surface $X_{(0,0,u_3)}$.
The fibers over the $u_1$- and $u_2$-axes of $\Delta^3$
are similar; the intersection complexes are
in the top row of Figure \ref{delaunayv2}.

\begin{figure}[ht]
\includegraphics[width=5in]{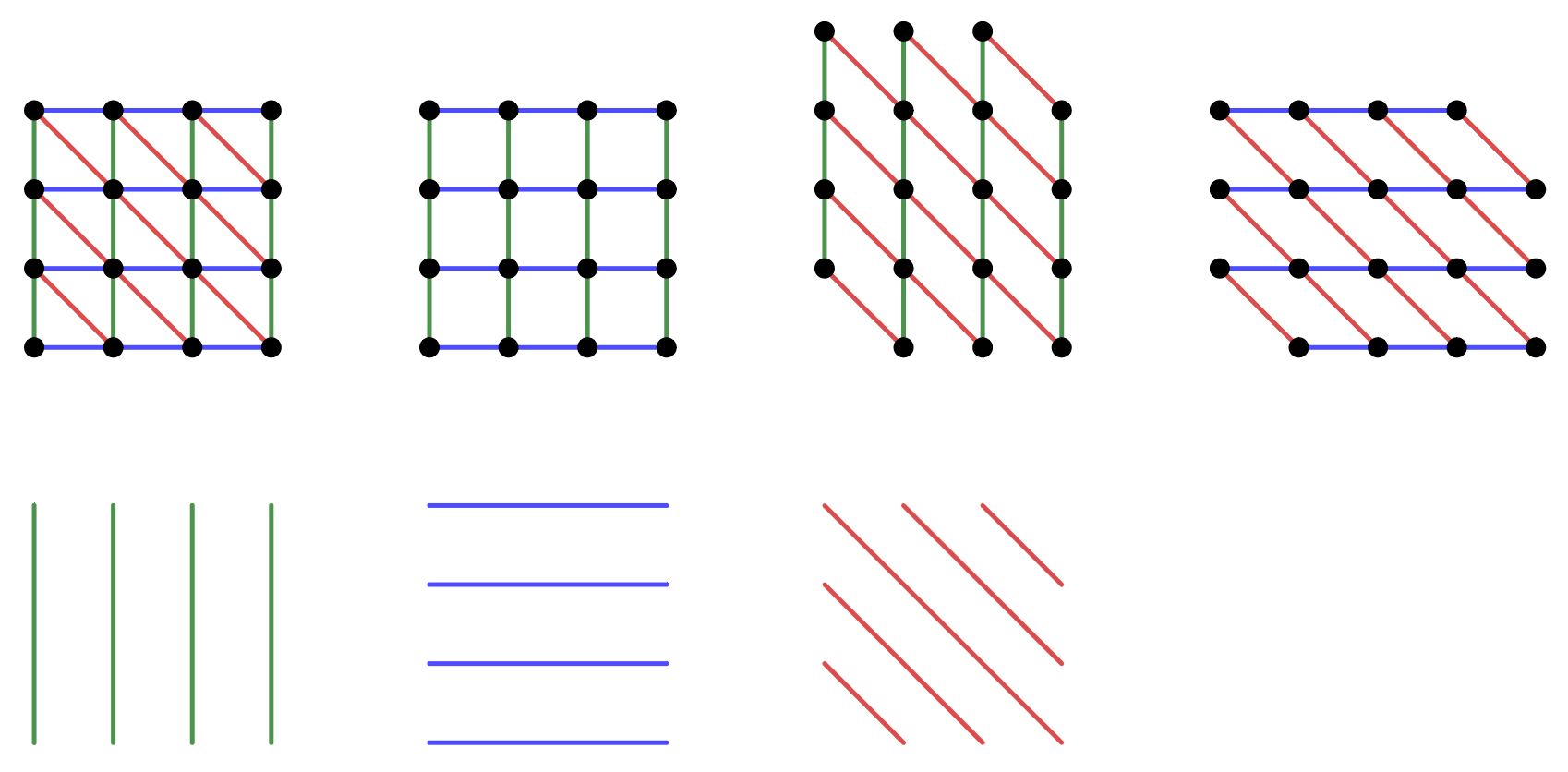}
\caption{Intersection complexes of the 
fibers of the Mumford construction.}
\label{delaunayv2}
\end{figure}

Over a general point of a coordinate hyperplane $V(u_1)$,
the fiber $X_{(0,u_2,u_3)}$ is the result of gluing
a $\bP^1$-bundle $\bP_E(\cO\oplus \cM)$ over an elliptic
curve $E$ to itself, $\cM\in {\rm Pic}^0(E)$,
by gluing the $0$- and $\infty$-sections
of the bundle by a translation
depending on $\cM$. 
Fibers over the other coordinate hyperplanes
are similar, and the corresponding intersection complexes 
are the first three figures in the second row 
of Figure \ref{delaunayv2}.
Finally, the general fiber $X_{(u_1,u_2,u_3)}$ over a
point $(u_1,u_2,u_3)\in (\Delta^*)^3$ is a smooth, principally
polarized abelian surface with $V(\Theta_{(0/1,\,0/1)})$ 
the theta divisor.
\end{example}

\subsection{Base change and Veronese embedding}\label{sec:base-change}

Let $X(b_1,\dots, b_k)\to \Delta^k$ be a Mumford construction
on a collection of convex sections $b_i\in 
H^0(\bT^g, \bZ{\rm PL}/\bZ{\rm L})$,
see Construction \ref{mumford-polytope},
and consider a monomial
base change, of the form $\Delta^{n}\to \Delta^k$,
such that the pullback of coordinates $u_i$ are of the form
\begin{align}\begin{aligned} \label{monomial-sub}
u_1 &= w_1^{r_{11}}\cdots w_n^{r_{1n}}=: w^{\vec{r}_1}, \\ 
 & \cdots \\
u_k &= w_1^{r_{k1}}\cdots w_n^{r_{kn}}=: w^{\vec{r}_k}. 
\end{aligned}\end{align}

\begin{proposition} The base change of the Mumford
degeneration $X(b_1,\dots,b_k)\to \Delta^k$ 
 along a monomial morphism
 $\Delta^n\to \Delta^k$ 
 is the Mumford construction $X(c_1,\dots,c_n)\to \Delta^n$ 
 associated to the convex PL functions 
 $$c_j\coloneqq 
r_{1j}b_1+\cdots+r_{kj}b_k.$$
\end{proposition}

\begin{proof} Substituting (\ref{monomial-sub}) into
the defining equations for $\Theta_{\overline{v}}$
in (\ref{theta-def}), we see that the result
is again a Mumford degeneration $X(c_1,\dots,c_n)\to \Delta^n$
where $c_j$ has the stated formula.
\end{proof}

A simple case is exhibited by Examples 
\ref{theta-1-param}, \ref{theta-3-param},
where we make the monomial
base change $\Delta\to \Delta^3$,
$w\mapsto (w^{r_1}, w^{r_2}, w^{r_3})$ to 
Example \ref{theta-3-param},
to get Example \ref{theta-1-param}.

\begin{remark}\label{ram-base}
A base change $\Delta^k\to \Delta^k$ ramified
over the coordinate hyperplane $V(u_i)$ 
to order $r_i$ is given
by $u_i=w_i^{r_i}$ and we have the simpler relation
$c_i=r_ib_i$.
\end{remark}

We now consider the effect of replacing $L$
with $dL$ for some positive multiple $d>0$ 
of the principal polarization. Equivalently, 
we are taking the Veronese subring 
$$R(b_1,\dots,b_k)^{(d)}\subset 
R(b_1,\dots,b_k)
$$ consisting of the theta 
functions of weights
$w$ divisible by $d$. 
The passage to the Veronese subring
suggests also a natural generalization
of Construction \ref{mumford-polytope},
which allows us to relax
the restrictive dicing condition: 

\begin{construction} \label{veronese}
Consider convex sections 
$$\overline{b}_i\in 
H^0(\bT^g, \,\tfrac{1}{d}\bZ{\rm PL}/
\tfrac{1}{d}\bZ {\rm L})$$
for some positive integer $d$, see Definition
\ref{dicing}. 
Then, we may define theta functions similarly
to formula (\ref{theta-def}), 
but only for the weights $w$ 
divisible by $d$.
Assuming that the $\overline{b}_i$
are $\tfrac{1}{d}$-dicing on $\bT^g$ 
(Def.~\ref{def:d-dicing}), we may 
deduce from Remark \ref{d-dicing} 
that $dL$ lifts to an ample line bundle on
the resulting degeneration, 
which we denote by 
$$X(d\mid b_1,\dots,b_k)\to \Delta^k.$$
While the general fiber of the degeneration
still admits a principal
polarization $L$, only $dL$ extends to a line bundle
on the total space, in general.

In terms of the polytope, we still take
the overgraph $\Gamma = \Gamma(b_1,\dots,b_k)+(\bR_{\geq 0})^k$,
which is now only a $\tfrac{1}{d}(\bfM \times \bZ^k)$-integral
polyhedron. Then, as in Remark \ref{toricpolytope},
we consider ${\rm Cone}(\Gamma)\subset 
(\bfM_\bR \times \bR^k)\times \bR$ but 
one only considers monomials, and
their $\bfM$-averagings (\ref{m-average}) to theta functions,
lying in $(\bfM\times \bZ^k)\times d\bZ$.
The general fiber of the Mumford construction
is still the quotient by $\bfM$ of $\bfN\otimes \bC^*$;
in particular, the exact sequence
(\ref{weight-exact}) still holds.\hfill $\clubsuit$ 
\end{construction}

\begin{remark}\label{rem-shifts} The isomorphism type of 
the degeneration $X(d \mid b_1,\dots,b_k)\to \Delta^k$ 
does not depend on the lifts $b_i$
of $\overline{b}_i$ but the choice of origin section does, 
since different lifts of $\overline{b}_i$ shift 
the normal fan, and thus
affect which subtorus forms the origin section, see Remark
\ref{section-location}. The same applies to Construction
\ref{mumford-polytope-2}---to produce an extension 
of the universal family over $\widetilde{\cA}_g^{\,\bB}$
requires a choice of lift of cone $\mathbbm{b}\subset H^0(\bT^g, \tfrac{1}{d}\bZ{\rm PL}/\tfrac{1}{d}\bZ{\rm L})$ into 
$H^0(\bfM_\bR, \tfrac{1}{d}\bZ{\rm PL})$
as the gluing with the universal family $\widetilde{\cX}_g\to \widetilde{\cA}_g$ 
depends on a choice of origin section.
\end{remark}

\begin{example}[Base change and resolution 
of the Tate curve]
\label{tate-base}

Consider the following Mumford constructions
of degenerating elliptic curves:

\begin{enumerate}
    \item The Tate curve, i.e.~Example \ref{ex:tate}.
    \item\label{2} The order $3$ base change 
    $u=w^3$ to Example \ref{ex:tate}.
    \item The order $3$ Veronese embedding of (\ref{2}).
\end{enumerate}

These are encoded respectively by the following
data:

\begin{enumerate}
    \item $\bR/\bZ$, $b^{(1)}=\tfrac{1}{2}(m^2-m)$ 
    on $m\in \bZ$, $d=1$.
    \item $\bR/\bZ$, $b^{(2)}=\tfrac{3}{2}(m^2-m)$ 
    on $m\in \bZ$, $d=1$.
    \item $\bR/\bZ$, 
    $b^{(3)} = b^{(2)}$, $d=3$ 
    (see the outer edge of Figure \ref{fig:tate-base}).
\end{enumerate}

The total space of the Tate curve
$X(b\,\!^{(1)})$ is smooth 
and the central fiber $X_0(b\,\!^{(1)})$ is an
irreducible nodal curve (i.e.~of Kodaira type $I_1$).
The total space of the base
change $X(b\,\!^{(2)})$ is singular, with 
an $A_2$-singularity at the node point
of the irreducible central fiber.
The spaces $X(3\mid b\,\!^{(2)})\simeq X(b\,\!^{(2)})$ 
are isomorphic degenerations
over $\Delta$, with the former polarized by $3L$
rather than $L$.

\begin{figure}
\includegraphics[width=3.5in]{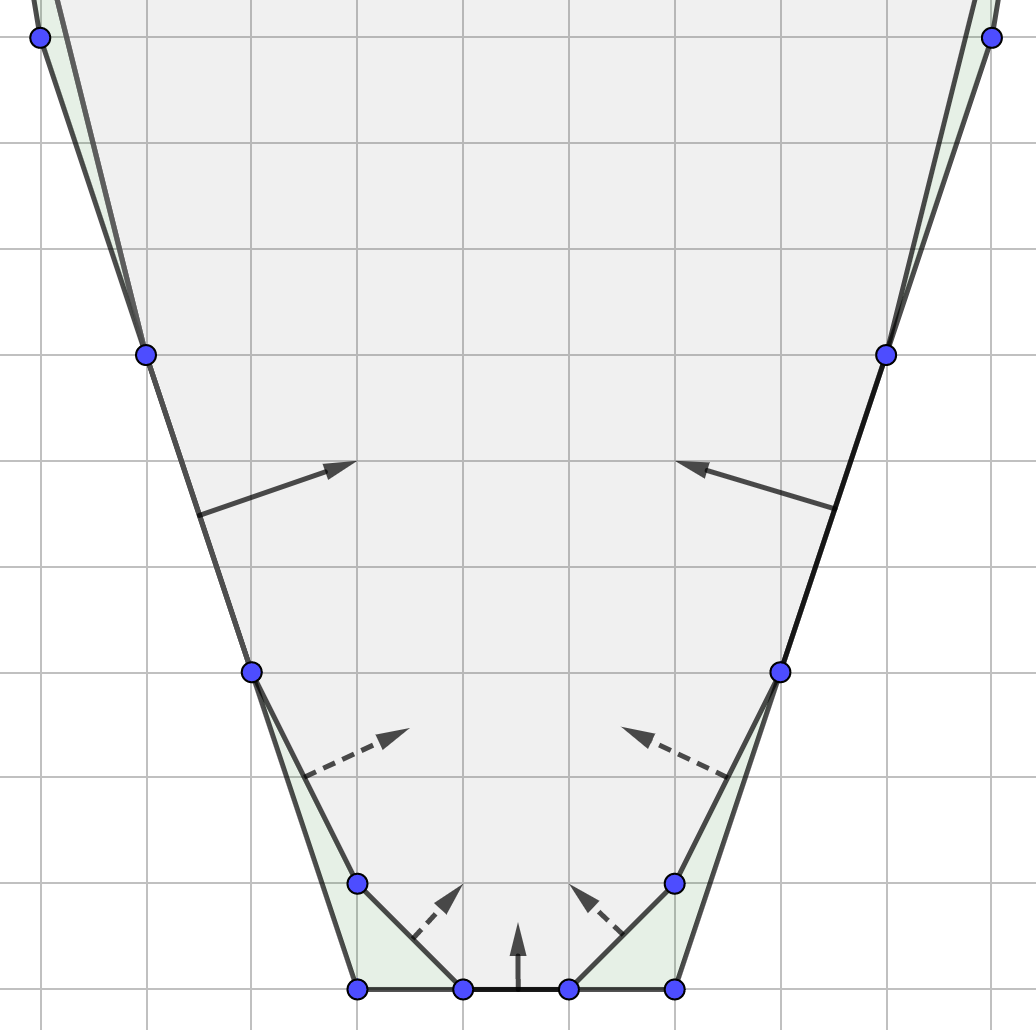}
\caption{Above:
order $3$ base change of the Tate curve,
with polarization $\cO(3)$, and resolution. Following
Remark \ref{veronese}, grid points are $(\tfrac{1}{3}\bZ)^2$.
Below: normal fan of the order $3$ base change, and
of the resolution.}

\includegraphics[width=4in]{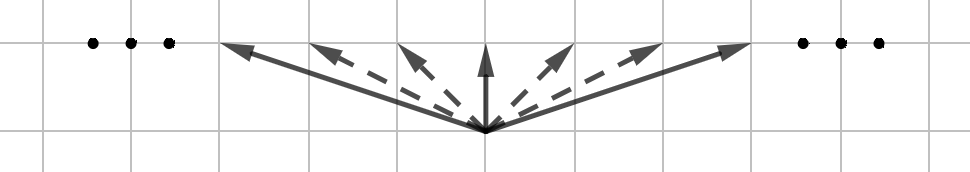}
\label{fig:tate-base}
\end{figure}

To resolve the total space 
$X(b\,\!^{(2)})$, we take
the minimal resolution $\wX\to
X(b\,\!^{(2)})$, which resolves
the $A_2$-singularity to a chain of two
$(-2)$-curves $E_1+E_2$. Then the central
fiber $\wX\to \Delta$ is an $I_3$-type Kodaira
fiber, i.e.~a wheel of $\bP^1$s of length $3$. 

To realize this resolution
as a {\it polarized} resolution, first,
we pull back $3L$ to $\wX$. It has multidegree
$(3,0,0)$ on the three components of the wheel.
Now, we twist, defining $$\wL \coloneqq 3L-E_1-E_2.$$
The resulting line bundle has multidegree
$(1,1,1)$ on the wheel,
and the overgraph of the
Mumford degeneration defining 
$(\wX,\wL)\to \Delta$ is shown in gray in
Figure \ref{fig:tate-base}.

On the $3\bZ$-prequotient,
this is the usual blow-up operation on polytopes
of polarized toric varieties,
which cuts a corner off the polytope,
whose size depends on the chosen polarization on
the blow-up. The cut corners 
are depicted in green in Figure \ref{fig:tate-base}. 
The normal fans are depicted in the bottom of 
Figure \ref{fig:tate-base}. We see that,
indeed the normal fan for $\wX = 
X(3\mid \tfrac{1}{3}b\,\!^{(1)}(3x))$ 
is a $3\bZ$-invariant refinement
of the normal fan to $X(b\,\!^{(2)})$.
Thus, there is a $3\bZ$-equivariant toric 
morphism between the corresponding toric varieties,
descending on quotients to give the minimal resolution.

\begin{figure}
\includegraphics[width=5in]{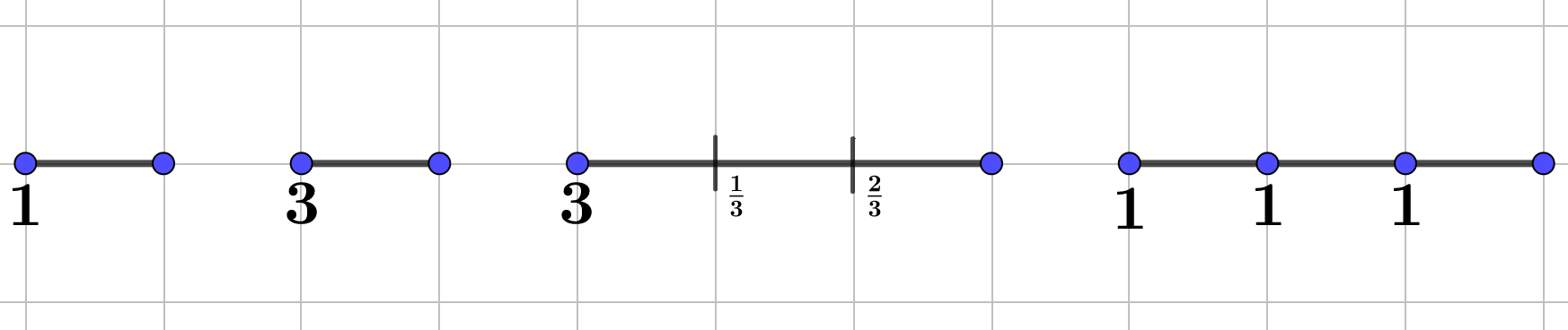}
\caption{Left-to-right: the Tate curve,
the order $3$ base change of the Tate curve,
the order $3$ base change of the Tate curve plus
$3$rd Veronese embedding, and its polarized resolution. 
Integers at blue vertices are the bending parameters.}
\label{fig:tate-resolve}
\end{figure}

In terms of $\bT^1$  and the
bending loci, the various operations 
are depicted in Figure \ref{fig:tate-resolve}.
\end{example}

\section{Regular matroids}
\label{sec:regular-matroids}

We now discuss the construction of Mumford degenerations
associated to regular matroids.

\subsection{Matroids, graphs,
and quadratic forms}\label{sec:matroid}

\begin{definition}\label{matroid-def} A {\it matroid} 
$\uR= (\underline{R},E)$
is a finite set $E$, together with a collection $\underline{R}$
of {\it independent subsets} of $E$, satisfying
the following axioms:

\begin{enumerate}
    \item The empty set is independent.
    \item Any subset of an independent set is independent.
    \item If $I,J\subset E$ are independent sets with $|I|>|J|$, 
    then there is an element $i\in I\setminus J$ for which 
    $J\cup \{i\}$ is independent.
\end{enumerate}
The set $E$ is called the {\it ground set} of $\uR$. 
\end{definition}

These axioms encapsulate the concept of linear independence of a 
collection of vectors in a vector space.
A {\it basis} is a maximal independent subset $E'\subset E$, 
and a {\it circuit} is a minimal dependent set. Note that 
all bases of $\underline{R}$ have the same 
cardinality by Def.~\ref{matroid-def}(3); 
this cardinality is called the \emph{rank} of 
$\underline{R}$. 

The {\it dual matroid} 
$\underline{R}^*$ is a matroid on the same
ground set $E$, whose bases are the complements
of bases of $\underline{R}$. A circuit of
$\underline{R}^*$ is a {\it cocircuit} of $\underline{R}$. 

\begin{definition} \label{def:regular}
A {\it realization} of $\underline{R}$ over
the field $\bF$ is a map $\phi\colon E\to \bF^g$ to an $\bF$-vector 
space for which the independent sets $E'\subset E$ are 
exactly those for which $\{\phi(i)\}_{i\in E'}$ are linearly independent.
A matroid $\underline{R}$ is {\it regular}
if it admits a realization over any field. 
An {\it integral realization}
of a regular matroid is a map $\phi\colon E\to \bfN$ 
to a free $\bZ$-module $\bfN$ 
which gives a realization of $\underline{R}$ upon
base change $\bfN\otimes_\bZ \bF$ to any field $\mathbb F$. 
\end{definition}

We will always assume that $\phi(E)$ generates the lattice
$\bfN$. In particular, the rank of $\uR$ agrees
with the rank of $\bfN$.

By a theorem of Tutte \cite{tutte}, 
every regular matroid can be defined
by a {\it totally unimodular matrix}, that is, a matrix
all of whose minors (in particular, all entries)
have determinant in $\{\pm 1,0\}$. 
Then, an integral realization of the matroid arises
by considering the set of column vectors. 
 More generally, any
{\it unimodular matrix}---an integer entry matrix whose
maximal minors have determinant in $\{\pm 1,0\}$---defines
an integral realization of a regular matroid
\cite[Ch.~3, Thm.~3.1.1]{White_1987}.
Equivalently, the lattice 
spanned by any collection of columns is saturated.

\begin{example}\label{cographic}
Let $G$ be a graph and let $E=E(G)$ be its
set of edges. Choose an orientation on the edges.
We have an inclusion $H_1(G,\bZ)\subset \bZ^E$
as every homology class $\gamma\in H_1(G,\bZ)$ can be viewed
as a $\bZ$-linear combination of directed edges. 

Let $e_i$ denote 
the basis vector of $\bZ^E$ corresponding 
to the $i$-th edge and let 
$\widetilde{\bf x}_i\coloneqq e_i^\vee\in (\bZ^E)^\vee$ 
be the corresponding coordinate
function. By restriction, we get a linear function
${\bf x}_i\in H_1(G,\bZ)^\vee\simeq H^1(G,\bZ).$
The {\it cographic
matroid} $M^*(G)$ of $G$, on the ground set $E$, 
has realization \begin{align*} E&\to H^1(G,\bZ), 
\\ i&\mapsto {\bf x}_i.\end{align*}
The {\it graphic matroid} 
$M(G)$ is the dual matroid,
and has realization 
\begin{align*} 
E&\to \bZ^E/H_1(G,\bZ), \\ 
i&\mapsto \overline{e}_i
\end{align*}
where $\overline{e}_i$ is the image of 
$e_i$ under the natural quotient map. Its rank is
$|E(G)|-\rk(H_1(G,\bZ))$.
The graphic and cographic matroids of $G$
are sometimes called the {\it cycle} and {\it bond
matroids} of $G$, respectively, in the matroid
literature.
\end{example}

\begin{remark}
Let $T\subset E$ be a spanning forest of $G$. Associated
to $T$ is a basis of $H_1(G,\bZ)$ indexed by the edges in
$E\setminus T$: each edge $i\in E\setminus T$ completes
a unique closed circuit $C_i$ of the graph $G$ whose edges lie in 
$T\cup \{i\}$. These closed circuits determine a $\bZ$-basis
of $H_1(G,\bZ)$ and form the circuits of the graphic
matroid $M(G)$.

Let $g=g(G)$ be the genus of the graph
and $k = |E|$ be the number of edges. Then,
in this basis, the integral realization of the cographic 
matroid $M^*(G)$ in Example \ref{cographic}
is a $g\times k$ matrix of the form
$M_G^*=({\rm Id}_g\,|\,P)$
where $P$ is a matrix of $0$'s and $\pm 1$'s, 
and whose $i$-th row consists of the directed
edges of $T$ involved in the circuit $C_i$.
\end{remark}

\begin{example}\label{theta-matroid}
 If $G$ is taken to be the theta graph
(see Figure \ref{theta-figure}), with $(g,k)=(2,3)$, and the 
spanning tree $T$ is $\{e_3\}\subset E$, then
$$M_{G}^* = \begin{pmatrix} 
1 & 0 & 1 \\
0 & 1 & 1 \\
\end{pmatrix} $$
because the circuits completed by $e_1$ and $e_2$
are, respectively, $e_1+e_3$ and $e_2+e_3$. 
\end{example}

\begin{figure}[ht]
\includegraphics[width=1.5in]{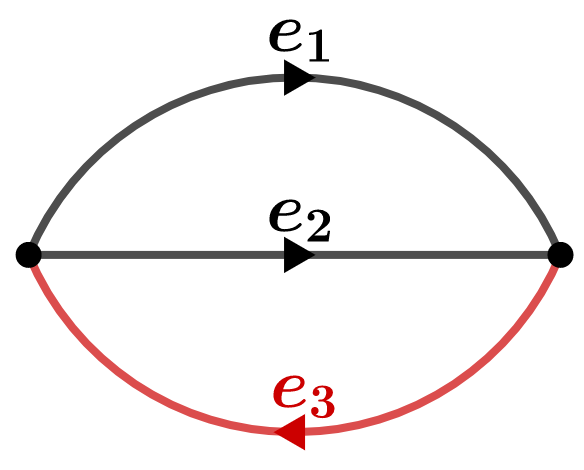}
\caption{The theta graph, with spanning tree
in red.}
\label{theta-figure}
\end{figure}

\begin{definition}\label{matroidal-cone} 
Let $i\mapsto {\bf x}_i\in \bfN\simeq \bfM^\vee$ 
be an integral realization of a regular
 matroid $\underline{R}$. 
 Then the associated {\it matroidal cone} $\bB_{\underline{R}}$
is the $\bR_{\geq 0}$-span of 
${\bf x}_i^2\in {\rm Sym}^2 \bfM^\vee$. 

For example, let $G$ be a graph. Its 
{\it cographic cone} $\bB_{M^*(G)}$
is the cone of symmetric, positive
semi-definite bilinear forms on $H_1(G,\bZ)$ given by 
\begin{align*} 
\bB_{M^*(G)}:&=\bR_{\geq 0}\{
{\bf x}_i^2\,\colon \,i\in E(G)\} \\ 
& = \{M_G^* D (M_G^*)^T \,:\,
D\textrm{ diagonal with }D_i\geq 0\}\subset \cP_g^+.
\end{align*}

\end{definition}

See Alexeev-Brunyate and Melo--Viviani for analyses of which 
matroidal cones appear in various toroidal
compactifications of $\cA_g$ 
\cite{brunyate, melo}.

\begin{example}[Seymour-Bixby {\cite{seymour, bixby}}] 
\label{seymour-bixby}
Consider the totally unimodular matrix
$$ R_{10} = \begin{pmatrix} 
1 & 0 & 0 & 0 & 0 & -1 & 1 & 0 & 0 & 1 \\
0 & 1 & 0 & 0 & 0 & 1 & -1 & 1 & 0 & 0 \\
0 & 0 & 1 & 0 & 0 & 0 & 1 & -1 & 1 & 0 \\
0 & 0 & 0 & 1 & 0 & 0 & 0 & 1 & -1 & 1 \\
0 & 0 & 0 & 0 & 1 & 1 & 0 & 0 & 1 & -1 \\
 \end{pmatrix}.$$ Then, the columns of 
 $R_{10}$ define a regular matroid on $10$ elements in $\bZ^5$.
We have already seen this matroid, which can be identified with 
the $10$ vanishing cycles $\gamma_i\in {\rm gr}^W_{-2}V_\bZ\simeq \bZ^5$ 
 of the nodes of the Segre cubic threefold, 
 $V_\bZ=H_3(Y_*,\bZ)(-1)$, see Example
 \ref{gwena-ex}.
 
  The associated matroidal cone is
 \begin{align*} \bB_{\underline{R}_{10}}=\bR_{\geq 0}
 \{{\bf x}_1^2,\, {\bf x}_2^2,&\,{\bf x}_3^2,\,
 {\bf x}_4^2,\,{\bf x}_5^2,\,
 ({\bf x}_5-{\bf x}_1+{\bf x}_2)^2,\,
 ({\bf x}_1-{\bf x}_2+{\bf x}_3)^2, \\ 
 &({\bf x}_2-{\bf x}_3+{\bf x}_4)^2,\,
 ({\bf x}_3-{\bf x}_4+{\bf x}_5)^2,
 \,({\bf x}_4-{\bf x}_5+{\bf x}_1)^2\}\subset 
 {\rm Sym}^2\, \bfM^\vee\otimes \bR\end{align*}
 where ${\bf x}_i$ for $i=1,\dots,5$ are the 
 coordinates on $\bfM\simeq \bZ^5$ given by
 the first five columns of $R_{10}$. A particularly
 nice realization of $\underline{R}_{10}$ over $\bF_2^5$
 is as the $10$ vectors in $(\bF_2)^{\oplus 6}/
 \bF_2\langle 1,1,1,1,1,1\rangle$ 
 which have exactly three nonzero entries, 
 cf.~Example \ref{gwena-ex}. In this realization,
 the full automorphism group 
 $S_6\simeq {\rm Aut}(\underline{R}_{10})$
 is readily visible.
\end{example}

\begin{remark} \label{R10notcographic}
Examples \ref{cographic} and \ref{seymour-bixby} are 
essentially different: the matroid $\underline{R}_{10}$ 
is not isomorphic to $M(G)$ or $M^\ast(G)$, 
for any graph $G$, 
see e.g.~\cite{seymour}. 
\end{remark}

\begin{definition}\label{mat-degen}
Let $f^\ast \colon 
X^\ast\to Y^\ast$ be a smooth,
projective family of PPAVs over a
smooth quasiprojective base $Y^\ast$,
and let $Y^\ast\hookrightarrow Y$ be an snc extension. 
We say that the morphism $f^\ast$ is {\it matroidal}
with respect to the extension $Y$ if:
\begin{enumerate}
    \item the monodromy at the boundary $Y\setminus Y^\ast$
    is unipotent, and
    \item the monodromy cones at all snc strata
    of $Y\setminus Y^\ast$ are matroidal cones.
\end{enumerate}
\end{definition}

Matroidal morphisms exist, in view of the following result:

\begin{proposition} \label{puncturedconstruction}
Let $(\underline R,E)$ be a regular matroid of rank 
$g$ on a $k$ element set $E = \{1, \dotsc, k\}$, 
with integral realization $E\to \bfN=\bfM^\vee$, 
$i\mapsto {\bf x}_i$, and let 
$(r_1,\dots,r_k)\in \bN^k$ be a 
vector of positive integers.
Then there is a smooth projective family 
$f^\ast \colon X^\ast\to Y^\ast$ of $g$-dimensional 
PPAVs over a smooth quasiprojective 
base of dimension $k$, 
such that the following hold: 
\begin{enumerate}
    \item There is a smooth extension 
    $Y^\ast\subset Y$ with snc boundary divisor 
    $D=Y\setminus Y^\ast$ 
    and an embedded polydisc $\Delta^k\subset Y$ such that 
    the restriction of $D$ to $\Delta^k$ agrees with the 
    union $\set{u_1 \cdots u_k = 0}$ of the coordinate 
    hyperplanes.
    \item Consider the base change 
    $X^\ast_{(\Delta^\ast)^k}\to (\Delta^\ast)^k$ 
    and let $t\in  (\Delta^\ast)^k$.
    Then there is an isomorphism 
    ${\rm gr}^W_0 H_1(X_t,\bZ)\simeq \bfM$ under which 
    the monodromy bilinear form $B_i$ around the 
    $i$-th coordinate hyperplane (Def.~\ref{def:Bi}) 
    is given by $r_i{\bf x}_i^2$.
\end{enumerate} 
\end{proposition}

\begin{proof}
Apply Corollary 
\ref{corollary:monodromyrealization}
to the symmetric
bilinear forms 
$B_i=r_i{\bf x}_i^2\in \cP_g^+\cap {\rm Sym}_{g\times g}(\bZ)$. 
\end{proof}

In the following sections, we will study
regular extensions $f\colon X\to Y$
of matroidal morphisms. 

\subsection{Mumford degenerations associated to regular matroids}
\label{sec:reg}

\begin{definition} \label{hyperplanearrangement}
 A \emph{hyperplane arrangement} is a finite 
collection $\{\oH_i\}_{i\in I}$
 of torsion translates 
 $\oH_i\subset \bT^g=\bfM_\bR/\bfM$ 
 of codimension $1$ subtori. Equivalently, 
 it is a finite collection $\{H_i\}_{i \in I}$ where
 $H_i$ is the union of all $\bfM$-translates of an affine
 linear hyperplane in $\bfM_\bR$ defined over $\bQ$. 
\end{definition}

Let ${\bf x}_1,\dots, {\bf x}_k\in {\bf N}$
be a finite collection of vectors, ${\rm rank}\,{\bf N}=g$, giving
an integral realization of a regular matroid $\underline{R}$.
Each ${\bf x}_i$ defines a family
of parallel hyperplanes \begin{align}\label{Hi}H_i\coloneqq \{{\bf m}\in 
{\bf M}_\bR\,:\,{\bf x}_i({\bf m})\in \bZ\}
\subset {\bf M}_\bR.\end{align} 
Then $\{H_1, \dotsc, H_k\}$
defines a hyperplane arrangement, 
see Definition \ref{hyperplanearrangement}.
Moreover, $H_i$ is the bending locus of 
the convex $\bZ$PL function 
$b_i$ on ${\bf M}_\bR$ that satisfies
$$b_i({\bf m})
=
\frac{{\bf x}_i({\bf m})^2-{\bf x}_i({\bf m})}{2}
\textrm{ for }{\bf m}\in {\bf M}.$$
Note that $b_i$ descends as a convex
section $\overline{b}_i\in 
H^0(\bT^g, \bZ{\rm PL}/\bZ {\rm L})$.
The regularity of the matroid implies 
that the $b_i$ are dicing, 
i.e.~the polytopes cut by the union of hyperplanes
$H_{\underline{R}} \coloneqq \bigcup_{i=1}^k H_i$ 
have integral vertices. See Erdahl--Ryshkov \cite{erdahl}.

\begin{definition}\label{matroidal-mumford} 
We define the {\it matroidal Mumford construction}
$X(\uR)\to\Delta^k$ to be the
Mumford construction $X(b_1,\dots,b_k)\to \Delta^k$
associated to the collection 
$(\bT^g, \overline{b}_1,\dots,\overline{b}_k)$ 
of sections $\overline{b}_i\in H^0(\bT^g, \bZ{\rm PL}/\bZ {\rm L})$
above, defined by the regular matroid $\underline{R}$.
\end{definition}

Then $k$ is the size of the ground set of $\uR$ while
the dimension $g$ of the fibers is the rank of $\uR$.
By construction, the monodromy cone
of the matroidal Mumford construction
on $\underline{R}$ is given by the matroidal cone $\bB_{\underline{R}}\subset \cP_g^+$ because
the bilinear form $B_i$ 
associated to $b_i$ is 
$B_i = {\bf x}_i^2$.

\begin{example}[Cographic matroids] \label{cographic-ex}
We have implicitly seen an important matroid, realized by the vanishing cycles associated
to a nodal projective curve $C_0$ 
with $k$ nodes, as in Example \ref{nodal}. 
We aim to 
\begin{enumerate}
\item\label{item1} 
identify the matroid realized
by the vanishing cycles as the cographic matroid 
$M^*(G)$ where  $G\coloneqq H_1(\Gamma(C_0),\bZ)$ 
is the dual graph of the nodal curve, and 
\item\label{item2} under the simplifying
assumption that the normalized components of $C_0$ have genus zero,
use this identification and Construction
\ref{matroidal-mumford} to compactify the 
relative Jacobian of the universal deformation
$\pi\colon 
\cC\to {\rm Def}_{C_0}\simeq \Delta^{3g-3}
\simeq \Delta^k\times 
\Delta^{(3g-3)-k}$.
\end{enumerate}

We begin with (\ref{item1}). 
Let $C_t$ be a smooth fiber nearby $C_0$.
Let $E=\{1,\dots ,k\}$ be the set of nodes of $C_0$ 
and let $\gamma_1,\dots ,\gamma_k\in H_1(C_t,\bZ)$ 
be the corresponding vanishing cycles, unique up to 
sign. This realizes a matroid on the ground set $E$.
A choice of sign for each $\gamma_i$ is equivalent to 
a choice of orientation of the edges of $G=\Gamma(C_0)$.
Using the intersection pairing and Poincar\'e duality 
on $C_t$, we may view each $\gamma_i$ as a linear form 
on $H_1(C_t,\bZ)\simeq H_1(JC_t,\bZ)$.
This linear form vanishes on $W_{-1}$ and so 
descends to a linear form on ${\rm gr}^W_0H_1(C_t,\bZ)
\simeq {\rm gr}^W_0H_1(JC_t,\bZ)$.

We also have an identification 
${\rm gr}^W_0H_1(C_t,\bZ)\simeq 
H_1(\Gamma(C_0),\mathbb Z)$; thus
$\gamma_i$ is identified with the linear form 
on $H_1(G,\bZ)$ giving the coordinate 
$i\mapsto {\bf x}_i=e_i^\vee\in H^1(G,\bZ)$
of the oriented edge $e_i\in E(G)$.
So the matroid realized by the 
$k$ vanishing cycles in $C_t$ is isomorphic to 
the cographic matroid $M^*(G)$.
Conversely, the cographic matroid $M^*(G)$
associated to any graph $G$ arises this way, 
because we can construct a nodal projective curve 
$C_0$ whose dual complex is $G$.

We now compactify the relative Jacobian
fibration $J\pi^\circ\colon
J\cC^\circ\to (\Delta^*)^k\times \Delta^{(3g-3)-k}$ 
of the  punctured family 
$\pi^\circ\colon 
\cC^\circ\to (\Delta^*)^k$ of smooth curves,
as in (\ref{item2}). 
So take ${\bf M}=H_1(G,\bZ)$
and ${\bf N}=H^1(G,\bZ)$
and apply the universal form
of the matroidal Mumford construction
of Definition \ref{matroidal-mumford}.
If no ${\bf x}_i=0$ nor ${\bf x}_i=
\pm {\bf x}_j$ for $i\neq j$
(in matroidal language:~$M^*(G)$ contains
no circuits of length $\leq 2$), 
we get an extension
$$X^{\rm univ}(M^*(G))\to 
\widetilde{\cA}_g^{\,\bB_{M^*(G)}}$$
of the universal family.
Otherwise, we only get a degeneration
$$X^{\rm univ}_\circ(M^*(G))\to (\Delta^{k})^{\rm univ}$$
where $(\Delta^{k})^{\rm univ}\to 
(\bC^*)^{{g+1\choose 2}-\ell}$ is a 
$k$-dimensional polydisk
bundle for some $\ell<k$, for which the 
classifying map $(\Delta^k)^{\rm univ}\to \cA_g$
loses dimension. 
So assume the former.
It follows then from
Proposition \ref{monodromy-extend}
that the Torelli map extends to a morphism 
$${\rm Def}_{C_0}\simeq
\Delta^k\times \Delta^{(3g-3)-k}
\hookrightarrow 
\widetilde{\cA}_g^{\, \bB_{M^*(G)}}.$$
Then the pullback of 
$X^{\rm univ}(M^*(G))$ defines 
an extension of the relative
Jacobian fibration 
$$J\pi\colon JC\to \Delta^{k}\times 
\Delta^{(3g-3)-k}$$ where $\{0\}\times 
\Delta^{(3g-3)-k} \to  \{0\}\times 
(\bC^*)^{{g+1\choose 2}-k}$ maps the 
locally trivial deformations of $C_0$ into the 
deepest toroidal stratum 
$\widetilde{\cA}_g^{\, \bB_{M^*(G)}}$.

Example \ref{theta-3-param} is an example
of a matroidal Mumford construction on the cographic matroid
of the theta graph (Example \ref{theta-matroid}),
where $C_0=\bP^1\cup_{\{0,\,1,\,\infty\}}\bP^1$
is the union of two smooth rational curves
along three points. The dual graph 
$\Gamma(C_0)$ is the theta graph, as depicted
in Figure \ref{theta-figure}.
\end{example}

\begin{example}[Matroidal Mumford degeneration on $\underline{R}_{10}$] \label{r10-ex}

Another example comes from the Seymour--Bixby matroid 
$\underline{R}_{10}$ (Example \ref{seymour-bixby}) 
which gives a degeneration 
$X(\uR_{10})\to \Delta^{10}$ 
of PPAVs of dimension $5$ over a $10$-dimensional 
polydisk.  To produce a universal 
degeneration whose monodromy
cone is $\bB_{\underline{R}_{10}}$
we must twist by $a\in (\bC^*)^5$ 
(here $5=15-10$ and $15=\dim \cA_5$).
The resulting universal Mumford degeneration
is an extension of the universal family
$$X^{\rm univ}(\uR_{10})
\to 
\widetilde{\cA}_5^{\, \bB_{\underline{R}_{10}}}.$$ 
If $\pi\colon Y\to {\rm Def}_{Y_0}\simeq \Delta^{10}$ 
is the universal deformation
of the Segre cubic threefold (Example \ref{gwena-ex}),
there is a morphism $ {\rm Def}_{Y_0} \to \widetilde{\cA}_5^{\, \bB_{\underline{R}_{10}}}$ 
transversely slicing the deepest toroidal
boundary stratum $(\bC^*)^5$. 
The intersection is transverse because 
the monodromy about each coordinate 
hyperplane is $B_i={\bf x}_i^2$---thus, 
in toroidal charts of 
$\widetilde{\cA}_5^{\,\bB_{\uR_{10}}}$, 
the period
map is approximated by a translate of a subtorus which
transversely slices the deepest boundary stratum.
The pullback
of $X^{\rm univ}(\uR_{10})$ defines an extension
$IJ\pi\colon IJY\to \Delta^{10}$ of the relative 
intermediate Jacobian fibration
$IJ\pi^\circ\colon IJY^\circ
\to (\Delta^*)^{10}$ over the smooth locus.
\end{example}

\begin{remark} By \cite[Lem.~4.0.5, Cor.~4.0.6]{melo}, 
we do not need to pass to some \'etale cover
$\widetilde{\cA}_g\to \cA_g$ to produce the toroidal
extension associated to a matroidal cone---up 
to quotienting by the group of symmetries of $\uR$
and identifying some faces,
every matroidal cone $\bB_{\uR}$ 
on a regular  matroid $\uR$ with no loops or parallel 
elements ({\it simple} regular 
matroids in {\it loc.~cit.}) embeds into
$(\cA_g)_{\rm trop}$ as 
in Example \ref{embed}. Indeed, 
there is a universal matroidal 
extension $\cA_g\hookrightarrow
\cA_g^{\rm mat}$ whose fan is the union
of all matroidal cones, on regular matroids
of rank $\leq g$, with no 
loops or parallel elements. By \cite[Thm.~A]{melo},
$\cA_g^{\rm mat}$ is the toroidal extension of
the maximal common subfan of the first and second Voronoi fans. 
\end{remark}

\subsection{Shifted and transversely shifted
matroidal degenerations}
\label{sec:shift-reg}

We describe here a modification of the matroidal
Mumford construction of Definition \ref{matroidal-mumford}
which produces a regular total space;
this property is quite special, and under some 
additional hypotheses, 
characterizes Mumford degenerations with
regular total space.

\begin{construction}[Shifted matroidal degenerations]\label{shifted-matroidal-construction}
Suppose that $E\to {\bf N}$, $i\mapsto {\bf x}_i$ 
gives an integral realization of a regular matroid. 
Consider the family of parallel hyperplanes
$H_i^o=\{{\bf m} \in \bfM_\bR \,:\, 
{\bf x}_i({\bf m})\in \bZ\}\subset {\bf M}_\bR$ for $i\in E$.
Then, all $H_i^o$ intersect at all lattices points
${\bf M}$, or in the quotient $\bT^g=\bfM_\bR/\bfM$, 
at the origin.
Thus, we consider the shifted hyperplanes 
$$H_i\coloneqq \{
{\bf m} \in \bfM_\bR \,:\, 
{\bf x}_i({\bf m})\in \epsilon_i+\bZ\}$$
for $\epsilon_i\in \tfrac{1}{d}\bZ$. For sufficiently
large $d$, it is possible to choose
values of $\epsilon_i$ for which
this shifted hyperplane arrangement $\{H_1, \dots, H_k\}$
satisfies the following additional property:

\begin{definition} \label{transversalarrangement}
A hyperplane arrangement $\{H_i\}_{i\in I}$
is {\it transversal} if at any 
intersection point $p\in \bigcap_{i\in I}H_i$
the normal vectors of $H_i$ for $i\in I$
are linearly independent.
\end{definition}

Associated to $H_i$ (transversal or not)
we define
a piecewise linear function 
$b_i\colon {\bf M}_\bR\to \bR$ 
by the properties that:
\begin{enumerate}
    \item 
$b_i(m) \coloneqq  
\tfrac{1}{2}({\bf x}_i({\bf m}-{\bf m}_0)^2-
{\bf x}_i({\bf m}-{\bf m}_0))$ for any 
${\bf m}\in {\bf m}_0+\bfM$
where ${\bf m}_0\in {\bf M}_\bR$ is any point 
for which ${\bf x}_i({\bf m}_0)=\epsilon_i$ and 
\item ${\rm Bend}(b_i) = H_i$ is the shifted family
of hyperplanes.
\end{enumerate}
The associated bilinear forms $B_i={\bf x}_i^2$
are the same as 
for the unshifted case.

While the function $b_i$ is not
integer valued on integer points, it is 
$\tfrac{1}{d}$-integer valued on 
$\tfrac{1}{d}$-integral points. Thus,
by Construction \ref{veronese}, we may take
a Mumford degeneration
$$f\colon X(d\mid b_1,\dots,b_k)\to \Delta^k$$
whose monodromies $B_i$ are the same as
those of the matroidal Mumford degeneration.
More generally, we
have, by Construction 
\ref{mumford-polytope-2},
a universal form
$X^{\rm univ}(d\mid b_1,\dots,b_k)\to 
\widetilde{\cA}_g^{\, \bB}$. 
We may further generalize
this set-up:

\begin{definition} \label{def:shifted-matroidal-degeneration}
Let $\underline{R}$ be a regular matroid.
A {\it shifted matroidal degeneration} 
on $\underline{R}$ is a 
Mumford construction $X(d\mid b_1,\dots,b_k)\to \Delta^k$
for which the bending locus of $b_i$ is a union
of parallel hyperplanes in $\bfM_\bR$ whose primitive
normal vectors
${\bf x}_i\in \bfN$ give an integral realization of $\uR$.
We furthermore call a shifted matroidal
degeneration {\it transversely shifted} if 
\begin{enumerate}
    \item the hyperplane arrangement 
    $\{{\rm Bend}(b_i)\,:\,i=1,\dots,k\}$
    is transversal, and
    \item the bending parameter of $b_i$ along each hyperplane is $1$.
\end{enumerate}
\end{definition}

We remark that Definition 
\ref{def:shifted-matroidal-degeneration} 
is well-defined, independent
of the choice of primitive 
normal vector ${\bf x}_i\in \bfN$, which is 
only unique up to sign. 
Closely related constructions go by the
name ``multiplicative hypertoric variety'' in more
representation-theoretic literature, see especially
\cite[Sec.~8.3]{mcbreen} (for the cographic case), 
generalizing the ``additive'' case
\cite{arbo, bd, hausel}.\hfill $\clubsuit$ 
\end{construction}

See the right hand sides of Figures 
\ref{fig:shifted-theta}, \ref{super_theta}
for examples of transversely shifted arrangements.

\begin{notation} \label{notation:XunivRH}
    The bending locus ${\rm Bend}(b_i)$ can be viewed as a multiset
    $\{H_i^{(1)},\dots,H_i^{(r_i)}\}$ of rational translates
    of the hyperplane normal to ${\bf x}_i$---a hyperplane $H$
    appears $m$ times in the multiset if the bending
    parameter of $b_i$ along $H$ is $m\in \bN$. Then, we assemble
    the hyperplane arrangement into a single symbol
    $$\mathscr{H}\coloneqq \{H_1^{(1)},\dots,H_1^{(r_1)},\dots,
    H_k^{(1)},\dots,H_k^{(r_k)}\}$$ 
    where $H_i^{(j)}$ is the set of
    $\bfM$-translates of a single
    hyperplane normal to ${\bf x}_i$. 
    To indicate the relation
    to the regular matroid $\uR$, we re-notate
    the Mumford construction of 
    Definition \ref{def:shifted-matroidal-degeneration},
    or the universal version, as $$X(\uR, \mathscr{H})\to \Delta^k \quad 
    \textrm{ or } \quad
    X^{\rm univ}(\uR, \mathscr{H})\to \widetilde{\cA}_g^{\,\bB}.$$ 
\end{notation}

\begin{example}[Transversely shifted matroidal degeneration
for the theta graph] \label{shifted-matroidal}

Beginning with the standard matroidal degeneration
$X(M^*(G))\to \Delta^3$ associated to the theta graph $G$
(Example \ref{theta-3-param}), take the degree
$2$ Veronese embedding. The resulting polyhedral decomposition
of $\bT^2$ and bending loci are depicted
in the top left of Figure \ref{fig:shifted-theta}.
Now consider the shifts of the families of hyperplanes 
$H_1^o$, $H_2^o$, $H_3^o$ (of colors red, green, blue, 
respectively) of the standard arrangement for $M^*(G)$,
by $(\epsilon_1,\epsilon_2,\epsilon_3)= (0, 0, \tfrac{1}{2})$. 
The result is a transversal arrangement 
$\mathscr{H}=\{H_1,H_2,H_3\}$ with $H_1=H_1^o$,
$H_2=H_2^o$, $H_3=H_3^o+(\tfrac{1}{2},0)$.

\begin{figure}
\includegraphics[width=4in]{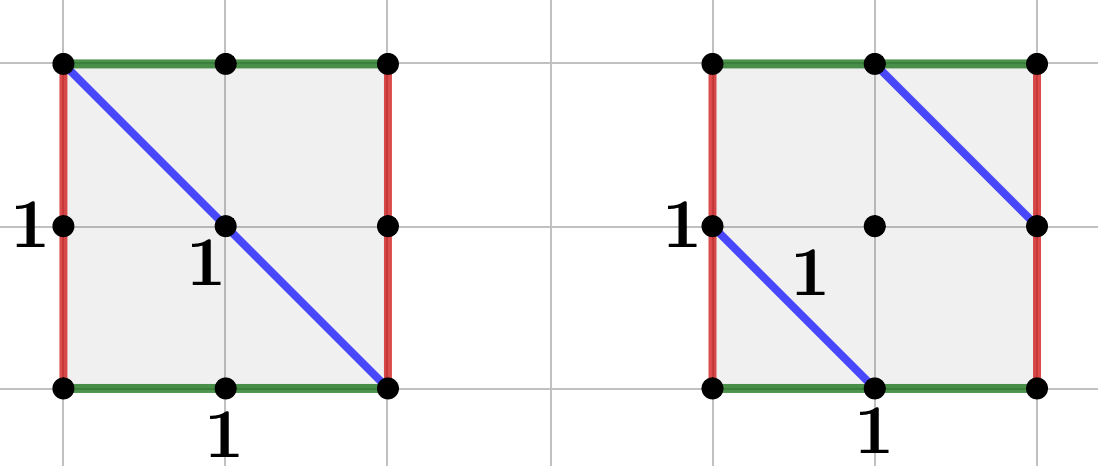}
\caption{Left: $2$nd Veronese embedding
of the matroidal degeneration of the theta graph.
Right: a shifted version. 
Grid points are $(\tfrac{1}{2}\bZ)^2$.
The integers are the multiplicities of given 
hyperplane in the arrangement $\mathscr{H}$.} 
\label{fig:shifted-theta}
\end{figure}

\begin{figure}
\includegraphics[width=2.5in]{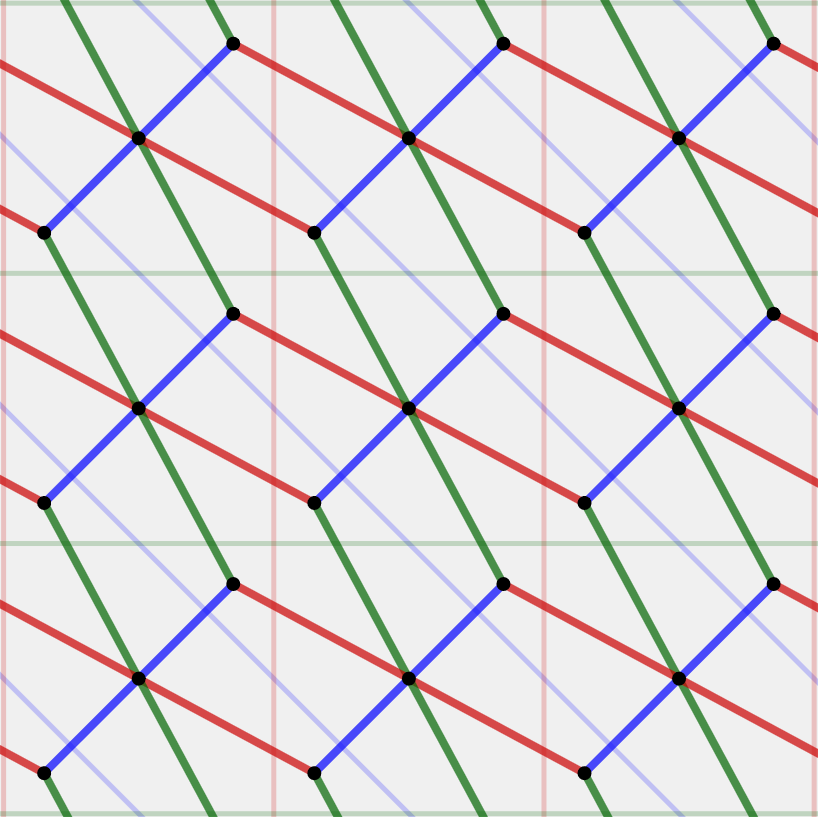}
\caption{
Dual complex
(edges colored) of the righthand shifted arrangement in Figure \ref{fig:shifted-theta}. It is a tiling of
a $2$-torus by cubes.}
\end{figure}

Applying Proposition \ref{bending-read}, it can be seen that
the central fiber of the left-hand Mumford 
degeneration in Figure \ref{fig:shifted-theta}
is the union of two copies of $\bP^2$,
both polarized by $\cO_{\bP^2}(2)$.
Similarly,
the central fiber of the righthand Mumford degeneration
$X(M^*(G),\mathscr{H})\to \Delta^3$
is the union of three surfaces: Two copies of
$\bP^2$, both polarized by $\cO_{\bP^2}(1)$,
and a Cremona surface 
$V\coloneqq Bl_{p_1,p_2,p_3}(\bP^2)$,
polarized by the anticanonical divisor $-K_V$.
\end{example}

\begin{figure}[ht]
  \centering
  \setbox1=\hbox{\includegraphics[width=2.5in]{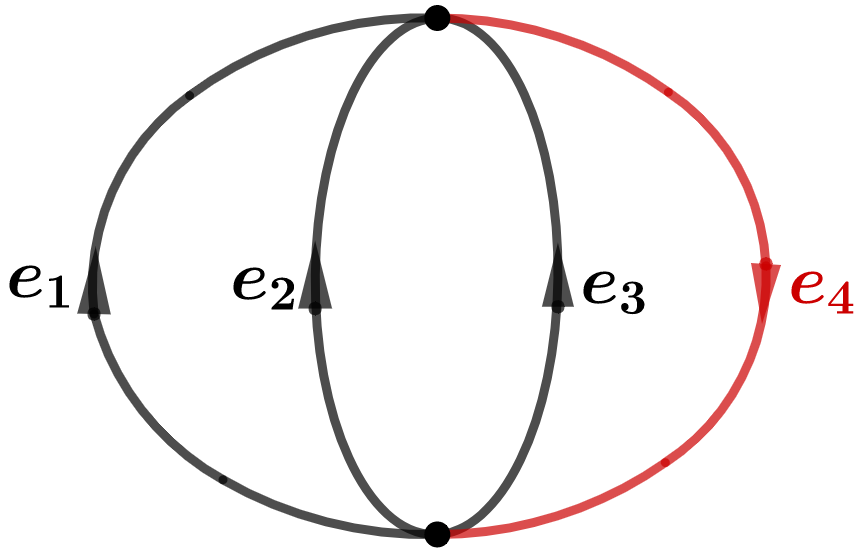}}
  \setbox2=\hbox{\includegraphics[width=2.5in]{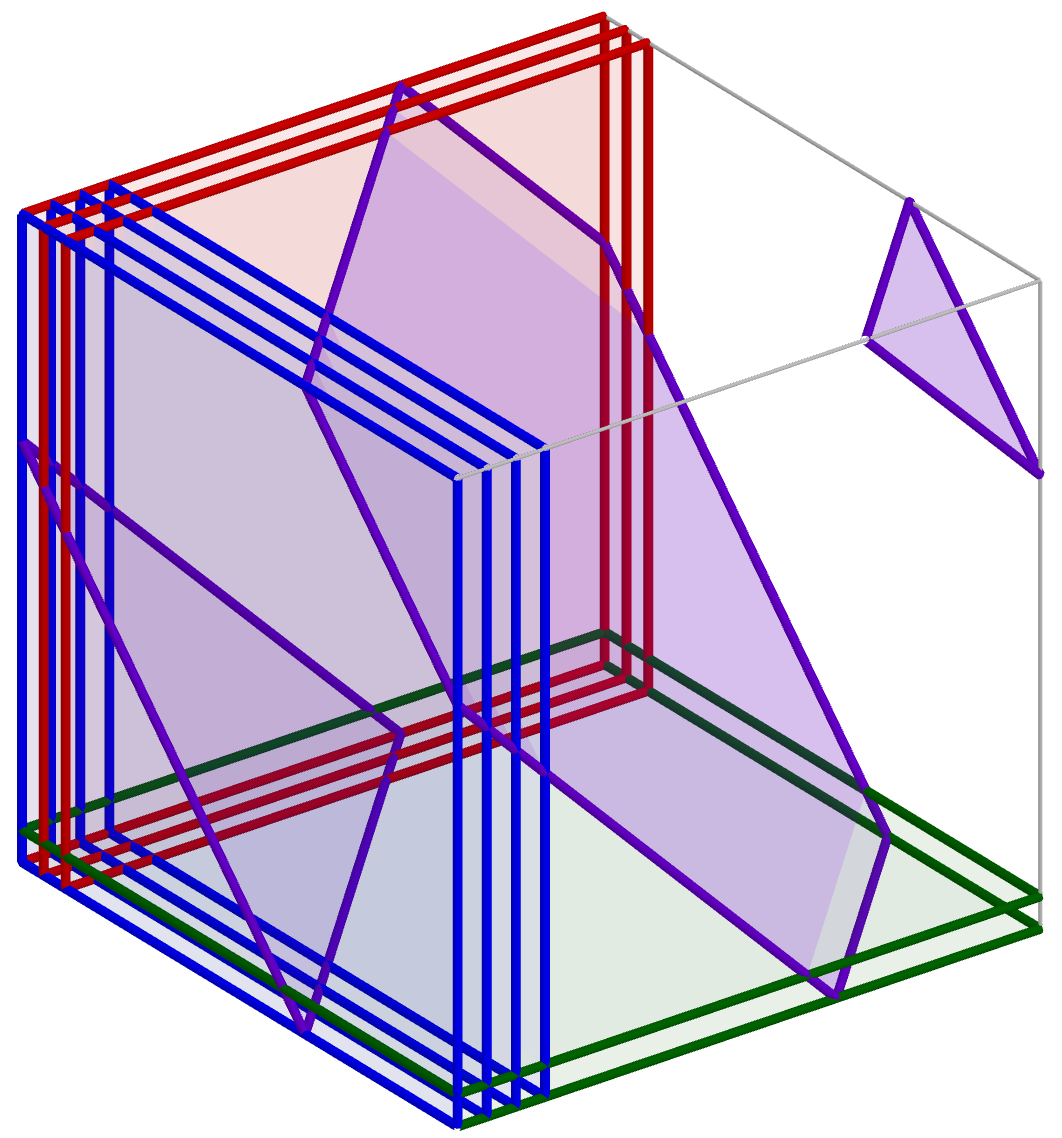}}
  \newlength{\imageheightA}
  \newlength{\imageheightB}
  \setlength{\imageheightA}{\ht1}
  \setlength{\imageheightB}{\ht2}
  \newlength{\maximageheight}
  \setlength{\maximageheight}{\maxof{\imageheightA}{\imageheightB}}
  
  \raisebox{0.5\dimexpr\maximageheight-\imageheightA}{\usebox1}
  \hspace{1em}
  \raisebox{0.5\dimexpr\maximageheight-\imageheightB}{\usebox2}

  \caption{Left: Oriented genus $3$ graph with spanning
  tree in red. 
  Right: Transversely shifted hyperplane arrangement in $\bT^3$.}
  \label{super_theta}
\end{figure}

\begin{example}\label{shifted-ex-2}

Consider the regular
matroid defined by the matrix $$\begin{pmatrix}
1 & 0 & 0 & 1 \\
0 & 1 & 0 & 1 \\
0 & 0 & 1 & 1
\end{pmatrix}.$$
It is the
cographic matroid of the oriented
graph $G$ depicted in the left of Figure 
\ref{super_theta}. Letting ${\bf x}_i\in H^1(G,\bZ)$ 
for $i=1,2,3,4$ be the linear forms corresponding
to the four oriented edges of $G$, we form a transversely
shifted matroidal degeneration $X(M^*(G),\mathscr{H})$
where $$\mathscr{H}=\{H_1^{(1)}, H_1^{(2)}, H_1^{(3)}, 
H_2^{(1)}, H_2^{(2)}, H_2^{(3)}, H_2^{(4)},
H_3^{(1)}, H_3^{(2)},
H_4^{(1)}
\}.$$ 
There are, respectively, $3$, $4$, $2$, $1$
hyperplanes perpendicular to 
${\bf x}_1=(1,0,0)$,
 ${\bf x}_2=(0,1,0)$,
  ${\bf x}_3=(0,0,1)$,
   ${\bf x}_4=(1,1,1)$. These hyperplanes
   are, respectively,
   depicted in red, blue, green, and purple
   in Figure \ref{super_theta}.
\end{example}

\section{Nodal and semistable morphisms
over higher-dimensional bases}
\label{sec:semistable}

\begin{definition}\label{D-nodal} 
Let $Y$ be a smooth analytic space,
and let $D\subset Y$ be an snc divisor, 
$D=\bigcup_i D_i$. 
Let $f\colon X\to Y$ be a morphism of 
analytic spaces. We say that $f$ is 
\begin{enumerate}
\item {\it $D$-nodal} if for every point $p\in X$, 
there are analytic coordinates
in which the morphism $f$ is 
of the form $$\prod_{i\in I} \{x_iy_i=u_i\}\times 
\Delta^{j+k}\to \prod_{i\in I}\Delta_{u_i} \times \Delta^j$$ 
where $u_i$ are local equations for some components 
$D_i\subset D$, $i\in I$ and $\Delta^{j+k}\to \Delta^j$
is the projection to the first $j$ coordinates,
\item {\it nearly $D$-nodal}\label{d-nearly-form} 
if we rather have a normal form
of shape $$\prod_{i\in I} \{x_i^{(1)}y_i^{(1)}=\cdots = 
x_i^{(n_i)}y_i^{(n_i)}=u_i\}\times \Delta^{j+k}\to 
\prod_{i\in I}\Delta_{u_i} \times \Delta^j,$$
\item\label{d-snc-form} {\it $D$-semistable} if we have
$$\prod_{i\in I} \{x_i^{(1)}\cdots x_i^{(n_i)} =u_i\}\times 
\Delta^{j+k}\to \prod_{i\in I}\Delta_{u_i} \times \Delta^j.$$ 
\end{enumerate}
In all three cases, if we furthermore have
that the irreducible components $V_i\subset X_i$ of the 
generic fiber of $f$ over each component of $D_i$ 
are smooth, we use the term {\it strict}.
\end{definition}

This definition works equally well 
in the algebraic category, replacing $\Delta$ with
$\bA^1$ and analytic-local charts with \'etale-local charts.
In the cases where $f$ is $D$-nodal or $D$-semistable, 
the total space $X$ is smooth, 
but if some $n_i\geq 2$ for a nearly $D$-nodal 
morphism, then $X$ is singular.

\begin{remark} 
The notion of a $D$-semistable morphism is already known
in the literature by the term {\it semistable morphism},
and when context is clear, we also drop the $D$.
By Adiprasito--Liu--Temkin's resolution
\cite{alt} of the conjecture of 
Abramovich--Karu \cite{ak}, for every dominant morphism
$f\colon X\to Y$, there is an alteration
$Y'\to Y$ and a modification $X'\to X\times_YY'$
of the base change which is $D$-semistable, 
for the discriminant divisor $D$.
\end{remark}

\subsection{Mumford degenerations with nodal singularities}

\begin{proposition}\label{X-smooth}
Let $f\colon X(\uR, \mathscr{H})\to \Delta^k$
be a transversely
shifted matroidal degeneration. Then,
 the morphism $f$ has $D$-nodal singularities,
 where $D\coloneqq V(u_1\cdots u_k)\subset \Delta^k$ 
 is the union of the coordinate hyperplanes. 
In particular, $X(\uR, \mathscr{H})$ is smooth, 
$f$ is flat, and the fibers of $f$ are reduced.
The same results hold for the morphism
$f^{\rm univ}\colon X^{\rm univ}(\uR, \mathscr{H})\to 
\widetilde{\cA}_g^{\,\bB}$ with $D$ the 
toroidal boundary. Conversely, any Mumford
degeneration with $D$-nodal singularities is a 
transversely shifted
matroidal degeneration on some regular matroid $\underline{R}$. 
\end{proposition}

\begin{proof}
First we prove the forward direction.

Since the action of $d{\bf M}$ on the universal
cover of the Mumford degeneration is free, it suffices
to check the statement on this universal cover.
We first show
that every cone of the normal fan to 
$\Gamma$
is a standard affine cone 
(i.e.~integral-affine equivalent
to $\bN^i$ for some $i$).

The cones of the normal fan
are in bijection with the faces of $\Gamma$.
Let $F$ be a polyhedral face in the decomposition
$\bigcup_{i=1}^k H_i$ of $\bT^g$.
Then $F$ also defines a face of $\Gamma$
by evaluating $(b_1,\dots,b_k)$ on $F$. 
Conversely, all faces of $\Gamma$
contain such a face in their closure.
So it suffices to examine the normal fan
of the faces adjacent to $F$.

By the transversality hypothesis,
$F$ is locally described as an intersection
$\bigcap_{i\in I} H_i$ for which ${\bf x}_i\in \bfN$
for $i\in I$ are linearly independent.
Furthermore, the normal vectors
${\bf x}_i$ for $i\in I$ generate a standard 
affine cone in $\bfN$, because $i\mapsto {\bf x}_i$
is an integral realization 
of a regular matroid, and so the sublattice
generated by them is saturated.

By hypothesis,
the bending parameter of $b_i$ is $1$ across any 
hyperplane $H_i$ with normal vector
${\bf x}_i$. Thus, after an integral change of basis 
and translation of
$F$ to the origin, 
$b_i$ is locally expressible as
\begin{align}\label{local-b}
b_i({\bf m})=
\twopartdef{0}{{\bf x}_i({\bf m})\leq 0}{{\bf x}_i({\bf m})}{{\bf x}_i({\bf m})\geq 0}
\end{align}
on $\bfM_\bR$ for $i\in I=\{1,\dots, k\}$, $k\leq g$.
Combining all the above considerations, we
deduce that the local monoid $\bfM_F$ of 
the face $F$ is a product
$$\bfM_F=\prod_{i\in I} (\bZ_{\geq 0})^2 \times \!\!\!\!\!\prod_{i\,\in\,\textrm{Basis}\setminus I} \!\!
\bZ \times \prod_{E\setminus I}
\bZ_{\geq 0}\subset {\bf M}\times \bZ^k\simeq \bZ^{g+k}$$
where the first factors, indexed by $i\in I$, correspond
to two vectors along the graph of $b_i$ in the
two (local) domains
of linearity of (\ref{local-b}), the
second factors go along the
face $F$, and the third factors, indexed
by $E\setminus I$, are ``vertical'' faces,
arising from the fact that we took the overgraph
$\Gamma(b_1,\dots,b_k)\vert_{F}+(\bZ_{\geq 0})^k$.
Note that the dimensions of the factors
add up to the correct
value $$2|I|+(g-|I|)+(|S|-|I|) = g+k.$$
The dual cone to $\bfM_F$
is isomorphic to $(\bZ_{\geq 0})^{|I|+k}$.

We deduce that the cone
$\sigma_F\in \cS$ corresponding
to $F$, in the normal fan 
$\cS$ of $\Gamma$, is
standard affine. So $X$ is smooth. Furthermore,
the morphism to the fan $(\bR_{\geq 0})^k$ 
is, on the first factors $(x,y)\mapsto x+y$,
on the second factor is zero, and on the
third factor, is an isomorphism to the coordinate
axis indexed by the corresponding 
element of $E\setminus I$.
We deduce that the morphism 
$\sigma_F\to (\bR_{\geq 0})^k$
is a product of node smoothings, with a smooth
morphism, as in the definition of a $D$-nodal morphism.
The second statement follows. 

We now prove the reverse direction.
The condition that 
$X(d\mid b_1,\dots,b_k)\to \Delta^k$
have $D$-nodal singularities implies that the 
polyhedral complex ${\rm Bend}(b_i)$
can only have codimension $1$ faces. By convexity, we deduce
that ${\rm Bend}(b_i)$ is a disjoint union of parallel 
hyperplanes $H_i$ in $\bT^g$.
At any face $F$ of $\bigcup_{i=1}^k {\rm Bend}(b_i)$ 
where these hyperplanes intersect,
the normal vectors ${\bf x}_i$ to the hyperplanes  
must be a subset of a $\bZ$-basis of $\bfN$ 
for the normal fan to be standard affine. 
It follows that any linearly independent collection of 
${\bf x}_i$ generate
a saturated sublattice of $\bfN$ and so ${\bf x}_i$ 
define a regular matroid $\underline{R}$.
Furthermore, the fact that 
the normal vectors must be linearly independent
proves that the $H_i$ define a transversal arrangement.

The same results hold for
$X^{\rm univ}(\uR,\mathscr{H})\to 
\widetilde{\cA}_g^{\,\bB}$, which
over the germ of each open stratum
of $\widetilde{\cA}_g^{\,\bB}$
forms a locally trivial deformation
of the degeneration 
$X(\uR,\mathscr{H})\to \Delta^k$.
\end{proof}

\begin{question} A $1$-parameter semistable degeneration
$f\colon X\to Y$ which is relatively $K$-trivial
(i.e.~$K_X\sim_f0$) is called a Kulikov model,
see Remark \ref{kulikov}. Proposition
\ref{X-smooth} shows that it is natural to 
generalize this notion to a {\it multivariable 
Kulikov model}---a proper, semistable,
relatively $K$-trivial morphism. 
It is unclear in what context they are guaranteed to exist.
Proposition \ref{X-smooth} gives nontrivial
examples of such, 
for abelian varieties.
Given a family of $K$-trivial varieties $f\colon X\to Y$,
is there an alteration of the base $Y'\to Y$
and birational modification $f'\colon X'\to Y'$
of the base change, which is a multivariable Kulikov model?
Do multivariable Kulikov models
exist for families of K3 surfaces?
\end{question}

\begin{proposition}\label{strict-D}
Suppose that every face $F$
of a transversal
arrangement $\mathscr{H}$ 
is embedded in $\bT^g$ (as opposed to
immersed) for all subsets $I\subset E$. 
Then, every stratum of 
$X=X(\uR, \mathscr{H})$ 
is smooth. In particular, 
$f\colon X\to \Delta^k$
is strictly $D$-nodal.
Thus, a transversely
shifted matroidal degeneration
is strictly $D$-nodal if and only if there are least
two hyperplanes with normal vector ${\bf x}_i\in \bfN$
for all $i=1,\dots,k$.
\end{proposition}

\begin{proof} It follows from the hypothesis
and Proposition \ref{bending-read}
that if the $d{\bf M}$-action on the universal
cover $X(\cS)$ of the Mumford degeneration 
identifies two points $p,q\in X(\cS)$
lying on the same smooth toric stratum, then $p,q$ must
lie in a subtorus $(\bC^*)^k$ being quotiented
to an abelian $k$-fold. The first part of the proposition 
follows.

To show the second part: Suppose 
that there are $r_i\geq 2$ 
hyperplanes with normal vector ${\bf x}_i$. Let $I_0$
be a subset of a basis, say 
$I_0=\{1,\dots,h\}\subset \{1,\dots,g\}$ for which
the span of ${\bf x}_i$ for 
$i\in I$ is generated by the ${\bf x}_i$
for $i\in I_0$.
Then $\bigcup_{H\in \mathscr{H}}H$ 
is, combinatorially, 
a tiling of $\bT^g$ by the product
of a subtorus $\bT^{g-h}\subset \bT^g$ with some
polyhedral subdivision of the tiling of $\bT^h$ by
cubes of size $1/r_1\times \cdots \times 1/r_h$. 
All such cubes are embedded in $\bT^h$ once $r_i\geq 2$.
\end{proof}

\begin{remark}\label{perturb}
Let $\mathscr{H}$ define a 
transversal arrangement. Then,
any small rational
perturbation 
of the $H\in \mathscr{H}$ which keeps
the combinatorics of the intersection complex
$\bigcup_{H\in \mathscr{H}} H \subset \bT^g$ 
constant produces an isomorphic
degeneration 
$X(\uR, \mathscr{H})\to \Delta^k$,
since the normal fan is unchanged by such a perturbation.
Thus, the only difference between these two Mumford 
constructions is the choice of polarization on the total
space. More generally:
\end{remark}

\begin{proposition}\label{perturb2}
For any arrangement $\mathscr{H}$,
and any sufficiently small perturbation $\mathscr{H}'$
of $\mathscr{H}$,
there is a birational morphism
$X(\uR,\mathscr{H}')\to X(\uR,\mathscr{H})$ over
$\Delta^k$. \end{proposition}

\begin{proof} The proposition
follows from the fact that the
normal fan associated
to $\mathscr{H}'$ is a refinement 
of that for $\mathscr{H}$---all domains
of linearity of the PL
function $(b_1,\dots,b_k)\colon \bfM_\bR\to \bR^k$ 
bending along $\mathscr{H}$
``persist'' (up to a small deformation)
as domains of linearity for the PL function
$(b_1',\dots,b_k')\colon \bfM_\bR\to \bR^k$ 
bending along $\mathscr{H}'$.
Hence, any normal cone to a face of the polytope
$\Gamma$ is a union of normal cones
of the corresponding faces of $\Gamma'$. \end{proof}

\begin{corollary}
A small, transversal perturbation $\mathscr{H}'$
of the hyperplane arrangement 
(\ref{Hi}) defining the standard matroidal
Mumford construction $X(\uR)\to \Delta^k$
(Def.~\ref{matroidal-mumford}) defines
a projective resolution of singularities 
$X(\uR,\mathscr{H}')\to X(\uR)$ which is $D$-nodal over $\Delta^k$.
\end{corollary}

\begin{proof} The corollary follows from
Propositions \ref{perturb2} and \ref{X-smooth}. 
\end{proof}

\subsection{Weight filtration of a 
semistable morphism}

Throughout this section, suppose that 
$f\colon X\to Y$ is a (strict) 
$D$-semistable, proper morphism, 
for an snc pair $(Y,D)$, with $X$ K\"ahler. 
We fix a point $0\in D$ and denote by
$D_I^o$ the open snc stratum containing $0$.
Assume without loss of generality that
$I=\{1,\dots,k\}$. We also fix a base point
$t\in Y\setminus D$ near $0$.

\begin{proposition}\label{restrict}
Let $f\colon X\to Y$ be a (strict) 
$D$-semistable morphism for an snc pair 
$(Y,D)$ and let $C\to (Y,D)$ be
a pointed curve which transversely 
intersects the open stratum of a 
component $D_i\subset D$. Then the base change $
X\times_Y C\to C$ is a (strict) semistable degeneration. 
\end{proposition}

\begin{proof}This follows immediately from the normal
form (\ref{d-snc-form}). \end{proof}

Consider the inclusion of the nearby fiber 
$X_t\hookrightarrow X_{\Delta^I}$ 
into the restriction of $X\to Y$
to a polydisk $\Delta^I\ni 0$ transversely slicing the
snc stratum $D_I^o\ni 0$. We claim:

\begin{proposition}\label{clemens-collapse}
There is a deformation-retraction 
$c\colon X_{\Delta^I}\to X_0.$ 
The fibers of $c_t\coloneqq c\vert_{X_t}$ 
are real tori; more precisely, if $p\in X_0$
lies in a product of snc strata 
$p\in \prod_{i\in I} V(x_i^{(1)},\cdots ,x_i^{(n_i)})$,
cf.~\eqref{d-snc-form}, 
then the fiber of $c_t$ 
is $c^{-1}_t(p)=\prod_{i\in I} (S^1)^{n_i-1}$
where $(S^1)^{n_i-1}\subset X_t$ 
is the vanishing torus
of the semistable degeneration
$x_i^{(1)}\cdots x_i^{(n_i)}=u_i$. 
\end{proposition}

It would be natural to call the retraction
$c$ the ``multivariable Clemens collapse'',
in analogy with the Clemens collapse of a
$1$-parameter semistable degeneration, as in
\cite[Thm.~5.7]{clemens},
\cite[Sec.~2.3]{persson},
\cite[Prop.~C.11]{peters-steenbrink-mixedhodge}.

\begin{proof}[Proof of Proposition \ref{clemens-collapse}]
  Consider first a $1$-parameter semistable degeneration over
  $\Delta_u$ as for the usual Clemens collapse.
  Near an snc stratum of the fiber, one defines $c$ 
  as the deformation-retraction
  of $\{x^{(1)}\cdots x^{(n)}=u\}\subset \Delta^n\times \Delta_u$ 
  to $\{x^{(1)}\cdots x^{(n)}=0\}\subset \Delta^n\times \{0\}$ given
  by keeping the arguments of the complex numbers $x^{(j)}$ 
    constant and linearly decreasing the absolute values 
    $|x^{(j)}|$ until one of these absolute values equals zero. 
    These local deformation-retractions 
    may be patched via partitions of unity, 
    cf.~\cite[p.~236]{clemens}, to give a piecewise
    smooth retraction of the total space of the semistable
    degeneration onto its central fiber.
    
    The same procedure works in the $D$-semistable case: 
    We define a deformation-retraction $c$ of 
    $$\prod_i \{x_i^{(1)}\cdots x_i^{(n_i)}=u_i\}
    \subset \Delta^n\times \prod_i \Delta_{u_i}$$ onto 
    $\prod_i \{x_i^{(1)}\cdots x_i^{(n_i)}=0\}
    \subset \Delta^n\times \{0\}$ where $n = \sum n_i$.
    It is the product of the $1$-variable retractions,
    linearly decreasing the absolute values of $x^{(j)}_i$
    until one of them equals zero, for each $i$. 
    Note that $c$ fibers over the 
    deformation-retraction of the base 
    $\prod_i \Delta_{u_i}$ 
    to the origin, which radially contracts 
    each coordinate $u_i$ until it equals zero.

    The fiber $c^{-1}_t(p)$ over the origin $p$ of
    the given chart is a
    product of tori---each 
    torus factor $(S^1)^{n_i-1}$ is given by 
    $$|x_i^{(1)}|=\cdots = |x_i^{(n_i)}|=|u_i(t)|^{1/n_i}.$$
    Since one linearly decreases the absolute
    values $|x_i^{(j)}|$ for $j=1,\dots, n_i$ 
    simultaneously, they all hit the value
    zero at the same time.
    Finally, one patches the local retractions thus defined
    over products of snc strata, in a manner similar to \cite{clemens}.
    \end{proof}

Denote the local monodromy
operators on (co)homology about the component $D_i\ni 0$ 
by $$T_i\colon H^q(X_t,\bZ)\to  H^q(X_t,\bZ) \textrm{ or }
H_q(X_t,\bZ)\to H_q(X_t,\bZ).$$
A semistable degeneration
over a curve has unipotent monodromy, 
so by Proposition \ref{restrict}, the
$T_i$ are commuting unipotent operators.
Let $$N_i\coloneqq 
\log T_i\coloneqq (T_i-{\rm Id}) - 
\tfrac{1}{2}(T_i-{\rm Id})^2+
\tfrac{1}{3}(T_i-{\rm Id})^3-\cdots $$ 
be their nilpotent logarithms. Any linear combination
$N\coloneqq \sum a_iN_i$ for $a_i\in \bN$ positive integers, 
defines the same weight filtration $W_\bullet$ on 
$H^q(X_t,\bZ)$ or $H_q(X_t,\bZ)$ with weights
lying between $0$ and $2q$, resp.~$-2q$ and $0$.
Here we use that $X$ is K\"ahler, so that
that $W_\bullet$ is the weight filtration of the limit
mixed Hodge structure.

\begin{proposition}\label{specialization} 
Let $X\to Y$ be a
strict $D$-semistable degeneration, $X$ K\"ahler,
$0\in Y$ a point and
$t\in Y$ a nearby point. There is a canonical
specialization 
map ${\rm sp}\colon H_q(X_t,\bZ)
\to H_q(\Gamma(X_0),\bZ)$ for all $q$.
Furthermore, ${\rm sp}$ is surjective when $q=1$. \end{proposition}

\begin{proof}
Proposition \ref{clemens-collapse} produces a map 
$c_t\colon X_t\to X_0$.
Our goal is now to define a homotopy 
equivalence $\wX_0\to X_0$ from
a new topological space $\wX_0$, 
that admits a map
$\wX_0 \to \Gamma(X_0)$, 
so that we have the following diagram:
$$X_t \to X_0 \leftarrow \wX_0\to  \Gamma(X_0);$$
our map ${\rm sp}$ will be the composition 
$H_1(X_t,\bZ) \to H_1(X_0,\bZ) \xleftarrow{\sim} 
H_1(\wX_0,\bZ) \to H_1(\Gamma(X_0),\bZ)$. 

The topological space $\wX_0$ is 
built as follows. Let $X_0=\bigcup V_j$ be the 
decomposition of $X_0$ 
into its irreducible components and 
$V_J$ denote (an irreducible component
of) $\bigcap_{j\in J} V_j$.
Each stratum $V_J\subset X_0$ is 
locally a product of snc strata,
with local form $$\textstyle \prod_{i\in I} \{x_i^{(1)}=\cdots=x_i^{(n_i)}=0\}\subset 
\prod_{i\in I} \{x_i^{(1)}\cdots x_i^{(n_i)}=0\},$$
see Definition \ref{D-nodal}(\ref{d-snc-form}).
Thus, the local dual complex of $X_0$ at 
any point $v\in V_J^o$ in the open snc stratum
is a product of $(n_i-1)$-simplices, 
and these local dual complexes
form a product-of-simplices, i.e.~polysimplex 
bundle $\Sigma_J\to V_J$. Since
we assume that $f\colon X\to Y$ is strict $D$-semistable, the 
polysimplex bundle $\Sigma_J$ is in fact trivial: 
$\Sigma_J\simeq_{\rm homeo} V_J\times \prod_i \sigma_{n_i-1}$. 
We define $$\textstyle \wX_{0} \coloneqq \bigsqcup_J \Sigma_J/\sim $$
where $\sim$ is the equivalence relation given by the inclusion
of $\Sigma_J\vert_{V_{J'}}\hookrightarrow \Sigma_{J'}$ corresponding
to the face inclusion of the product of simplices 
$\prod \sigma_{n_i-1}$
corresponding to the inclusion of subsets $J\subset J'$.
See Figure \ref{homotopy-equivalence}.

\begin{figure}
\includegraphics[width=5in]{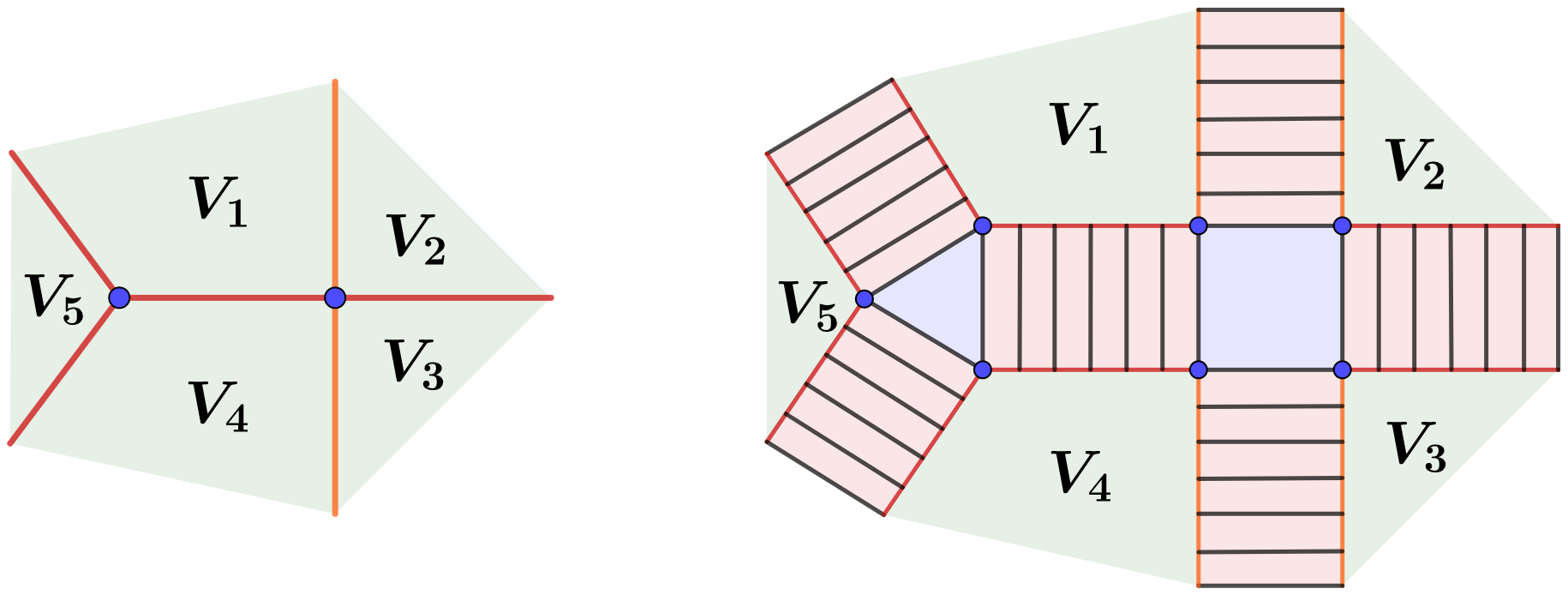}
\caption{Left: The central fiber of a semistable degeneration 
$X\to \Delta_{u_1}\times \Delta_{u_2}$ with five 
components $V_i\subset X_0$, $i=1,2,3,4,5$ of the 
central fiber, in green. Double loci extending
over $V(u_1)$ in 
red and double loci over $V(u_2)$ in orange. Local
equation of the smoothing of the lefthand triple locus
$V_{145}$ (in blue) is $x^{(1)}y^{(1)}z^{(1)}=u_1$ and
local equation of the smoothing
of the righthand codimension $2$
stratum $V_{1234}$ (in blue)
is $\{x^{(1)}y^{(1)}=u_1\}\times 
\{x^{(2)}y^{(2)}=u_2\}$. Right:
Topological space $\wX_0$ with
double loci $V_{12}$, $V_{23}$,
$V_{34}$, $V_{45}$, $V_{51}$
replaced with $1$-simplex bundles
$\Sigma_{12}$, $\Sigma_{23}$, $\Sigma_{34}$, $\Sigma_{45}$, $\Sigma_{51}$ and with $V_{145}$
and $V_{1234}$ replaced, respectively, with $2$-simplex
and $(1,1)$-polysimplex (i.e.~square) bundles $\Sigma_{145}$ and $\Sigma_{1234}$. The dual complex
$\Gamma(X_0)$ is the blue 
triangle glued to the blue square.
}
\label{homotopy-equivalence}
\end{figure}

Then $\wX_0$ has a homotopy 
equivalence to $X_0$
by decreasing the proportions of the 
polysimplices from side length $1$ to $0$.
Furthermore, there is a natural contraction map
$$\mu\colon \wX_0\to \Gamma(X_0)$$ 
given by collapsing
each open snc stratum to a point, which collapses 
the polysimplex bundle $\Sigma_J$ to a 
polysimplex
$\prod_i \sigma_{n_i-1}$. For instance,
in righthand side of 
Figure \ref{homotopy-equivalence},
the components are contracted to points,
and the interval bundles over double
loci are contracted to intervals.

Noting that
$H_q(X_0,\bZ)\simeq H_q(\wX_0,\bZ)$ 
by the homotopy
equivalence $\wX_0\to X_0$ we may
then define 
${\rm sp}\coloneqq 
\mu_*\circ (c_t)_*\colon 
H_q(X_t,\bZ)\to H_q(\Gamma(X_0),\bZ)$.

To prove surjectivity of $\rm{sp}$ when $q=1$, 
let $\alpha_0\in H_1(\Gamma(X_0),\bZ)$.
Then, we may lift $\alpha_0$ to an element
$\alpha\in H_1(\Gamma^{[1]}(X_0),\bZ)$
as the map $H_1(\Gamma^{[1]}(X_0),\bZ)\to 
H_1(\Gamma(X_0),\bZ)$ is surjective for any
polyhedral complex.

Starting with the vertices $v_j\in \alpha^{[0]}$,
we fix a lift $\widetilde{v}_j\in V_j^o$
in the open stratum of the corresponding 
component.  For an edge
$e_{jj'}\in \alpha^{[1]}$ connecting vertices
$v_j$ and $v_{j'}$, we lift to
a path $\widetilde{e}_{jj'}$ in $X_0$ connecting
$\widetilde{v}_j$ and $\widetilde{v}_{j'}$ 
and crossing the open stratum $V_{jj'}^o$. 
This produces a lift $\widetilde{\alpha}$
of $\alpha_0$ to a closed singular 
$1$-chain in $X_0$. To further lift
to $H_1(X_t,\bZ)$, consider the inverse
image $$\alpha_t\coloneqq 
\textstyle c_t^{-1}(\widetilde{\alpha}
\cap X_0^{\rm reg})$$ of the intersection
of $\widetilde{\alpha}$ with the regular locus. 

Consider the point $p_{jj'}=\widetilde{e}_{jj'}
\cap V^o_{jj'}$ where the edge 
$\widetilde{e}_{jj'}$ crosses
a double locus. There are two 
limit points of $\alpha_t$ on 
the circle $c_t^{-1}(p_{jj'})$.
We may connect these limiting points
by an arc of the circle so as
to lift the corresponding path $\widetilde{e}_{jj'}$ into
$X_t$. The result is a closed
$1$-chain in $X_t$ whose homology
class maps to $\alpha_0$
under ${\rm sp}$.
\end{proof}

\begin{proposition} \label{prop:gr-H_lim=gr-H-Gamma}
Let $f\colon X\to Y$ be 
a strict $D$-semistable morphism over an snc pair $(Y,D)$, 
$X$ K\"ahler, and let $0\in Y$. We have a canonical
isomorphism ${\rm gr}^{W}_0 H_1(X_t,\bZ)\simeq H_1(\Gamma(X_0),\bZ)$.
\end{proposition}

\begin{proof}
By the definition of the integral weight filtration,
we have $${\rm gr}^{W}_0H_1(X_t,\bZ)= H_1(X_t,\bZ)/\ker(N)$$
and $\ker(N) = \bigcap_{i\in I}\ker(N_i)$. 
By the second part of Proposition
\ref{specialization}, it suffices to prove that 
${\rm ker}({\rm sp}) = \ker(N)$ rationally.

Let $\Delta\to \Delta^k$, 
$u\mapsto (u,\dots,u)$
be the diagonal cocharacter.
Then the pullback of $X\to Y$ along
$\Delta$ is a $1$-parameter degeneration
$X_\Delta\coloneqq X\times_Y\Delta \to \Delta$,
whose singularities are analytically locally 
of the form
$$x_1^{(1)}\cdots x_1^{(n_1)}=\cdots = 
x_k^{(1)}\cdots x_k^{(n_k)}=u,$$ 
i.e.~a fiber product of snc singularities.
By subdividing into lattice simplices
the corresponding dual polysimplex 
$\prod_{i=1}^k \sigma_{n_i-1}$
to this stratum, we produce a toroidal 
resolution $X_\Delta'\to \Delta$ which
is a semistable degeneration, and for
which $\Gamma(X_{\Delta,0}')\simeq \Gamma(X_0)$
is a subdivision of the dual complex.

In particular, we have a canonical isomorphism
$H_1(\Gamma(X_{\Delta,0}'),\bZ)\simeq H_1(\Gamma(X_0),\bZ)$ and furthermore, the specialization map ${\rm sp}_\Delta$ 
for the $1$-parameter semistable
degeneration $X_\Delta'$ agrees with ${\rm sp}$
under this isomorphism. Thus, it suffices
to prove that the kernel of 
${\rm sp}_\Delta\colon H_1(X_t,\bQ)\to 
H_1(\Gamma(X_{\Delta,0}'),\bQ)$ is 
$W_{-1}\otimes \bQ$.
The result now follows from 
\cite[Sec.~1, p.~105]{clemens-schmidt}.
\end{proof}

\begin{proposition}\label{mono-1}
Let $f\colon X\to Y$ be a $D$-nodal degeneration,
with $0\in Y$ and $t\in Y$ a nearby point. Let
$T_i\colon H_1(X_t,\bZ)\to H_1(X_t,\bZ)$ 
be the monodromy
about a component $D_i\subset D$
passing through $0$. Then $$N_i(x) =  -
\sum_{\{j,j'\}} (x\cdot \widetilde{\gamma}_{jj'})
\gamma_{jj'}$$ 
where the indices $\{j,j'\}$ 
run through
all double loci $V_{jj'}$ 
over the general point of $D_i$, 
$\gamma_{jj'}\coloneqq [c_t^{-1}(p_{jj'})]$
and $\widetilde{\gamma}_{ij} \coloneqq  [c_t^{-1}(V_{jj'})]$
for any point $p_{jj'}\in V_{jj'}$, 
where $c_t\colon X_t\to X_0$ is the continuous map from Proposition \ref{clemens-collapse}.
\end{proposition}

\begin{proof} The formula follows
from a theorem of Clemens which computes the monodromy
of any semistable degeneration \cite[Thm.~4.4]{clemens-picardlefschetz}, though the case at hand is easier since we only
have simple nodes in the general fiber
over $D_i$ and the computation
is essentially the same as the Picard--Lefschetz 
formula. \end{proof}

This formula is compatible with the formula
$r_iB_i$ for the monodromy bilinear form of a shifted
matroidal degeneration $X(\uR,\mathscr{H})\to \Delta^k$.
Indeed, any double locus $V_{jj'}$ of the general fiber
over the $i$-th coordinate hyperplane of $\Delta^k$ 
has, by construction, vanishing cycle 
$\gamma_{jj'}={\bf x}_i\in \bfN\simeq {\rm gr}^W_{-2}H_1(X_t,\bZ)$. Thus, Proposition \ref{mono-1}
gives $$B_i(x,x)=\sum_{\{j,j'\}} (x\cdot \widetilde{\gamma}_{jj'})L(x, {\bf x}_i)\quad \textrm{for }x\in H_1(X_t,\bZ),$$
where $\widetilde{\gamma}_{jj'}\in H_{2g-1}(X_t,\bZ)$ 
is defined as above and $L$ is the principal polarization. 
But for all $\{j,j'\}$, we have
$(-\cdot \widetilde{\gamma}_{jj'})= L(-,\gamma_{jj'})$.
Thus, $B_i(x,x) = r_i {\bf x}_i^2$ where $r_i$
is the number of hyperplanes 
normal to ${\bf x}_i$ in 
the multiset $\mathscr{H}$.

\subsection{Resolution of the base change
of a nodal morphism}\label{sec:resolution}
The goal of this section is to prove
the following general theorem.

\begin{theorem}\label{can-res}
Let $\pi\colon Y'\to Y$ be a morphism, 
and let $f'\colon X'\to Y'$ be the base change
of a strictly $D$-nodal morphism 
$f\colon X\to Y$ along $\pi$. 
Suppose furthermore that $Y'$ 
is smooth and 
the reduction of $E\coloneqq \pi^{-1}(D)$
is an snc divisor. 

Then an ordering of the components of $E$, 
and an ordering of the components
$V_i$ over each $D_i$ determines, 
in a canonical manner, a relatively projective
resolution of singularities
$X'''\to X'$ for which the morphism 
$f'''\colon X'''\to Y'$ is strictly $E$-semistable,
and an intermediate partial 
resolution $X'''\to X''\to X'$ for which 
$f''\colon X''\to Y'$ is strictly nearly $E$-nodal.
\end{theorem}

This theorem can be viewed as an explicit special case
of the functoriality theorem for multivariable semistable
reductions, see \cite[Thm.~4.4]{alt}. The proof
works by observing that the base change 
is locally toroidal. This allows us to apply toroidal 
resolutions locally, which glue to a global resolution.

\begin{proof}
Define a bijection between the components of $E$
and the non-negative
integers \begin{align}
    \label{set-E}
\{1,\dots,\#\textrm{ components of }E\},\end{align}
increasing in the total order, so that any snc stratum
$E_J\coloneqq \bigcap_{j\in J} E_j$ defines a unique subset
of (\ref{set-E}).
Suppose that $E_J$ is an snc stratum
of codimension $n$ in $Y'$, 
so that $|J|=n$. Say $J=\{1,\dots,n\}$ for indexing convenience.
At any point in the open snc stratum $E_J^o$,
the hypothesis that $E=\pi^{-1}(D)$ is snc implies that
$\pi$ induces a local monomial transform 
\begin{align}\begin{aligned} \label{monomial-substitution}
u_1 &= w_1^{r_{11}}\cdots w_n^{r_{1n}} \eqqcolon  w^{\vec{r}_1}, \\ 
 & \cdots \\
u_k &= w_1^{r_{k1}}\cdots w_n^{r_{kn}} \eqqcolon w^{\vec{r}_k}. 
\end{aligned}\end{align}
where $u_i$ are a subset of local 
coordinates on $Y$, which cut out
the stratum $D_I=\bigcap_{i\in I}D_i$ 
into which $E_J$ maps,
and $w_j$ cuts out $E_j$.
By the hypothesis that $f$ is $D$-nodal, 
the base change $f'$ has a local 
form which is the product of a smooth morphism with
\begin{align}\label{monomial-substitution2}
\{x_1y_1 = w^{\vec{r}_1},\,
\cdots, \, x_m y_m = w^{\vec{r}_m}\}\mapsto (w_1,\dots,w_n), 
\end{align} up to relabeling the indices $\{1,\dots,m\}$
of the fiber components.

Note that $x_i$ and $y_i$ are local equations of 
components of $V_i\subset X_i$ over $D_i$. 
By convention, take $x_i$ to cut out
the component earlier in the total order (here we use smoothness of $V_i$
to ensure that $x_i=y_i=0$ is not a self-nodal locus of a component).

We will first
construct the partial resolution 
$f''\colon X''\to Y$ which is nearly
$E$-nodal. The equations (\ref{monomial-substitution2})
define a morphism of toric varieties. 
The domain of (\ref{monomial-substitution2})
is described by the normal fan of a 
polytope $P(\vec{r}_1,\dots,\vec{r}_n)$,
which we define now.

Let $b_i\colon \bR\to \bR^J\simeq \bR^n$ for $i=1,\dots,k$ 
be the $\bZ$-piecewise linear function
 $$b_i(z)\coloneqq \twopartdef{-z\vec{r}_i}{z \leq  0,}{0}{z\geq 0,}$$
and let $b(z_1,\dots,z_k)\coloneqq \sum_{i=1}^k b_i(z_i)$. Then the 
graph $\Gamma(b)\subset \bR^{k}\times \bR^J$
is the boundary of the polytope 
$P(\vec{r}_1,\dots,\vec{r}_n)\coloneqq 
\Gamma(b)+(\bR_{\geq 0})^J$---the monomials
$x_i$ and $y_i$ respectively correspond to the primitive 
integral vectors along the restriction
of the graph of $b$, to the positive- and negative 
$i$-th coordinate axis $\bR\subset \bR^k$, respectively.

The {\it bending parameter} of a piecewise linear 
function $b_0\colon \bR\to \bR^J$ at $z=z_0$ is 
defined by $\tfrac{\partial b_0}{\partial z}(z_0+\epsilon)-
\tfrac{\partial b_0}{\partial z}(z_0-\epsilon)\in \bR^J,$ 
for $\epsilon \ll 1$. Then the function
$b_i\colon \bR\to \bR^J$
above is uniquely characterized by the following properties:
\begin{enumerate}
\item $b_i(z)=0$ for $z\gg 0$,
\item $b_i$ only bends at $z=0$, and
\item the bending parameter at $z=0$ is $\vec{r}_i=r_{i1}e_1+\cdots +r_{in}e_n\in \bR^J$.
\end{enumerate}

Fix a very large integer $N\gg 0$. We may uniquely define a (continuous, piecewise linear)
function $c_i\colon \bR\to \bR^J$ for $i=1,\dots, m$ by the following properties:
\begin{enumerate}
\item $c_i(z)=0$ for $z\gg 0$,
\item $c_i(z)$ only bends at $z=jN+\ell$ 
for $j\in J$ and $\ell\in \{1,\dots,r_{ij}\}$ and
\item\label{3} the bending parameter 
at $z=jN+\ell$ is the basis vector $e_j\in \bR^J$.
\end{enumerate}

See Figure \ref{monomial-separation} for a depiction
of the domains of linearity in an example,
when $k=2$.
We define $c\colon \bR^k\to \bR^J$ 
by the formula $c(z_1,\dots,z_k)\coloneqq \sum_{i=1}^k c_i(z_i)$.
Then since each $c_i$ is convex, 
$\Gamma(c)$ is the boundary of a polytope 
$Q(\vec{r}_1,\dots,\vec{r}_n) \coloneqq \Gamma(c)
+(\bR_{\geq 0})^J$. 

\begin{figure}
    \centering
    \includegraphics[width=3.5in]{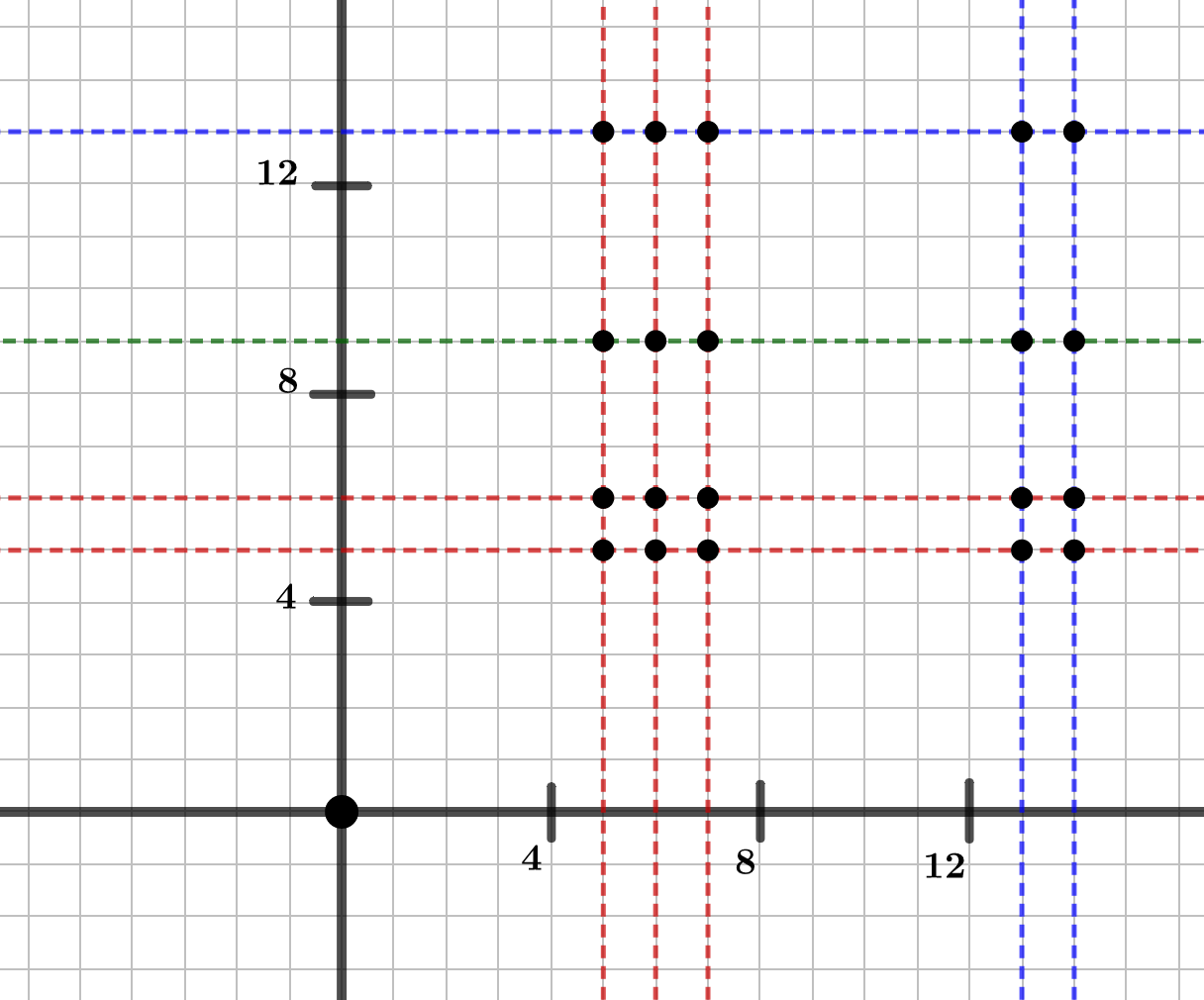}
    \caption{Domains of linearity of $b$ and $c\colon \bR^2\to \bR^3$. 
    The relevant monomial transform is $u_1=w_1^3w_3^2$ 
    and $u_2=w_1^2w_2w_3$. The bending
    locus of $b$ is depicted in solid black, while the bending
    locus of $c$ is depicted in dotted red, green, and blue lines. 
    The red, green, and blue lines have bending parameter 
    $e_1$, $e_2$, and $e_3$, respectively.
    Here, we took $N=4$.}
    \label{monomial-separation}
\end{figure}

Note that the normal fan $\fF_J$ 
of $Q(\vec{r}_1,\dots,\vec{r}_n)$ depends only on 
labelling of the indices
$J=\{1,\dots,n\}$. Furthermore, the 
normal fan of $Q(\vec{r}_1,\dots,\vec{r}_n)$ is a refinement
of the normal fan of $P(\vec{r}_1,\dots,\vec{r}_n)$ since 
the linear parts of $b_i(z)$ and $c_i(z)$
are the same for any $z<0$ or $z\gg 0$, 
e.g.~$z>(\#\textrm{ components of }E)\cdot N+\max r_{ij}$ suffices. 

Thus, in a neighborhood of a point $p\in (f')^{-1}(E_J)$, 
the normal fan of $Q(\vec{r}_1,\dots,\vec{r}_n)$ 
defines a toroidal birational morphism $X''_p\to X'_p$.
To check that this birational modification is globally
well-defined, it suffices to prove that these birational 
modifications are compatible with the incidences
$E_J\subset E_{J'}$ for some $J'\subset J$, and that they 
are compatible on overlapping charts
over a given stratum $E_J$. 
The latter is automatic since the
modification $X''_p\to X'_p$ depended only on the ordering of $J$
and the snc divisor $E$ has global normal crossings.

To check the compatibility between strata, 
the restriction of the morphism 
of fans 
$$\fF_J\to \prod_{j\in J} \bR_{\geq 0} e_j^\vee$$
to the coordinate subspace 
$\fF_{J'}\to  \prod_{j=J'} \bR_{\geq 0} e_j^\vee$ 
should be the normal fan
of the corresponding polytope for $J'$. Dually, 
in terms of the above defined PL function $c=c^J\colon 
\bR^k\to \bR^J$, 
if we consider the projection 
$p_{J,J'}\colon \bR^J\to \bR^{J'}$, then
the composition $$\bR^k\xrightarrow{c_J} \bR^J 
\xrightarrow{p_{J,J'}} \bR^{J'}$$
should agree with the PL function $c_{J'}$ 
and indeed it does---this projection simply forgets
the bending (\ref{3}) along any $z_i =  jN+\ell$ 
for the $j\in J\setminus J'$ and any $i\in \{1,\dots,m\}$.
In Figure \ref{monomial-separation}, this corresponds
to forgetting the colors indexed by $J\setminus J'$.

Hence, there is a globally well-defined toroidal 
birational modification $X''\to X'$ which is locally
defined by the morphism from the normal 
fan of $Q(\vec{r}_1,\dots,\vec{r}_n)$ to the normal
fan of $P(\vec{r}_1,\dots,\vec{r}_n)$.
It is furthermore relatively projective, 
since we defined it in terms of polytopes.

We claim that $f''\colon X''\to Y'$ has nearly 
$E$-nodal singularities.
We check that $f''\colon X''\to Y'$ has the 
desired local form of Definition \ref{D-nodal}(\ref{d-nearly-form}) 
by examining a neighborhood of a face of the polytope
$Q(\vec{r}_1,\dots,\vec{r}_n)$, i.e.~a neighborhood 
of a domain of linearity of the function 
$c\colon \bR^k\to \bR^J$. 
Such a domain of linearity is given by equations $$\textstyle
\bigcap_{i\in I} \{z_i = j_iN+\ell_i\}$$ for some subset 
$I\subset \{1,\dots,k\}$ and some indices
$j_i\in J$. In the neighborhood of such
a domain, the bending parameter of  $c_i(z_i)$ is 
$e_{j_i}$ and thus, the local equation of the morphism $f''$
is $$x_1y_1=w_{j_1},\quad \cdots \quad ,x_my_m=w_{j_m}$$ 
where $x_i$ and $y_i$ are the local
equations of the reduced union of components corresponding, respectively,
to the facets of $Q(\vec{r}_1,\dots,\vec{r}_n)$ given
by $z_i\leq j_iN+\ell_i$ and $z_i\geq j_iN+\ell_i$.
Thus, $f''$ has nearly $E$-nodal singularities, which are 
$E$-nodal if and only if
the $j_i$ are distinct, ranging over all possible strata over all $E_J$. 

Thus, we have completed our first goal: producing a 
birational modification $X''\to X'$ for which
$f''\colon X''\to Y$ is nearly $E$-nodal. But, as noted before, 
$X''$ may not be smooth, due to
the presence of the local form \begin{align}\label{cube}
x_1y_1=\cdots=x_m y_m=w,\end{align}
 and products thereof.
The fan of this local form may be described as follows: 
Let $[0,1]^m\subset \bR^m$ denote the unit cube
and let ${\rm Cone}\,[0,1]^m\subset \bR^{m}\times \bR$ denote 
the cone over the cube. Then, the morphism
to $\bC_w$ is given by the morphism of fans 
${\rm Cone}\,[0,1]^m\to \bR_{\geq 0}$ which is projection
to the last coordinate.

We may define a small,
regular resolution of (\ref{cube}) by subdividing 
the fan into standard affine cones,
in a manner which introduces no new rays. Note that the 
original rays, corresponding
to the components over $w=0$, are the cones over the $2^m$ 
vertices of the cube.
Equivalently, we must decompose the cube $[0,1]^m$ 
into lattice simplices of minimal volume
$(1/m!)$. 
Such a subdivision 
arises from a sequence of toric blow-ups, 
by blowing up the 
components of the fiber over $w=0$ 
of (\ref{cube}) in any order.
Blowing up the component
corresponding of a vertex of the cube $[0,1]^m$
produces a subdivision of the cube which inserts
all diagonals of the cube emanating from that
vertex. The resulting polyhedral cells are cones
over lower dimensional cubes. Further
blow-ups further subdivide these cones, until
all cells are standard simplices.

Thus, to define globally a projective  
subdivision, requires a total ordering
of the components of $f''\colon X''\to Y$ 
over each component $E_j\subset E$. Since (\ref{cube})
only involves a single coordinate $u$, 
it suffices to resolve by blowing up,
in the specified order, the
components over each $E_j=V(w_j)$.
Then, over a deeper stratum $E_J$,
these blow-ups are a product of blow-ups 
of the local forms (\ref{cube}), ranging over $j\in J$.
Thus, they induce the product resolution 
over deeper strata $E_J$.

The total ordering of all components $V_i$ over
$D_i$ for all $i\in I$
induces a total order on the components
over $E_j$. For instance,
the total order on the original
components induces a lexicographical ordering
on the components introduced by the partial
resolution $X''\to Y'$ over each $E_j$, 
which are naturally indexed by the top-dimensional
cells of a cuboid whose corners
are the strict transforms of components of the fibers
of $X'\to Y'$ over $E_j$, 
see Figure \ref{monomial-separation}.
Blowing up these components in order,
we produce a 
subdivision for each 
(singular) stratum over $E_j$,
giving a projective resolution
of singularities $X'''\to X''$.

Examining the cones of the resulting fan for $X'''$, 
we see that the local form for $f'''\colon X'''\to Y'$
is given by a product of morphisms of fans of the form 
$$\prod \,[{\rm Cone}\{0\leq z_1\leq \cdots 
\leq z_{m}\leq 1\}\to \bR_{\geq 0}]$$
(here, we allow $m=1$ to include factors which are
smoothings of nodes) with a smooth morphism. We deduce
that the morphism $f'''$ is $E$-semistable.
\end{proof}

\begin{remark} Suppose $\pi\colon Y'\to Y$ is a birational
modification. Consider the strict
transforms $E_i\coloneqq \pi_*^{-1}D_i$. For any stratum
$E_J\subset E_i$ 
contained in the strict transform (i.e.~$J\ni i$),
the local monomial transform (\ref{monomial-substitution})
is of the form $u_j=w^{\vec{r}_j}$ where
$w^{\vec{r}_j}$ does not involve the variable $w_i$
for all $j\neq i$, and $w^{\vec{r}_i}=
w_i\cdot (\textrm{a monomial
in }w_j\textrm{ for }j\neq i)$. Thus, in 
Theorem \ref{can-res}, nodes over $D_i$
are in natural bijection with the nodes over $E_i$
and indeed, $E_i^o\to \pi(E_i^o)\subset D_i^o$ 
is an isomorphism onto its image,
with the restriction of the map
$X'''\to X$ an isomorphism.
So, at least over $E_i^o$, the morphism 
$f'''\colon X'''\to Y'$ is $E$-nodal.

More generally, when $\pi$ is an alteration,
the nodes over $E_i^o$ are \'etale over
the nodes of $D_i^o$.
\end{remark}

\subsection{Resolution of the base change 
of a transversely shifted matroidal degeneration}
\label{sec:res-shifted}

Suppose that $X=X(\uR, \mathscr{H})$ is a 
transversely
shifted matroidal degeneration, so that, in particular,
$X$ is regular and $f\colon X\to \Delta^k$ has
$D$-nodal singularities (Prop.~\ref{X-smooth}). 
Then $f$ is strictly $D$-nodal
if and only if for each element 
${\bf x}_i\in \bfN$ of the matroid 
$\underline{R}$, there are at least two hyperplanes
with normal vector ${\bf x}_i$ (Prop.~\ref{strict-D}). 
In this case,
we may directly apply the resolution algorithm of 
Theorem \ref{can-res}. But in fact, even if
an irreducible component over $V(u_i)$ is self-nodal,
the two branches are not permuted by monodromy,
because it is possible to choose globally
a normal vector to a hyperplane $H\in \mathscr{H}$.
So the resolution algorithm of Theorem \ref{can-res} 
still works analytically, by consistently choosing
one of the two branches and performing the blow-up in 
local charts, as in the proof. 
However, since one does not blow up
global Weil divisors, it is unclear whether the
result is, in general, projective.

\begin{example}\label{ex:mon-sep}
Let $X(\uR, \mathscr{H})=X(2\mid b_1,b_2,b_3)$ be the 
transversely shifted matroidal
degeneration depicted in the righthand 
side of Figure \ref{fig:shifted-theta}.
Pass to a Veronese embedding for some large
$d\gg 0$ (we take $32$ times the principal polarization,
in the present example).

Consider the monomial base change $\Delta^2\to \Delta^3$
given by \begin{align*}
u_1 =w_1^2w_2, \quad
u_2 =w_2^4, \quad
u_3 =w_1^3w_2.
\end{align*}
The pullback and its 
nearly $E$-nodal resolution $X(32\mid c_1, c_2)$ are depicted
in Figure \ref{fig:partial-resolution}, where $E\coloneqq V(w_1w_2)$
is the reduced inverse image of $D\coloneqq V(u_1u_2u_3)$.

\begin{figure}
\includegraphics[width=6in]{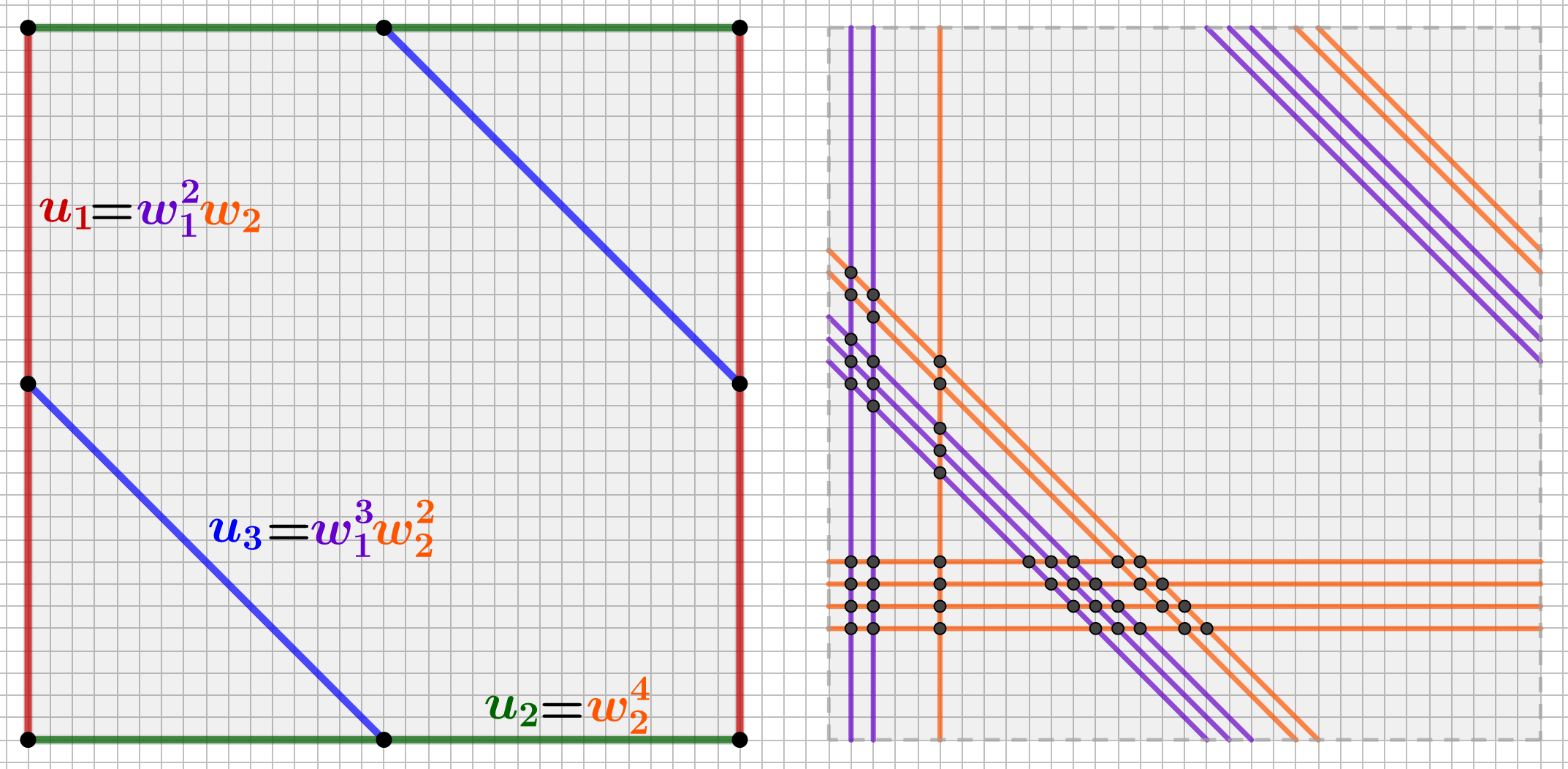}
\caption{Left: Bending of $b_1$, $b_2$, $b_3$
for $X(32\mid b_1,b_2,b_3)$ a shifted matroidal
degeneration. Right: Bending of $c_1$, $c_2$ for the 
nearly $E$-nodal partial resolution
of the base change, in purple, 
orange, respectively.
Grid points are $(\tfrac{1}{32}\bZ)^2$. We have
taken $N=4$, as in Figure \ref{monomial-separation}.}
\label{fig:partial-resolution}
\end{figure}

\begin{figure}
\includegraphics[width=4in]{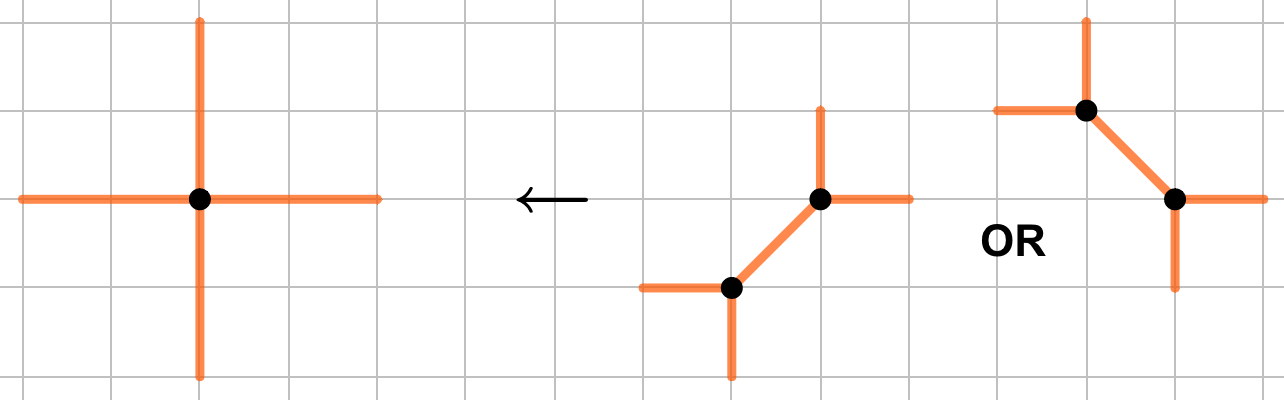}
\caption{The two possible small resolutions
for a $4$-valent intersection 
point of $2$-dimensional cuboids in
${\rm Bend}(c_i)$. The normal fan is
a subdivision of $[0,1]^2$.}
\label{fig:small}
\end{figure}

The original red hyperplane in the left 
of Figure \ref{fig:partial-resolution}
generates hyperplanes to its right, in the direction
of positive intersections with ${\bf x}_1=(1,0)$,
the green hyperplane generates hyperplanes above it,
in the direction of positive intersections with 
${\bf x}_2=(0,1)$,
and the blue hyperplane generates hyperplanes above
and to the left, in the direction of positive intersection
with ${\bf x}_3=(1,1)$. Purple hyperplanes (corresponding
to $u_1$), always precede orange hyperplanes (corresponding
to $u_2$), because of the ordering on the components 
$V(u_1)$, $V(u_2)$ of $E$.

The righthand figure is then the Mumford construction
which describes the nearly $E$-nodal partial resolution,
as in Theorem \ref{can-res}, of the base-changed 
Mumford construction.

In terms of Mumford constructions, 
the further small resolution
to an $E$-semistable morphism is a bit more 
complicated to describe.
Essentially, the relevant blow-ups 
resolve the $4$-valent intersection
points of the (monochromatic) cuboids 
in ${\rm Bend}(c_i)$ for $i=1,2$, 
into two $3$-valent intersection
points, in a manner which is locally 
of the form shown in Figure \ref{fig:small}.
\end{example}

\section{The second Voronoi fan and Alexeev's theorem} 
\label{sec:alexeev}

\subsection{The universal family
of abelian torsors with theta divisor}

One of the most celebrated
applications of the Mumford construction 
is the modular compactification
of the moduli space $\cA_g$ of PPAVs 
of dimension $g$, due to Alexeev \cite{alexeev}, 
building on work of Namikawa,
Nakamura, and Faltings--Chai 
\cite{nakamura75, nakamura77,
namikawa76, namikawa79, fc, alexeev-nakamura}; see 
\cite[Thm.~9.20]{namikawabook}.

In previous sections, we have extracted
from a section $\overline{b}_i\in 
H^0(\bT^g,\bZ{\rm PL}/\bZ{\rm L})$  on a torus
$\bT^g=\bfM_\bR/\bfM$, or its
PL lift $b_i\colon \bfM_\bR\to \bR$, an integral
bilinear form $B_i\in {\rm Sym}^2\bfM^\vee$. Here
we reverse this procedure, extracting
from a bilinear form $B_i$ a PL function $b_i$
with periodic bending locus. In this manner,
we produce both a canonical choice
of fan for $\cA_g$ (see Def.~\ref{def:fan}),
and a ``tautological'' Mumford construction over its cones. 
The procedure is straightforward:
we graph (a function closely related to) 
$B_i({\bf m},{\bf m})$ 
over the lattice points ${\bf m}\in \bfM$, 
take the convex hull of the corresponding integral
points, and take the unique PL function whose graph
is the boundary of this hull.

\begin{definition}\label{vor-cell}
Let $B\in \cP_g$ be a positive-definite 
symmetric bilinear form on $\bfM_\bR$. 
It defines a square-distance 
function $d_B$ on $\bfM_\bR$ 
by $x\mapsto B(x,x)$. The {\it Voronoi decomposition}
${\rm Vor}_B$ of $\bfM_\bR$ is the one 
whose maximal open polyhedral cells 
are defined as follows: $${\rm Vor}_{B,\,{\bf m}}=
\{x\in \bfM_\bR\,\big{|}\,d_B(x,{\bf m})
<d_B(x,{\bf m}')\textrm{ for all }{\bf m}'\in \bfM
\setminus {\bf m}\},$$
ranging over all ${\bf m}\in \bfM$. 
That is, the maximal cells are those points closer
(with respect to $d_B$) 
to one lattice point ${\bf m}\in \bfM$ than any other.

The {\it Delaunay decomposition} 
${\rm Del}_B$ is the polyhedral decomposition
of $\bfM$ whose cells are dual to the 
cells of the Voronoi decomposition, and whose 
vertices are $\bfM\subset \bfM_\bR$.
 \end{definition}

  \begin{figure}
\includegraphics[width=2.7in]{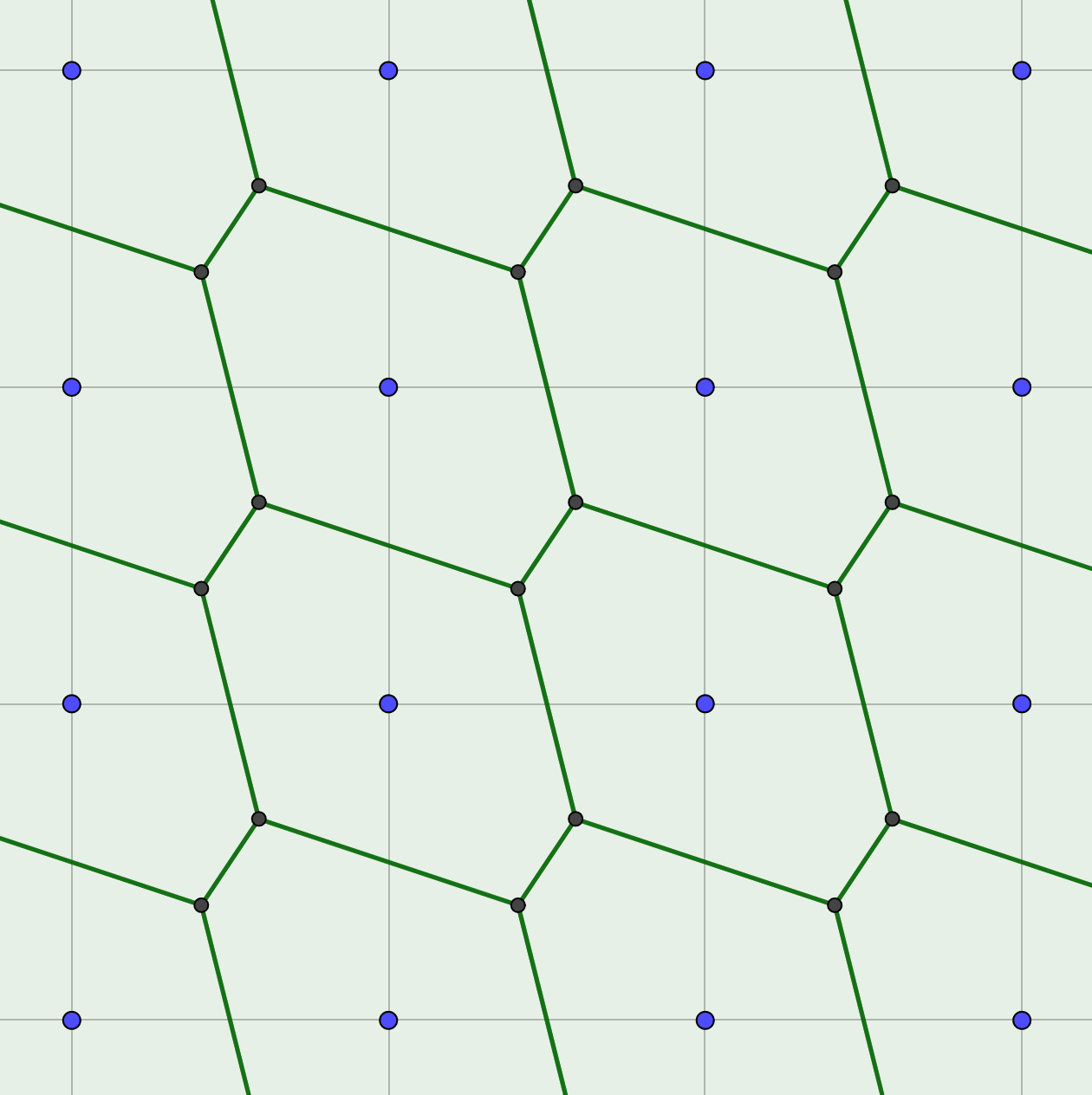}
\caption{Voronoi cells for 
$4{\bf x}^2+2{\bf x}{\bf y}+3{\bf y}^2$.
Lattice points $m\in \bfM$ in blue.
}
\label{fig:voronoi-cells}
\end{figure}

 \begin{example}
Consider the bilinear form $B$ on $\bfM_\bR\simeq \bR^2$
corresponding to the matrix $$B=\twobytwo{4}{1}{1}{3}$$
from Figure \ref{fig:tate6} and Example \ref{theta-1-param}.
The associated
square distance function is given by the quadratic form 
$4{\bf x}^2+2{\bf x}{\bf y}+3{\bf y}^2$. The Voronoi
cells are depicted in Figure \ref{fig:voronoi-cells}.
The corresponding Delaunay decomposition is depicted
in Figure \ref{fig:tate5}.
 \end{example}

 \begin{remark} When $B$ is degenerate, 
 the Voronoi cells are still defined,
 but they are of infinite volume, as they are invariant 
 under translation by the null subspace of $B$. \end{remark}
 
On the one hand, the Voronoi decomposition varies 
continuously with $B\in \cP_g^+$, 
and while its cells are polytopes,
they are not integral. On the other hand, 
the Delaunay decomposition has integral polytope cells,
which do not vary continuously,
but rather are constant along the 
relative interiors of the cones
of a fan:
 
 \begin{definition}\label{vor-fan} 
 The {\it second Voronoi fan} 
 $\mathfrak{F}_{\rm vor}$ is the polyhedral 
 decomposition of $\cP_g^+$ whose cones
 are the closures of loci on which the Voronoi decomposition is 
 combinatorially constant, or equivalently, on which the 
 Delaunay decomposition is constant. More precisely, 
 $B, B' \in \cP_g^+$ are in the relative interior $\tau^o$
 of the same cone $\tau\in \mathfrak{F}_{\rm vor}$ if 
 and only if  $B$ and $B'$ are connected by a path
 along which the Delaunay decomposition is constant.
 The {\it second Voronoi compactification} 
 is the toroidal compactification 
 (see Section \ref{sec:toroidal-extensions})
 $$\cA_g\hookrightarrow \overline{\cA}_g^{\rm vor}\coloneqq \overline{\cA}_g^{\fF_{\rm vor}}.$$
 \end{definition}
 
As required by Definition \ref{def:fan},
$\mathfrak{F}_{\rm vor}$ is 
invariant under the action of $A\in \GL_g(\bZ)$
 on $B\in \cP_g^+$ via $B\mapsto ABA^T$, as these transformations
 correspond to changes-of-basis of the lattice 
 $\bfM\simeq \bZ^g$. It is a theorem due to Voronoi
 that the number of $\GL_g(\bZ)$-orbits of cones is finite.

\begin{construction}[Mumford construction of second Voronoi type]
\label{voronoi-mumford}
Over each cone $\bB\in \fF_{\rm vor}$ intersecting
 $\cP_g^+$ there is a ``tautological''
 Mumford construction, which we will now define. 
 
 Let $B_i\in \bB$
 be primitive integral vectors generating the rays
 of $\bB$. Then, by considering the characteristic vector
 of the bilinear form $B_i$ mod $2$, 
 it is possible to choose {\it characteristic}
 linear forms $L_i\colon \bfM\to \bZ$ for which 
 \begin{align}\label{characteristic} {\bf m}
 \mapsto b_i({\bf m})\coloneqq 
 \frac{B_i({\bf m},{\bf m})-L_i({\bf m})}{2}\end{align}
 is integer-valued
 on $\bfM\simeq \bZ^g$. For instance, 
 we may take the coefficients
 of $L_i$ to be the diagonal entries of the matrix $B_i$
 in some basis.
 Then, there is a convex section 
 $\overline{b}_i\in 
 H^0(\bT^g, \bQ {\rm PL}/\bQ {\rm L})$ admitting a
 lift to $\bfM_\bR$ which agrees with the 
 above function $b_i$
 on $\bfM$ (in contrast to the sheaf 
 $\tfrac{1}{d}\bZ{\rm PL}$, sections
 of $\bQ$PL allow
 all $a_i\in \bQ$ in the notation of 
 Definition \ref{def:d-dicing}, 
 see \eqref{loc-zpl}).
 
 Since $b_i({\bf m})$ is integer-valued 
 on $\bfM$, the function
 $b_i$ is $\bZ{\rm PL}$ on each
 cell of the Delaunay decomposition that
 {\it generates} $\bfM$, in the sense
 that the differences between all
 vertices of the cell generate $\bfM$.
 When $g\leq 4$, every Delaunay cell is
 generating, but once $g\geq 5$, there
 are counterexamples,
 see e.g.~\cite[p.~796]{erdahl-ryshkov-emptysphere} and 
 \cite[Sec.~1.14, Ex.~1.15]{alexeev-nakamura}.

It is a simple verification from
Definition \ref{vor-cell} 
that $\bigcup_i {\rm Bend}(\overline{b}_i)$ 
is exactly  ${\rm Del}_{\sum r_iB_i}$ for any 
$(r_1,\dots,r_k)\in \bN^k$---this condition translates
into a condition that the bending
locus of the convex $\bQ$PL function which agrees
with ${\bf m}\mapsto \sum_{i=1}^k r_ib_i({\bf m})$ 
on $\bfM$ is the same for all $(r_1,\dots,r_k)\in \bN^k$.
Thus, the $\overline{b}_i$
are dicing. Furthermore, the additional condition 
of Construction \ref{mumford-polytope-2} is satisfied:
The associated bilinear forms $B_i$ span extremal
rays of a polyhedral cone in $(\cA_g)_{\rm trop}$.
Despite $b_i$ only being $\bQ{\rm PL}$,
we may still take the overgraph 
$\Gamma$ of the $b_i$ as a polytope, 
and form its normal fan, see Constructions 
\ref{singular-base}, \ref{mumford-polytope-2}.
Thus, we get a relatively proper 
extension of the universal family
$X^{\rm univ}(\mathbbm{b})\to 
\widetilde{\cA}_g^{\, \bB}.$ 
Note that not all fibers need
be reduced, see Remark \ref{non-reduced}.
\hfill $\clubsuit$
\end{construction}

\begin{construction}\label{torsor}
We now construct a torsor
$\cX_g^\star\to \cA_g$ over the universal abelian
variety $\cX_g\to \cA_g$ and an extension of it
over the second Voronoi compactification.

The issue begins in the interior $\cA_g$,
see e.g.~the discussion in 
\cite[Sec.~1]{grushevsky-hulek} and 
\cite[Sec.~19 and bottom of p.~209]{namikawa76}:~For 
a given abelian 
variety $(A,0,L)$ with origin
$0\in A$ and principal polarization
$L\in {\rm NS}(A)$, there are
$2^{2g}$ different possible $(-1)$-symmetric
lifts $\cL\in {\rm Pic}(A)$ of $L$. 
These lifts define naturally a torsor over 
the $2$-torsion subgroup
$A[2]$ and thus, on the universal 
family $\cX_g\to \cA_g$
we have a natural torsor 
${\rm Lifts}(L) \to \cA_g$ 
under the group scheme $\cX_g[2] \to \cA_g$ 
of relative $2$-torsion in $\cX_g\to \cA_g$.  
But ${\rm Lifts}(L)$ admits no section---it 
is impossible to globally lift $L$ to some $(-1)$-symmetric
$\cL\in {\rm Pic}(\cX_g/\cA_g)$ 
when $g \geq 2$. 

There are two ways
to resolve the issue: Either one passes to a finite
\'etale cover $\widetilde{\cA}_g\to \cA_g$ over which this
torsor is trivialized, or one defines a new universal family
$\cX_g^\star\to \cA_g$ of abelian torsors $(X,\cL)$, with 
a lift of the principal polarization
to a line bundle.

The family $\cX_g^\star$ will be, \'etale-locally
over $\cA_g$, isomorphic to $\cX_g\to \cA_g$.
Over an \'etale open chart $U_i\to \cA_g$
over which there is a lift $\cL$ of $L$, we have
a family $((\cX_g)_{U_i}, \,\cL_{U_i})\to U_i$. 
We may uniquely glue
these families over the double overlaps $U_i\cap U_j$
to produce a universal family $(\cX_g^\star,\cL)\to \cA_g$. 
Notably, the gluing of $(\cX_g)_{U_i}$
and $(\cX_g)_{U_j}$ may not respect the
origin sections, but must respect the lift $\cL$.

For a cone $\bB\in \fF_{\rm vor}$,
Construction \ref{voronoi-mumford} 
gives a Mumford construction
$X^{\rm univ}(\mathbbm{b})\to \widetilde{\cA}_g^{\, \bB}$. In the category of DM analytic stacks, this family
descends {\it as a family of polarized 
varieties} over an \'etale neighborhood
of the boundary strata of $\cA_g^{\bB}$. 
The reason is that $\overline{b}_i\in 
H^0(\bT^g, \bQ{\rm PL}/\bQ {\rm L})$ which define
the relevant polytopal Mumford Construction
\ref{mumford-polytope-2} are defined canonically
by the $B_i$---one may worry that some
non-canonicity is introduced
by the choice of the characteristic
linear form $L_i({\bf m})$ in (\ref{characteristic})
which determine the lifts $b_i$.
But a different choice $L_i\mapsto L_i+2L_i'$ produces 
the same section $\overline{b}_i$.

On the other hand, due to the shifts $L_i({\bf m})$,
the resulting Mumford construction
has no canonical origin section, see Remark \ref{rem-shifts}. 
Thus, the output
of Construction \ref{voronoi-mumford} does not glue
canonically (i.e.~in a manner independent of the choice
of $L_i$) to the universal family $\cX_g\to \cA_g$ (which
has an origin section), but rather to the universal family 
$\cX_g^\star\to \cA_g$ (which has a canonical lift
of the principal polarization). If one were to take $L_i=0$
in (\ref{characteristic}), we would retain a canonical
origin point, but the principal polarization
may not lift to a theta divisor, 
see \cite{namikawa76}.

By their canonicity and the uniqueness of gluings, 
the Mumford constructions of Construction
\ref{voronoi-mumford} are 
compatible between adjacencies of cones
in $\fF_{\rm vor}$. Thus, we may glue them
via the unique gluings respecting the lift of $L$,
to produce a proper extension $$\overline{\cX_g^\star}^{\rm vor} \coloneqq \cX_g^\star\cup \textstyle 
\bigcup_{\,\bB\,\in\, \GL_g(\bZ)\backslash \fF_{\rm vor}} 
X^{\rm univ}(\mathbbm{b})\to \overline{\cA}_g^{\rm vor}
.$$
In summary, $\overline{\cX_g^\star}^{\rm vor}$ admits
a relatively projective, 
surjective morphism (a priori, just in the 
category of DM analytic stacks)
$$f_{\rm vor}\colon \overline{\cX_g^\star}^{\rm vor}\to 
\overline{\cA}_g^{\rm vor},$$
extending the universal family of abelian torsors with
lift of principal polarization.
It follows from Serre's GAGA for 
Deligne--Mumford stacks, see \cite[Cor.~5.13]{toen}, that
$ \overline{\cX_g^\star}^{\rm vor}$ is a DM algebraic
stack and $f_{\rm vor}$ is
projective. For instance, $f_{\rm vor}$
becomes a morphism of projective schemes after taking
the pullback toroidal compactification of 
an appropriate \'etale cover. 

Our construction also produces 
an extension $$f_{\rm vor}\colon 
(\overline{\cX_g^\star}^{\rm vor},
\overline{\Theta}_g^{\rm vor})\to 
\overline{\cA}_g^{\rm vor}$$
of the universal pair $(X,\Theta)$. 
Here, the theta divisor
$\Theta \in |\cL|$ is the unique element
of the linear system. Note that $\Theta$
extends as an effective, relatively ample divisor over 
$\overline{\cA}_g^{\rm vor}$ 
as the vanishing locus of the unique weight 
$w=1$ theta function
$\Theta=V(\Theta_{(0/1,\dots,0/1)})$ of 
Construction \ref{mumford-polytope}.
\hfill $\clubsuit$
\end{construction}

 \begin{example}[Second Voronoi fan for $g\leq 6$]
 \label{voronoi-ex}
 We now describe, in varying levels of detail, the second
 Voronoi fan of $\cA_g$ for small dimensions, and the
 extension of the universal family over it, defined
 by Construction \ref{torsor}. 
 This line of Russian mathematical inquiry is notable 
 for extending across more than a century.

 \begin{figure}
\includegraphics[width=4.5in]{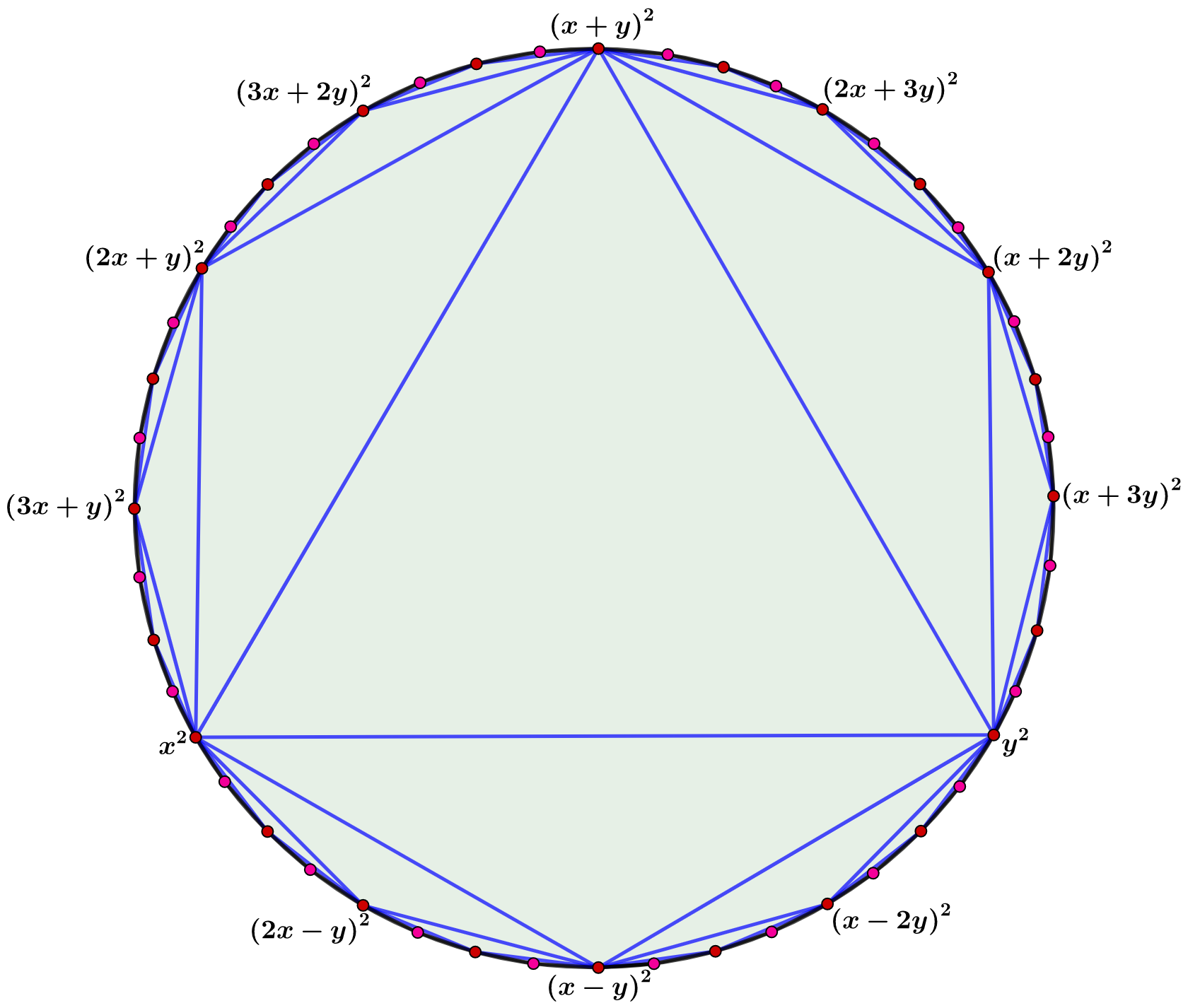}
\caption{Projectivization of the second Voronoi 
fan decomposition of $\cP_2^+$. Cones of dimension
$3$, $2$, $1$ in green, blue, red, respectively.
}
\label{fig:second-voronoi}
\end{figure}
 
\fbox{{$g=1$:}} Here, any fan for $\cA_1$ is the same, 
and there is only one Voronoi cone,
corresponding to the ray 
$\bR_{\geq 0}\{x^2\}\subset \cP_1^+\simeq \bR_{\geq 0}$. The 
corresponding Delaunay decomposition is Figure \ref{fig:tate3},
and the resulting Mumford construction 
is the Tate curve. The universal family 
$$\overline{\cX_1^\star}^{\rm vor}\to 
\overline{\Sp_2(\bZ)\backslash \cH_1}^{\rm vor}
\simeq \bP(4,6)$$ is the extension
of the universal elliptic curve by a nodal elliptic curve.
Over the coarse space $\bP^1_j$ of $\bP(4,6)$, 
i.e.~the $j$-line, the 
nodal curve fibers over $j=\infty$.\smallskip

\fbox{$g=2$:} Here, the second Voronoi fan is 
$\fF_{\rm vor} = \GL_2(\bZ)\cdot 
\bR_{\geq 0}\{{\bf x}^2, {\bf y}^2,
({\bf x}+{\bf y})^2\}$.
See Figure \ref{fig:second-voronoi}.
Thus, $\fF_{\rm vor}$ is the orbit 
of a single cographic cone, associated to the theta graph,
with the lower dimensional faces corresponding to contractions
of the theta graph (caveat lector:~edge contractions of $G$
give, in the sense of matroids, 
deletions of the cographic matroid $M^*(G)$).

There is one orbit each of 
$3$-, $2$-, $1$-, and $0$-dimensional cones, corresponding
respectively to the cographic cones
of the theta graph, the wedge of two circles,
a single circle, and a point. Over
a maximal, $3$-dimensional cone, the universal family
$\overline{\cX^\star_2}^{\rm vor}
\to \overline{\cA}_2^{\rm vor}$
is extended by a Mumford construction
$X(\mathbbm{b})\to \Delta^3$ 
isomorphic to Example \ref{theta-3-param}.
The reduction theory, i.e.~analysis of 
$\GL_2(\bZ)$-equivalence classes,
of positive-definite bilinear forms of rank $2$,
goes back at least to work of 
Fricke--Klein \cite{klein}; see Vallentin
\cite[Ch.~2]{Vallentin2003SphereCL} for some historical
discussion.
\smallskip

\fbox{$g=3$:} Here, the second Voronoi fan is 
$\fF_{\rm vor} = \GL_3(\bZ)\cdot 
\bR_{\geq 0}\{{\bf x}_i^2,({\bf x}_i-{\bf x}_j)^2\}$ 
for $1\leq i<j\leq 3$. 
There is an analogous Voronoi cone in any rank $g$, called
{\it Voronoi's principal domain of the first type}.
It is the graphic cone $\bB_{M(K_{g+1})}$ associated to
the graphic matroid $M(K_{g+1})$ of the complete 
graph $K_{g+1}$ on $g+1$ vertices. 
The number of $\GL_3(\bZ)$-orbits of cones of dimensions 
$6$, $5$, $4$, $3$, $2$, $1$, $0$ are, respectively,
$1$, $1$, $2$, $2$, $1$, $1$, $1$. Over
the maximal, $6$-dimensional cone, the universal family
$\overline{\cX^\star_3}^{\rm vor}
\to \overline{\cA}_3^{\rm vor}$
is extended by a Mumford construction
$X(\mathbbm{b})\to \Delta^6$ associated to the 
cographic or graphic matroid of the complete graph $K_4$ 
(note that $M(K_4)\simeq M^*(K_4)$ 
since $K_4$ is a planar, self-dual graph). 
See Example \ref{cographic-ex}.
\smallskip

\fbox{$g=4$:} Here, the second Voronoi fan is 
$\fF_{\rm vor} = \GL_4(\bZ)\cdot \{\bB_{\rm black}, \,
\bB_{\rm grey},\,
\bB_{\rm white}\}$, see for instance
\cite[Sec.~4.4.1]{Vallentin2003SphereCL}.
That is, there are three $\GL_4(\bZ)$-orbits
of maximal, $10$-dimensional cones of $\fF_{\rm vor}$.
The original computation goes back to
Delaunay \cite[Thm.~III]{delaunay1929},
who found all but one
of the $\GL_4(\bZ)$-orbits of
cones of $\fF_{\rm vor}$, and Shtogrin 
\cite{shtogrin}, who filled the gap.

The cone $\bB_{\rm black}$ is Voronoi's principal
domain of the first type, but unlike for $g\leq 3$,
it no longer forms a fundamental domain for the action
of $\GL_4(\bZ)$. It is a matroidal cone,
associated to the graphic matroid $M(K_5)$
of the complete graph $K_5$.
This cone is {\it not} cographic---the 
dual of a graphic matroid is graphic
if and only if the graph is planar, 
and $K_5$ is not planar.

The cones $\bB_{\rm grey}$ and $\bB_{\rm white}$
are simplicial, but are
not matroidal---they both have one ray 
generated by the positive-definite quadratic form
giving the $D_4$-lattice, whereas all rays
of a matroidal cone are quadratic forms of rank $1$.
There is one additional maximal, 
matroidal cone $\bB_{M^*(K_{3,3})}$ of dimension $9$. It is
the matroidal cone of the 
cographic matroid $M^*(K_{3,3})$
of the complete bipartite graph $K_{3,3}$ and 
is the facet shared between two white cones.

The number of orbits of cones of dimensions $10$,
$9$, $8$, $7$, $6$, $5$, $4$, $3$, $2$, $1$, $0$ are, 
respectively,
$3$, $4$, $7$, $11$, $11$, $9$, $7$, $4$, $2$, $2$, $1$,
see Hulek--Tommasi \cite[p.~232]{huleka4}.
In particular, the toroidal compactification
$\overline{\cA}_4^{\rm vor}$ has $2$ boundary divisors
and $3$ zero-dimensional strata.\smallskip

\fbox{$g=5$:}
Correcting the nearly complete computations of
Baranovskii--Ryshkov \cite{ryshkov} 
to find one missing case,
Peter Engel \cite{engel} verified by computer that
there are $222$ maximal, $15$-dimensional Voronoi
cones for $g=5$. There are $9$ rays, giving the
boundary divisors of $\overline{\cA}_5^{\rm vor}$.
Dutour Sikiri\'c et al.\ \cite{sikiric} 
proved that there are $110305$ total 
$\GL_5(\bZ)$-orbits
of cones in $\fF_{\rm vor}$ ({\it loc.~cit.}~give
a slightly smaller number, as they only count
cones which intersect $\cP_5$).

There are $4$ maximal, matroidal cones, of
dimensions $15$, $12$, $12$, $10$.
The first of these is Voronoi's principal
domain of the first type $\bB_{M(K_6)}$
and the last of these is the matroidal cone 
$\bB_{\underline{R}_{10}}$ associated
to the Seymour--Bixby matroid, see 
Example \ref{seymour-bixby}. The two maximal, 
matroidal cones of dimension $12$ are the 
cographic cones of two trivalent
genus $5$ graphs (one of which is
the $1$-skeleton of a cube). 
\smallskip

\fbox{$g=6$:} By work of Danilov--Grishukhin
\cite[Sec.~9]{danilov-grishukhin},
there are $11$ maximal matroidal cones, 
with $8$ cographic of dimension $15$, and the 
remaining three of dimensions $21$, $16$, $12$.
Respectively, these are the graphic cone of
$K_7$ and two matroidal cones, associated
to regular matroids on $16$ and $12$ 
elements which are neither
graphic nor cographic.
The number of orbits of maximal cones is unknown,
but exceeds $567,613,632$ by 
computations of Baburin--Engel \cite{baburin}.
 \end{example}

\begin{remark}\label{torsor-remark}
Construction \ref{torsor} shows that 
there exists a canonical element
of the group 
$H^1(\cA_g, \cX_g[2])$
giving the abelian torsor $\cX_g^\star$. 
We have an isomorphism with 
the group cohomology
$H^1(\cA_g, \cX_g[2])\simeq 
H^1({\Sp_{2g}}(\bZ), ({\bZ}/2{\bZ})^{2g})$ where
${\rm Sp}_{2g}(\bZ)$ acts on 
$(\bZ/2\bZ)^{2g}$ by the standard representation.
Furthermore, the class of $\cX_g^\star$
is nontrivial for $g\geq 2$, 
cf.~\cite{grushevsky-hulek}. It is natural
to ask whether for $g \geq 2$ the (\'etale)
Tate--Shafarevich
group $H^1(\cA_g, \cX_g)$ satisfies 
$H^1(\cA_g, \cX_g)\simeq \bZ/2\bZ$,
 where we view $\cX_g$ as a group
 scheme over $\cA_g$. An affirmative answer
 would show that $\cX_g^\star$ is the only
 abelian torsor under the universal 
 abelian variety for $g \geq 2$.
 
 The question of whether $H^1(\cA_g, \cX_g)\simeq \bZ/2\bZ$ 
 is equivalent to the question of whether we have 
 $H^1({\Sp_{2g}}(\bZ), (\bZ/2\bZ)^{2g})\simeq \bZ/2\bZ$, 
 as the map $$H^1({\rm{Sp}}_{2g}(\bZ), (\bZ/2\bZ)^{2g}) 
 \simeq H^1(\cA_g, \cX_g[2]) \to H^1(\cA_g, \cX_g)$$ is an 
 isomorphism. To see this, note first that by 
 \cite[Prop.~XIII.2.3]{raynaud}, the group $H^1(\cA_g, \cX_g)$ is 
 torsion. Second, for each $n \in \bN$, the natural map 
 $H^1(\cA_g, \cX_g[n]) \to H^1(\cA_g, \cX_g)[n]$ is an isomorphism 
 by the long exact sequence in cohomology arising from the 
 short exact sequence $$0 \to \cX_g[n] \to 
 \cX_g \xrightarrow{\cdot n} \cX_g \to 0;$$ 
 moreover, the group $H^1(\cA_g, \cX_g) = 
 H^1({\rm{Sp}}_{2g}(\bZ), \bZ^{2g})$ is $2$-torsion, 
 because ${\rm{Sp}}_{2g}(\bZ)$ contains an element that 
 acts as $-1$ on $\bZ^{2g}$. 
 
  A computation via the description
  $\SL_2(\bZ)=(\bZ/4\bZ)*_{(\bZ/2\bZ)} (\bZ/6\bZ)$ 
  shows that
  $H^1(\cA_1, \cX_1)\simeq H^1({\rm{SL}}_2(\bZ), (\bZ/2\bZ)^2)$ 
  is isomorphic to $\bZ/2\bZ$, even though 
  $\cX_1^\star \simeq \cX_1$. Thus, there exists a 
  universal non-trivial torsor under 
  $\cX_1 \to \cA_1$, i.e.\ a family of genus one 
  curves over $\cA_1$ with no section whose 
  Jacobian is $\cX_1 \to\cA_1$.  We do not know a 
  geometric construction of this family.
  \end{remark}

\begin{definition}
Let $(X,D)$ be a pair of a projective variety and
a $\bQ$-divisor $D$. We say that $(X,D)$ is {\it KSBA-stable}
if:
\begin{enumerate}
    \item the pair $(X,D)$ has slc singularities
    (see e.g.\ \cite{kollar_book}), and 
    \item $K_X+D$ is $\bQ$-Cartier and ample.
\end{enumerate}
\end{definition}

\begin{proposition}\label{ksba-extension}
Let $\epsilon>0$ be a sufficiently
small positive rational number. 
There is a canonical
stacky toroidal compactification
$\cA_g\hookrightarrow \overline{\cA}_g^{\rm nvor}$ admitting a partial coarsening $\overline{\cA}_g^{\rm nvor}\to \overline{\cA}_g^{\rm vor}$ along which 
every fiber of the normalized pullback 
$(\overline{\cX_g^\star}^{\rm nvor},\epsilon \overline{\Theta}_g^{\rm nvor})\to 
\overline{\cA}_g^{\rm nvor}$
of 
Construction \ref{torsor}
is a KSBA-stable pair. 
\end{proposition} 

The toric stack structure corresponds
to ``toric stacky data'' in the sense of 
\cite[Def.~4.1]{tyomkin}:~For the cones
$\bB\in \fF_{\rm vor}$ we require the additional
data of $\GL_g(\bZ)$-equivariant,
finite index sublattices
of $\bR\bB\cap {\rm Sym}_{g\times g}(\bZ)$,
which are compatible along face inclusions
of cones in $\fF_{\rm vor}$.

\begin{proof}[Sketch.] 
It follows from
a mild generalization of
Proposition \ref{k-triv}
to affine 
toric bases that 
in a Mumford fan Construction \ref{singular-base},
the fibers are reduced with trivial canonical
bundle $K_X\simeq \cO_X$
and slc singularities,
whenever the slices
$\cS_{(r_1,\dots,r_k)}$ of the normal fan,
for $r_1B_1+\cdots+r_kB_k\in {\rm Sym}_{g\times g}(\bZ)$ integral,
are integral tilings of $\bfN_\bR$.
For the Mumford Construction \ref{voronoi-mumford}
of second Voronoi type, $\cS_{(r_1,\dots,r_k)}$ is, 
up to translation, 
the image of the Voronoi decomposition 
${\rm Vor}_B$ under the map 
$N_\bR\colon \bfM_\bR\to \bfN_\bR$ corresponding
to $B=\sum_{i=1}^k r_iB_i$.

In general, the desired integrality condition
fails for $g\geq 5$, see \cite[Sec.~1.14]{alexeev-nakamura}, exactly because the $b_i$ of 
Construction \ref{torsor}
may only be $\bQ{\rm PL}$.
But we may define canonically
a finite index sublattice of $\bR\bB\cap {\rm Sym}_{g\times g}(\bZ)$ as the one generated
by the linear combinations $r_1B_1+\cdots+r_kB_k$
for which the slices $\cS_{(r_1,\dots,r_k)}$ {\it are}
integral polytopes. This defines
toric stacky data, which is $\GL_g(\bZ)$-equivariant, and so determines a DM stack
$\overline{\cA}_g^{\rm nvor}$ admitting
$\overline{\cA}_g^{\rm vor}$ as a partial coarsening.
See also \cite{molcho}. The normalized
pullback of the Mumford
Construction \ref{voronoi-mumford}
along the partial coarsening map 
$\overline{\cA}_g^{\rm nvor}\to\overline{\cA}_g^{\rm vor}$ is now flat with reduced fibers.

The log canonical centers of a fiber $X^n$ 
of this pullback are exactly
the toric strata of the Mumford construction.
Given an effective divisor $D\subset X^n$,
there is an $\epsilon\ll 1$ for which $(X^n,\epsilon D)$ 
defines a KSBA-stable pair if and only if $D$ contains
no log canonical centers, i.e.~toric strata.
In fact, in our setting, 
any $\epsilon\leq 1$ suffices, 
see \cite[Thm.~3.10]{a}, 
generalizing \cite[Thm.~17.13]{kollar-shaf}.

We claim this property follows from the definition
of $\Theta_{(0/1,\,\dots,\,0/1)}$. The key observation is
that, for every vertex of a polyhedral face  
$F\subset {\rm Del}_B$, $B\in \bB^\circ$, the restriction
$\Theta_{(0/1,\,\dots,\,0/1)}\vert_{Y_F}$ of the theta
divisor to the stratum $Y_F\subset X^n$ 
is a section of a (toric) line bundle, for which 
the coefficient of any monomial corresponding 
a vertex of $F$ is nonzero, i.e.~lies in $\bC^*$.
This property ensures that the restriction
of the theta divisor contains no 
toric strata of $Y_F$.
\end{proof}

We remark that the procedure
of choosing a finite index subcone
of $\bB\cap {\rm Sym}_{g\times g}(\bZ)$
is unnecessary for matroidal cones,
as in this setting,
the Voronoi cell $\cS_{(r_1,\dots,r_k)}$
is always integral---it is a Minkowski sum
of integral line segments $[0,r_i{\bf x}_i]$.

\begin{theorem}[{\cite[Thm.~1.2.17]{alexeev}}] \label{KSBAvor}
For $\epsilon\ll 1$, 
$(\overline{\cX_g^\star}^{\rm nvor},
\epsilon\overline{\Theta}_g^{\rm nvor})
\to  \overline{\cA}_g^{\rm nvor}$ is the universal
family over the normalization of the KSBA compactification 
of the space of KSBA-stable pairs $(X,\epsilon \Theta)$, 
with $X$ a torsor under a $g$-dimensional PPAV and 
$\Theta\subset X$ the theta divisor. 
\end{theorem}

\begin{proof}[Sketch.]
We will only address the equality
of coarse spaces.

By Proposition \ref{ksba-extension},
there is a classifying morphism 
$c\colon \oA_g^{\rm vor}
\to \oA_g^{\Theta}$
where the latter is, by definition, the closure, 
taken
with reduced scheme structure,
of the coarse space of pairs 
$(X,\epsilon \Theta)$ 
as in Construction \ref{torsor} in the 
moduli space of
KSBA-stable pairs \cite{ksb, a, kollar_book}. 
By Zariski's main theorem and the normality 
of toroidal compactifications, it suffices to check
that $c$ is finite.

It is easy to see that $c$ defines a morphism
over the Baily--Borel compactification $\oA_g^{\rm BB}$,
e.g.~by considering the Albanese variety
of the normalization of any component
of $(X,\epsilon \Theta)$. 
So if $c$ contracted some curve,
this curve would lie in a fiber of the morphism
$\oA_g^{\rm vor}\to \oA_g^{\rm BB}$. Any such curve
admits an algebraic deformation to a union of $1$-dimensional
torus orbits---first move the image
point in $\oA_g^{\rm BB}$ to the deepest cusp,
then apply the torus action.
Thus, $c$ would contract some $1$-dimensional
toric boundary
stratum $(\bP^1,0,\infty)\to \oA_g^{\rm vor}$. 
But, for any cone $\bB\in \fF_{\rm vor}$, 
the combinatorial types of the KSBA-stable
fibers over $0$
and over $u\in \Delta^*\subset \bP^1$ are distinct, 
by 
Construction \ref{voronoi-mumford}.
It follows that $c$ contracts no algebraic curves,
and hence is finite.
\end{proof}

A similar strategy was employed
in \cite[Thm.~1, Thm.~5.14]{alexeev-engel} 
to prove the semitoroidality
of certain KSBA compactifications 
of the moduli of polarized K3 surfaces.

\subsection{Algebraicity and projectivity}

We now analyze under what circumstances an extension
of the universal family $\cX_g\to \cA_g$ or 
$\cX_g^\star\to \cA_g$ of principally polarized
abelian varieties or torsors, by
Mumford constructions, are 
either algebraic or projective. 

\begin{proposition}\label{algebraicity}
Let $\bB\subset \cP_g^+$ be a rational 
polyhedral cone and
$\cS$ be a fan satisfying the hypotheses
of Construction \ref{singular-base}. 
The corresponding Mumford
construction 
$f \colon X^{\rm univ}(\cS)\to
\widetilde{\cA}_g^{\,\bB}$
is a proper, equidimensional 
morphism of algebraic spaces. 

Suppose, furthermore, that 
$f\colon X^{\rm univ}(\mathbbm{b})\to 
\widetilde{\cA}_g^{\, \bB}$
is a polytopal Mumford construction,
as in Construction \ref{mumford-polytope-2}.
Then $f$ is \'etale-locally 
projective.
\end{proposition}

\begin{proof}
By replacing $\widetilde{\cA}_g$ with
a suitable further cover, 
we may assume that the distinction between $\cX_g$ and 
$\cX_g^\star$ is erased.
Let $\fF$ be a common refinement of the fans 
$\Gamma\cdot \bB$ and $\fF_{\rm vor}$ 
whose support is $\Gamma\cdot \bB$,
for $\Gamma\subset \GL_g(\bZ)$ the Levi
quotient.
Then we have morphisms $$\widetilde{\cA}_g^{\, \bB}
\leftarrow \widetilde{\cA}_g^{\,\fF}\to 
\widetilde{\cA}_g^{\rm vor}$$ and we may pullback
the (a priori) analytic
family $X^{\rm univ}(\cS)\to 
\widetilde{\cA}_g^{\, \bB}$ and 
the algebraic universal family 
$\widetilde{\cX}_g^{\rm vor}\to 
\widetilde{\cA}_g^{\rm vor}$ 
(see Construction \ref{torsor})
to produce two families
$X^{\rm univ}\,\!'(\cS)$ and 
$X^{\rm vor}\to \widetilde{\cA}_g^{\, \fF}$ 
in the analytic and algebraic categories, respectively.

Taking a common refinement of 
the fans defining $X^{\rm univ}\,\!'(\cS)$ and $X^{\rm vor}$ 
we may dominate $X^{\rm univ}\,\!'(\cS)$ and 
$X^{\rm vor}$ by a common (universal) 
Mumford construction 
$\wX\to \widetilde{\cA}_g^{\,\fF}$. Since
$\wX\to X^{\rm univ}\,\!'(\cS)$ 
and $\wX\to X^{\rm vor}$ 
are both toroidal morphisms, 
we can connect 
$X^{\rm vor}\dashrightarrow X^{\rm univ}\,\!'(\cS)$ 
by a sequence
of toric modifications, 
all of which are
algebraic. 
We deduce that $X^{\rm univ}\,\!'(\cS)$
is an algebraic space.
In turn, its contraction
$X^{\rm univ}(\cS)$ is a proper
algebraic space over $\widetilde{\cA}_g^{\,\bB}$. 

Finally, we address the case of a polytopal Mumford
construction
$f\colon X^{\rm univ}(\mathbbm{b})\to 
\widetilde{\cA}_g^{\,\bB}$.
Then $f$ is a morphism of algebraic spaces by what we have just proved, and by (a generalization of) Theorem \ref{poly-is-fan} that compares the polytope and fan Mumford constructions over $\widetilde{\cA}_g^{\,\bB}$. 
Moreover, 
$f$ is 
analytically-locally projective over 
$\widetilde{\cA}_g^{\, \bB}$, see Construction \ref{mumford-polytope-2}.
The second statement of the proposition is therefore  a consequence of the following general fact: 
If a separated and finitely presented morphism
of algebraic spaces $X\to Y$ is analytically-locally projective, then it is \'etale-locally projective.

The proof follows
from Artin approximation. 
Indeed, the Hilbert scheme $\Hilb_{X/Y}$ is an algebraic space locally of finite presentation over $Y$ by
\cite[\href{https://stacks.math.columbia.edu/tag/0D01}{Tag 0D01}]{stacks-project}.
Take a point $p\in Y$.
An analytic family
of ample divisors $D_U\subset X_U\to U$ 
over an analytic neighborhood 
$U \ni p$, may be approximated
by an algebraic family of ample divisors, 
$D'_{U'}\subset X_{U'}\to U'$ over
an \'etale neighborhood $U'\ni p$, which 
coincides with the restriction of $D_U$ to $p$.
Possibly replacing $U'$ with a smaller,
Zariski open neighborhood of $p\in U'$,
the divisor $D_{U'}'$ is relatively ample.
\end{proof}
 
We now consider the much subtler question of
when $X^{\rm univ}(\mathbbm{b})$ is {\it projective}
over $\widetilde{\cA}_g^{\,\bB}$, as opposed to 
\'etale-locally projective.

\begin{definition} Let 
$\overline{b}$, $\overline{b}\,\!'\in 
H^0(\bT^g,\tfrac{1}{d}\bZ{\rm PL}/\tfrac{1}{d}\bZ{\rm L})$. 
We say that $\overline{b}\sim \overline{b}\,\!'$ lie in the same
{\it shift class} if $\overline{b}-\overline{b}\,\!'$ 
lifts to an $\bfM$-periodic section 
$b-b'\colon \bfM_\bR\to \bR$ of $\tfrac{1}{d}\bZ{\rm PL}$. 
\end{definition}

Recall that ${\bf M}\simeq \bZ^g=H_1(\bT^g,\bZ)$
in Construction \ref{veronese}.
A necessary, but insufficient, condition
for $\overline{b}\sim \overline{b}\,\!'$ is that they
define the same monodromy bilinear form 
$B$ via formula (\ref{Bi}).

\begin{example}
    Let $g=1$ and consider 
    $ \overline{b}$, 
    $\overline{b}\,\!'$, 
    $\overline{b}\,\!''$ for which
    ${\rm Bend}(\overline{b})=2[\tfrac{0}{1}]$,
    ${\rm Bend}(\overline{b}\,\!')=[\tfrac{0}{2}]+[\tfrac{1}{2}]$, and
    ${\rm Bend}(\overline{b}\,\!'')=[\tfrac{1}{3}]+[\tfrac{2}{3}]$ as $\bZ$-weighted linear combinations
    of $\tfrac{1}{d}$-integral
    codimension $1$ polytopes,
    see Definition \ref{bendinglocus}. 
    All three define the same monodromy
    bilinear form $B=2x^2$. But we 
    have $\overline{b}\sim \overline{b}\,\!''$ and
    $\overline{b}\not\sim \overline{b}'$. See 
    Figure \ref{fig:shift_class}. The fundamental
    issue is that, while we could
    subtract from $b-b'$ a linear function of slope 
    $\tfrac{1}{2}$ to make it periodic, such a function
    is not a section of $\tfrac{1}{2}\bZ {\rm L}$ 
    on $\bfM_\bR\simeq \bR$, see Definition \ref{def:d-dicing}, because it does not have integral slope.

    \begin{figure}
        \centering
        \includegraphics[width=0.8\linewidth]{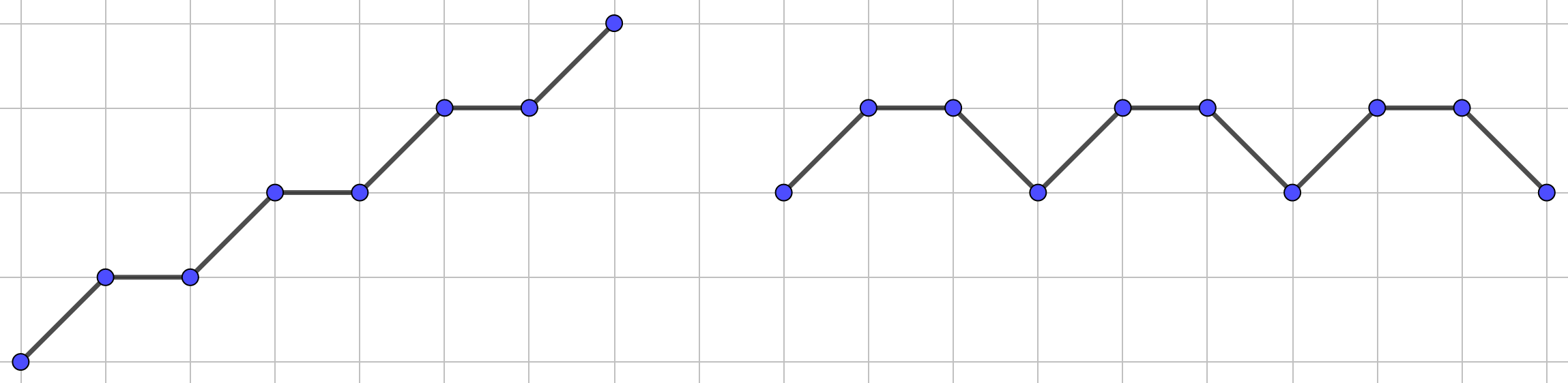}
        \caption{Left: lift of section
        $\overline{b}-\overline{b}\,\!'$ with bending 
        $[\tfrac{0}{2}]-[\tfrac{1}{2}]$. 
        Right: lift of section $\overline{b}-\overline{b}\,\!''$ 
        with bending 
        $2[\tfrac{0}{3}]-[\tfrac{1}{3}]-[\tfrac{2}{3}]$.}
        \label{fig:shift_class}
    \end{figure}
\end{example}

\begin{proposition}\label{bundle-algebraicity}
Let $f\colon X^{\rm univ}(\mathbbm{b})\to 
 \widetilde{\cA}_g^{\, \bB}$ and 
 $f'\colon X^{\rm univ}(\mathbbm{b}')\to 
 \widetilde{\cA}_g^{\, \bB}$ be two
 universal polytopal Mumford 
 Constructions \ref{mumford-polytope-2}, for 
 lifts of two cones $\mathbbm{b},\mathbbm{b}'\subset H^0(\bT^g, 
 \tfrac{1}{d}\bZ {\rm PL}/\tfrac{1}{d}\bZ{\rm L})$
 into $H^0(\bfM_\bR,  \tfrac{1}{d}\bZ {\rm PL})$, mapping
 to the same monodromy cone $\bB\subset {\rm Sym}^2\bfM^\vee$.
 Suppose, that for any $B \in \bB$,
 the PL lifts $b,b'\colon \bfM_\bR\to \bR$ satisfy
 the property that $b'-b$ is $\bfM$-periodic. 
 Then, the canonical analytic
 polarization on the Mumford construction
 extends as a relatively ample
global section of the relative Picard  
${\rm Pic}_{X^{\rm univ}(\mathbbm{b})/\widetilde{\cA}_g^{\,\bB}}$
 if and only if 
 the same holds for $X^{\rm univ}(\mathbbm{b}')/\widetilde{\cA}_g^{\,\bB}$.
\end{proposition}

We note that the data of the gluing of the Mumford construction
onto the universal family 
$\widetilde{\cX}_g\to \widetilde{\cA}_g$ requires
the data of lifts of $\mathbbm{b}$,
$\mathbbm{b}'$
into PL functions on $\bfM_\bR$
by Remark \ref{rem-shifts}.

\begin{proof}
Let $\cL$ and $\cL'$ be the defining
relatively ample line bundles of the two Mumford constructions
over the analytic tubular neighborhood $T(\bB)\subset Y(\bB)$ 
of the deepest toric stratum.

Take a $\bfM$-periodic, regular 
refinement $\cS$ of the normal fans
for $\mathbbm{b}$, $\mathbbm{b}'$, and let
$X^{\rm univ}_\circ(\cS)\to T(\bB)$ 
be the corresponding fan 
Construction \ref{singular-base}. Let $\cL_\cS$
and $\cL_\cS'$ be the pullbacks of $\cL$
and $\cL'$ to $X^{\rm univ}(\cS)$, and define 
$\cE:=\cL_\cS'\otimes \cL_\cS^{-1}$.
Finally, let $\widetilde{\cE}$, $\widetilde{\cL}_\cS$,
$\widetilde{\cL}\,\!'_S$
be the pullbacks of $\cE$, $\cL_\cS$, $\cL_{\cS}'$ 
to the universal cover of $X^{\rm univ}_\circ(\cS)$. Then
$\widetilde{\cE}$, $\widetilde{\cL}_\cS$,
$\widetilde{\cL}\,\!'_S$ are 
$\bfM$-equivariant line bundles.
The condition that
$b'-b$ is $\bfM$-periodic implies 
that we have an $\bfM$-equivariant
isomorphism $$\widetilde{\cE}\simeq 
\cO(\sum_{\substack{\textrm{rays} \\ 
\bR_{\geq 0}v_i\in\cS}} a_{v_i}D_{v_i}),$$
with $a_{v_i}\in \bZ$ depending only on the 
$\bfM$-equivalence
class $\overline{v}_i$ of the ray $\bR_{\geq 0}v_i\in \cS$.
Quotienting, we deduce
that the line bundle $\cE$ is represented by a
finite $\bZ$-linear sum 
$\sum a_{\overline{v}_i}D_{\overline{v}_i}$ of components
over the boundary of $X^{\rm univ}_\circ(\cS)$. 
As components
over the boundary, $D_{\overline{v}_i}$ descend
to algebraic divisors
on the algebraic space $X^{\rm univ}(\cS)\to \widetilde{\cA}_g^{\,\bB}$ and thus $\cL_\cS$
and $\cL'_{\cS}$ differ by twisting by a linear combination
of vertical divisors,
over the boundary of $\widetilde{\cA}_g^{\,\bB}$. 
So one extends as a section of 
relative
Picard if and only if the other does. Furthermore,
the relative ampleness of $\cL$ and $\cL'$ 
over the interior $\widetilde{\cA}_g$
are equivalent. On the other hand, 
the relative ampleness
of either over the boundary $\widetilde{\cA}_g^{\,\bB}\setminus 
\widetilde{\cA}_g$ is automatic, by construction.
\end{proof}

\begin{proposition}\label{projectivity2}
Let $f\colon X^{\rm univ}(\mathbbm{b})\to
\widetilde{\cA}_g^{\,\bB}$ 
be a universal
polytopal Mumford construction associated to
a lift of $\mathbbm{b}\subset H^0(\bT^g, 
\tfrac{1}{d}\bZ{\rm PL}/\tfrac{1}{d}\bZ{\rm L})$, 
extending (as an algebraic space)
the universal family 
$\widetilde{\cX}_g\to \widetilde{\cA}_g$
of abelian varieties. Then $f$
is a projective morphism whenever 
$b_i({\bf m})-b_i(-{\bf m})$
is $\bfM$-periodic. 
Similarly, there
is a relatively projective
extension $f\colon
X^{{\rm univ}\star}(\mathbbm{b})\to 
\widetilde{\cA}_g$ 
of the universal family $\widetilde{\cX}_g^\star
\to \widetilde{\cA}_g$ of abelian
torsors when 
$\overline{b}\sim \overline{b}\,\!^{\rm vor}$
lie in the
same shift class, for any $B\in \bB$, and for 
$\overline{b}\,\!^{\rm vor}$ defined as in
(\ref{characteristic}).
\end{proposition}

\begin{proof} To prove the extension result for $\widetilde{\cX}_g$,
we follow the proof strategy of Proposition 
\ref{bundle-algebraicity}: We may pass to a smooth
$(-1)$-symmetric common refinement $X^{\rm univ}_\circ(\cS)\to T(\bB)$
of the Mumford constructions for $b_i({\bf m})$ 
and $b_i(-{\bf m})$. The pulled back line bundles $\cL_S$, 
$\cL_\cS'$ associated to $b_i({\bf m})$, 
$b_i(-{\bf m})$ are interchanged by the $(-1)$-involution:
$(-1)^*\cL_\cS \simeq \cL_{\cS}'$.
If $b_i({\bf m})-b_i(-{\bf m})$
is $\bfM$-periodic, we may conclude that
$\cL_{\cS}'\simeq \cL_{\cS}(\,\sum a_{\overline{v}_i}D_{\overline{v}_i})$ differ by twisting
by vertical divisors. 
It follows that $(\cL_\cS)^{\otimes 2}$
defines an algebraic extension of 
$(\cL\,\!_{\widetilde{\cX}_g})^{\otimes d}$ 
where 
$\cL\,\!_{\widetilde{\cX}_g}\in {\rm Pic}_{\widetilde{\cX}_g/\widetilde{\cA}_g}(\widetilde{\cA}_g)$
is a tensor square of a $(-1)$-symmetric
local lift of the principal polarization. 

The case of extending $\widetilde{\cX}_g^\star$ is similar,
but again Proposition \ref{bundle-algebraicity} does not directly
apply, since we are gluing onto the universal abelian torsor.
The extension of $\widetilde{\cX}_g^\star$ by the Mumford
construction of second Voronoi type (Construction \ref{torsor})
is relatively projective. Replacing $\bB$ with 
a common refinement of $\bB$ and $\fF_{\rm vor}$, 
we may assume that $\bB$
is contained in a second Voronoi cone.
Choosing lifts $b^{\rm vor}$ as in 
(\ref{characteristic}), defines a local
analytic section of $\overline{\cX_g^\star}^{\rm vor}$
near the boundary stratum associated
to a second Voronoi cone. The hypothesis
that $\overline{b}\sim \overline{b}\,\!^{\rm vor}$,
for all $B\in \bB$ and argument of
Proposition \ref{bundle-algebraicity}
show that, with respect to the chosen local
origin section of $\widetilde{\cX}_g^\star$,
and a well chosen lift of $\mathbbm{b}$,
there is a gluing $X^{{\rm univ}\star}(\mathbbm{b})\to \widetilde{\cA}_g^{\,\bB}$
of $X^{{\rm univ}}(\mathbbm{b})\to T(\bB)$ and
$\widetilde{\cX}_g^\star\to \widetilde{\cA}_g$
for which the ample line bundle on 
the former extends to a section in
${\rm Pic}_{X^{{\rm univ}\star}(\mathbbm{b})/
\widetilde{\cA}_g^{\,\bB}}(\widetilde{\cA}_g^{\,\bB})$.

 Thus, in either case $\widetilde{\cX}_g$ or $\widetilde{\cX}_g^\star$, 
 the canonical polarization on the Mumford 
 construction extends to a relatively ample section
 of relative Picard over
$\widetilde{\cA}_g^{\,\bB}$. 
 
 After passing
 to some further tensor power, we may lift to 
 an element $\cE$ of the (algebraic) Picard group of
 $X^{\rm univ}(\mathbbm{b})$ or 
 $X^{{\rm univ}\star}(\mathbbm{b})$,
 which is relatively very ample. 
 Pushing forward, we get a vector bundle $f_*\cE$ over
 $\widetilde{\cA}_g^{\,\bB}$ over the \'etale site, and 
 therefore over the Zariski site. It follows that the 
 projectivization of $(f_*\cE)^\vee$ is relatively 
 projective over $\widetilde{\cA}_g^{\,\bB}$. Thus,
 $X^{\rm univ}(\mathbbm{b})$ or 
 $X^{{\rm univ}\star}(\mathbbm{b})$ admit closed,
 algebraic embeddings into a projective space over
 $\widetilde{\cA}_g^{\,\bB}$. The result follows.
\end{proof}

Projectivity criteria for extensions of 
$\widetilde{\cX}_g$ should be compared 
to \cite[Ch.~VI]{fc}. Following a standard
toric construction, the polytope 
$\Gamma$ of Section 
\ref{sec:polytope} defines a convex PL function
$\tilde{\phi}\colon \cS\to \bR$ on the normal
fan $\cS$ of Section \ref{sec:fan}. The function
$\tilde{\phi}$ should be 
an ``admissible homogeneous
principal polarization function'' as in
\cite[Ch.~VI, Def.~1.5]{fc}, with our
conditions on the shift class of $b_i$ 
related to Def.~1.5.(vi) of {\it loc.cit.}
Our projectivity results should then
follow from \cite[Ch.~VI, Thm.~1.13]{fc}, though 
translating between the language
used here and that in 
{\it loc.cit.}~is somewhat involved.

\begin{corollary}
Suppose that $\mathscr{H}$ is a 
hyperplane arrangement for the regular matroid $\uR$,
for which the parallel hyperplanes 
normal to $\vec{x}_i$ are 
$H_i^{(j)} \coloneqq 
\vec{x}_i({\bf m})\in \epsilon_{ij}+\bZ$.
Then the extension 
$$X^{{\rm univ}\star}(\uR, \mathscr{H})
\to \widetilde{\cA}_g^{\,\bB_{\uR}}$$ 
of $\widetilde{\cX}_g^\star\to \widetilde{\cA}_g$
is projective whenever
$\sum_j \epsilon_{ij}=0\in \bQ/\bZ$ for all $i=1,\dots,k$
and the extension 
$$X^{\rm univ}(\uR, \mathscr{H})\to 
\widetilde{\cA}_g^{\,\bB_{\uR}}$$ of 
$\widetilde{\cX}_g\to \widetilde{\cA}_g$ 
is projective whenever 
$\sum_j (\epsilon_{ij}+\tfrac{1}{2})=0\in \bQ/\bZ$
for all $i=1,\dots,k$.
\end{corollary}

\begin{proof} 
Let $r_i$ be the number of hyperplanes
$H_i^{(j)}\in \mathscr{H}$ normal to $\vec{x}_i$.

For $\bB_\uR\in \fF_{\rm vor}$ a matroidal
cone, the Delaunay 
decomposition as in Construction
\ref{voronoi-mumford} is the unshifted
hyperplane arrangement (\ref{Hi}). 
Thus, it follows from 
Proposition \ref{projectivity2}
that $\mathscr{H}$ defines a projective
extension of 
$\widetilde{\cX}_g^\star\to \widetilde{\cA}_g$ 
whenever the section 
$\overline{b}_i\in H^0(\bT^g, 
\tfrac{1}{d}\bZ{\rm PL}/\tfrac{1}{d}\bZ{\rm L})$ 
bending along 
$H_i^{(j)}:=\{\vec{x}_i({\bf m}) \in \epsilon_{ij}+\bZ\}$ lies
in the same shift class
as $\overline{b}_i^o\in 
H^0(\bT^g, \tfrac{1}{d}\bZ{\rm PL}/
\tfrac{1}{d}\bZ{\rm L})$
which bends $r_i$ times along the unshifted
hyperplane $\vec{x}_i({\bf m}) \in 0+\bZ$.
Equivalently, 
\begin{align}\label{eq-un}
\textstyle \int_0^1 \int_{0^-}^x (-r_i\delta_0+ \sum_{j=1}^{r_i} \delta_{\epsilon_{ij}}) \,dy \, dx\in \bZ\end{align} where $\delta_p$ denotes the Dirac delta
function at $p$. Now (\ref{eq-un}) holds if and only if $\sum_{j=1}^{r_i} \epsilon_{ij}\in \bZ$.

For the case of extending 
$\widetilde{\cX}_g\to \widetilde{\cA}_g$ 
projectively, we observe that an arrangement
which bends $r_i$ times along the half-shifted hyperplane
$\vec{x}_i({\bf m})\in \tfrac{1}{2}+\bZ$ satisfies
the hypotheses of Proposition \ref{projectivity2}
by lifting to a $\tfrac{1}{2}\bZ{\rm PL}$ 
function $b_i'$ which is identically
zero in a neighborhood of the origin of $\bfM_\bR$.
So $b_i'$ defines a relatively projective extension
of the universal family 
$\widetilde{\cX}_g\to \widetilde{\cA}_g$. Then
by Proposition \ref{bundle-algebraicity}, it suffices
to understand when $\overline{b}_i'\sim \overline{b}_i$
lie in the same shift class, i.e.
\begin{align*}
\textstyle \int_0^1 \int_{0^-}^x (-r_i\delta_{\frac{1}{2}}+ \sum_{j=1}^{r_i} \delta_{\epsilon_{ij}}) \,dy \, dx\in \bZ.\end{align*}
This holds exactly when $\sum_{j=1}^{r_i} (\epsilon_{ij}+\tfrac{1}{2})\in \bZ$. 
\end{proof}

\section{Proof of 
Theorem \ref{regular-extension-thm}}
\label{proof}

Our goal is to prove Theorem 
\ref{regular-extension-thm}, and leverage 
Proposition \ref{algebraicity} 
to prove more algebraic formulations of the 
results therein, see Theorem \ref{theorem:extension} and 
Corollary \ref{corollary:degeneration} below.

\begin{proof}[Proof of 
Theorem \ref{regular-extension-thm}]
Let $f^\ast\colon X^\ast \to (\Delta^*)^k$ be a family
of PPAVs which is matroidal with respect to the snc extension $(\Delta^*)^k\hookrightarrow \Delta^k$ 
(Def.~\ref{mat-degen}). 
Then there are integers 
$r_i>0$ for which the monodromy bilinear
forms about $\{u_i=0\}$ are $r_iB_i$ where 
$B_i={\bf x}_i^2$ is an integral generator of 
the matroidal cone $\bB_\uR$ of the 
corresponding regular
matroid $\uR$. Constructions 
\ref{shifted-matroidal-construction} and 
\ref{mumford-polytope-2},  see also 
Notation \ref{notation:XunivRH}, give
a universal Mumford degeneration
$$f^{\rm univ}\colon 
X^{\rm univ}(\uR, \mathscr{H})\to 
\widetilde{\cA}_g^{\, \bB_\uR}$$ 
on a transversely shifted hyperplane arrangement
for the associated regular matroid $\uR$,
which has exactly $r_i$ bending loci in $\bT^g$, with
bending parameter $1$, along hyperplanes normal
to ${\bf x}_i$.  

The monodromies about the boundary divisors
of $\widetilde{\cA}_g^{\, \bB_\uR}$ 
are exactly $r_iB_i$
and by Proposition \ref{monodromy-extend}, 
the classifying morphism 
$(\Delta^*)^k\to \cA_g$ will
(lift and) extend to
$\Delta^k\to 
\widetilde{\cA}_g^{\, \bB_\uR}$.
Pulling
back $X^{\rm univ}(\uR, \mathscr{H})$ along the extension
of the classifying morphism to $\Delta^k$ produces the desired extension $f\colon X\to \Delta^k$. It has smooth
total space and nodal singularities,
because $f^{\rm univ}$ is locally
trivial along the deepest toroidal stratum of
$\widetilde{\cA}_g^{\, \bB_\uR}$ 
and so smoothness, resp.~nodality,
of $X^{\rm univ}$, resp.~$f^{\rm univ}$ 
(see Proposition \ref{X-smooth}), 
implies smoothness, resp.~nodality, of the 
restrictions $X$, resp.~$f$, 
to the transversal slice $\Delta^k$
 to this deepest toroidal stratum.  

The condition that $f\colon X\to \Delta^k$ be
strictly nodal follows from the condition $r_i\geq 2$ 
by Proposition \ref{strict-D}.

Finally, we address the $K$-triviality of $X$. It
suffices, by Proposition \ref{k-triv}, to show that
the slice $\cS_{(1,\dots,1)}$ of the normal fan
has integral vertices. Indeed, this holds for
any shifted matroidal degeneration, as each
top-dimensional
cell in $\cS_{(1,\dots,1)}$ is a Minkowski
sum of segments $\sum_{i\in I_v}[0,{\bf x}_i]\subset \bfN_\bR$ 
corresponding to hyperplanes $H_i\in \mathscr{H}$ 
meeting at the dual
vertex $v\in \bigcap_{i\in I_v} \oH_i$ 
of the arrangement in $\bfM_\bR/\bfM$.
(For a transversely shifted arrangement, these Minkowski
sums are integral-affine unit cubes).
We deduce the first part of 
Theorem \ref{regular-extension-thm}.

The second
part of Theorem \ref{regular-extension-thm}
follows from the existence of algebraic, 
transversely shifted matroidal
degenerations with specified monodromies, see 
Corollary \ref{corollary:degeneration} below.
\end{proof}

\begin{theorem}
\label{theorem:extension}
Let $f^\ast \colon X^\ast\to Y^\ast$ be a
projective family
of PPAVs over a base $Y^*=Y\setminus D$ for $D\subset Y$
an snc divisor in a smooth quasiprojective variety $Y$. Let $0\in D$ and assume the local monodromy bilinear
forms $B_i$ 
about the components $D_i\ni 0$
are $r_i{\bf x}_i^2$ for an 
integral realization $i\mapsto {\bf x}_i\in \bfM^\vee$ of a regular matroid $\uR$,
where $\bfM\simeq {\rm gr}^W_0H_1(X_t,\bZ)$ 
for $t$ near $0$.
Up to passing to
an \'etale neighborhood of $0\in Y$,
there is a flat, projective, $D$-nodal,
relatively $K$-trivial  extension $f\colon X\to Y$,
which is furthermore strictly $D$-nodal when all
$r_i\geq 2$.
\end{theorem}
\begin{proof}
We have a classifying map $Y^\ast\to \cA_g$ and by
the hypothesis on monodromy, we have,
\'etale-locally about $0\in Y$, an extension
and lift $Y \to \widetilde{\cA}_g^{\, \bB_\uR}$,
$\bB_\uR\coloneqq \bR_{\geq 0}\{B_1,\dots,B_k\}$,
e.g.~because this lift exists analytically-locally about
$0\in Y$ (cf.~Prop.~\ref{monodromy-extend}). 
The morphism $Y \to \widetilde{\cA}_g^{\, \bB_\uR}$ 
is algebraic, for instance by Borel algebraicity. 
Pulling back the family $X^{\rm univ}(\uR, \mathscr{H})\to \widetilde{\cA}_g^{\,\bB_\uR}$, which is 
algebraic and \'etale-locally projective
by Proposition \ref{algebraicity}, we deduce
the result 
(the nodality and relative $K$-triviality
of the extension follow as in the proof above).
\end{proof}

\begin{corollary} \label{corollary:degeneration}
Let $\underline{R}$ be a regular
matroid of rank $g$ on a $k$ element set
and let $(r_1,\dots,r_k)\in \bN^k$. 
There exists a projective 
morphism $f\colon X\to Y$ of smooth 
quasiprojective varieties,
$k=\dim Y$, $g+k=\dim X$,
an snc divisor $D\subset Y$,  
and a zero-dimensional
stratum $0\in D$, 
satisfying the following
conditions:
\begin{enumerate}
\item 
The monodromies 
about the components $D_i\ni 0$ of $D$ are 
of the form $B_i=r_i{\bf x}_i^2$ and 
generate the matroidal cone $\bB_\uR$ 
(Defs.~\ref{def:Bi}, \ref{matroidal-cone}).
    \item The morphism $f \colon X \to Y$ is a transversely 
shifted matroidal 
degeneration on the matroid $\underline{R}$ 
near $0\in Y$ (Def.~\ref{def:shifted-matroidal-degeneration}), 
and the restriction of $f$ to $Y^*\coloneqq 
Y \setminus D$ is a family of principally polarized abelian varieties of dimension $g$.
\item 
The morphism $f$ is, up to shrinking $Y$, a $D$-nodal 
morphism, which is furthermore strictly $D$-nodal if
$r_i\geq 2$ for all $i=1,\dots,k$ (Def.~\ref{D-nodal}). 
\end{enumerate}
\end{corollary}

\begin{proof} By Proposition \ref{puncturedconstruction},
there is a smooth quasiprojective variety $Y$,
snc divisor $D\subset Y$, zero-stratum
$0\in D$, and family $f^*\colon X^*\to Y^*$ 
of PPAVs over $Y^*=Y\setminus D$,
whose monodromies about the components $D_i\ni 0$
are given by $r_i{\bf x}_i^2$. 
Applying Theorem \ref{theorem:extension} and passing
to an \'etale chart about $0$, we produce a
projective  extension $f\colon X\to Y$ with
the desired properties.\end{proof}

\bibliographystyle{mrl}
\bibliography{references}

\end{document}